\documentclass[11pt,a4paper]{article}

\usepackage{a4wide}
\usepackage{amsmath,amsthm,amssymb,amsfonts}
\usepackage[hidelinks]{hyperref}
\usepackage[capitalise]{cleveref}
\usepackage{graphicx}
\usepackage{enumerate}
\usepackage{enumitem}
\usepackage{bm}
\usepackage{color}
\usepackage[normalem]{ulem}
\usepackage{float}

\usepackage{tikz}

\numberwithin{equation}{section}
\newtheorem{theorem}{Theorem}[section]
\newtheorem{definition}[theorem]{Definition}
\newtheorem{lemma}[theorem]{Lemma}

\newtheorem{corollary}[theorem]{Corollary}

\theoremstyle{definition}
\newtheorem{remark}[theorem]{Remark}
\newtheorem{example}[theorem]{Example}

\newcommand{\unboundednessComment}[1]{}

\newcommand{\mc}[1]{\mathcal{#1}}
\newcommand{\tn}[1]{\textnormal{#1}}

\newcommand{\avint}{-\hspace{-2.5ex}\int}

\renewcommand{\c}{c} 
\newcommand{\cRad}{\ell} 

\newcommand{\Floc}{F}
\newcommand{\Fint}{\mc{F}}

\newcommand{\FlocKL}{F_{\tn{KL}}}
\newcommand{\cHK}{c_{\tn{HK}}}

\renewcommand{\d}{\tn{d}}
\newcommand{\C}{C}
\newcommand{\resid}{R}

\newcommand{\measp}{{\mathcal M_+}}
\newcommand{\cont}{\mathcal{C}}
\newcommand{\prob}{{\mathcal P}}
\newcommand{\RadNik}[2]{\tfrac{\d #1}{\d #2}}
\newcommand{\pushforward}[2]{{#1}_{\#} #2}
\newcommand{\Lebesgue}{\mc{L}}
\newcommand{\restr}{{\mbox{\LARGE$\llcorner$}}}
\newcommand{\hd}{{\mathcal H}}
\newcommand{\dist}{{\mathrm{dist}}}
\newcommand{\diam}{{\mathrm{diam}}}
\newcommand{\dom}{{\mathrm{dom}}}
\DeclareMathOperator{\spt}{spt}

\newcommand{\E}{\mc{E}} 
\newcommand{\HK}{\tn{HK}}
\newcommand{\WOT}{W_{\tn{OT}}}
\newcommand{\W}{W}

\DeclareMathOperator{\Hell}{Hell}

\newcommand{\G}{\mc{G}} 

\newcommand{\R}{\mathbb{R}}
\newcommand{\N}{\mathbb{N}}

\DeclareMathOperator{\Id}{Id}

\newcommand{\Hex}{H}
\newcommand{\EDens}{B}

\graphicspath{{figures/}}
\begin{document}
\author{D.~P.~Bourne\footnote{Heriot-Watt University, Edinburgh, UK; D.Bourne@hw.ac.uk}, $\;$ B.~Schmitzer\footnote{University of G\"ottingen, Germany; Schmitzer@cs.uni-goettingen.de}, $\;$ B.~Wirth\footnote{University of M\"unster, Germany; Benedikt.Wirth@uni-muenster.de}}
\title{Semi-discrete unbalanced optimal transport and quantization}

\maketitle

\begin{abstract}
In this paper we study the class of optimal entropy-transport problems introduced by Liero, Mielke and Savar\'e in Inventiones Mathematicae 211 in 2018.
This class of unbalanced transport metrics allows for transport between measures of different total mass, unlike classical optimal transport where both measures must have the same total mass. In particular, we develop the theory for the important subclass of semi-discrete unbalanced transport problems, where one of the measures is diffuse (absolutely continuous with respect to the Lebesgue measure) and the other is discrete (a sum of Dirac masses). We characterize the optimal solutions and show they can be written in terms of generalized Laguerre diagrams. We use this to develop an efficient method for solving the semi-discrete unbalanced transport problem numerically. As an application we study the unbalanced quantization problem, where one looks for the best approximation of a diffuse measure by a discrete measure with respect to an unbalanced transport metric.
We prove a type of crystallization result in two dimensions -- optimality of a locally triangular lattice with spatially varying density -- and compute the asymptotic quantization error as the number of Dirac masses tends to infinity.
\end{abstract}

\section{Introduction}
\label{sec:Motivation}
In this paper we study \emph{semi-discrete unbalanced} optimal transport problems: What is the optimal way of transporting a diffuse measure to a discrete measure (hence the name \emph{semi-discrete}), where the two measures may have different total mass (hence the name \emph{unbalanced})? As an application we study the \emph{unbalanced quantization problem}: What is the best approximation of a diffuse measure by a discrete measure with respect to an unbalanced transport metric?

\subsection{Unbalanced optimal transport}
Classical optimal transport theory asks for the most efficient way to rearrange mass
between two given probability distributions. Its origin goes back to 1781 and the French engineer Gaspard Monge, who was interested in the question of how to transport and reshape a pile of earth to form an embankment at minimal effort. It took over 200 years to develop a complete mathematical understanding of this problem, even to answer the question of whether there exists an optimal way of redistributing mass. Since the  mathematical breakthroughs of the 1980s and 1990s, the field of optimal transport theory has thrived and found applications in
crowd and traffic dynamics, economics, geometry, image and signal processing, machine learning and data science, PDEs, and statistics.
Depending on the context, mass may represent the distribution of particles (people or cars), supply and demand, population densities, etc.
For thorough introductions see, e.g., \cite{Galichon2016,PeyreCuturi2018,Santambrogio-OTAM,Villani-OptimalTransport-09}.

In classical optimal transport theory the initial and target measures must have the same total mass. In applications this is not always natural.
Changes in mass may occur due to creation or annihilation of particles or a mismatch between supply and demand.
Therefore so-called \emph{unbalanced} transport problems, accounting for such differences, have recently received increased attention \cite{Figalli2010,KMV-OTFisherRao-2015,ChizatOTFR2015,LieroMielkeSavare-HellingerKantorovich-2015a}. Brief overviews and discussions of various formulations can be found, for instance, in \cite{ChizatDynamicStatic2018,SchmitzerWirth-DynamicW1-2017}.
Further theoretical properties are examined in \cite{LaMi2017,LMS-GeodesicConvexity}, examples for applications in data analysis can be found in \cite{Lavenant2021,LinHK2021,UnbalancedGWBrain2022}.
In this article we study the class of unbalanced transport problems called \emph{optimal entropy-transport problems} from \cite{LieroMielkeSavare-HellingerKantorovich-2015a}; see \cref{def:UnbalancedTransport}. In particular, we develop this theory for the special case of semi-discrete transport.

\subsection{Semi-discrete transport}
Semi-discrete optimal transport theory is about the best way to transport a diffuse measure, $\mu \in L^1(\Omega)$, $\Omega \subset \R^d$, to a discrete measure, $\nu = \sum_{i=1}^M m_i \delta_{x_i}$. These type of problems arise naturally, for instance, in economics in computing the distance between a population with density $\mu$ and a resource with distribution $\nu = \sum_{i=1}^M m_i \delta_{x_i}$, where $x_i \in \Omega$ represent the locations of the resource and $m_i > 0$ represent the size or capacity of the resource. The classical semi-discrete optimal transport problem, where $\mu$ and $\nu$ are probability measures, has a nice geometric characterization. For example, for $p \in [1,\infty)$, the Wasserstein-$p$ metric $W_p$ is defined by
\[
W_p(\mu,\nu) = \min \left\{ \sum_{i=1}^M \int_{T^{-1}(x_i)} |x-x_i|^p \mu(x) \, \d x \, \right. \left| \, T:\Omega \to \{ x_i \}_{i=1}^M, \, \int_{T^{-1}(x_i)} \mu(x) \, \d x = m_i \right\}^{1/p}
\]
where $\sum_{i=1}^M m_i = \int_{\Omega} \mu(x) \, \d x = 1$.
This is an optimal partitioning (or assignment) problem, where the domain $\Omega$ is partitioned into the regions $T^{-1}(x_i)$ of mass $m_i$, $i \in \{1,\ldots,M\}$, and each point $x \in  T^{-1}(x_i)$ is assigned to point $x_i$. For example, in two dimensions, $\Omega$ could represent a city, $\mu$ the population density of children,  $x_i$ and $m_i$ the location and size of schools, $T^{-1}(x_i)$ the catchment areas of the schools, and $W_p(\mu,\nu)$ the cost of transporting the children to their assigned schools. If $p=2$, it turns out that the optimal partition $\{T^{-1}(x_i)\}_{i=1}^M$ is a \emph{Laguerre diagram} or \emph{power diagram}, which is a type of weighted Voronoi diagram: There exist weights $w_1,\ldots,w_M \in \R$ such that
\[
\overline{T^{-1}(x_i)} = \{ x \in \Omega \,|\, |x-x_i|^2 - w_i \leq |x-x_j|^2 - w_j \, \forall \, j \in \{1,\ldots,M\} \}.
\]
The transport cells $T^{-1}(x_i)$ are the intersection of convex polytopes (polygons if $d=2$, polyhedra if $d=3$) with $\Omega$. The weights $w_1,\ldots,w_M \in \R$ can be found by solving an unconstrained concave maximization problem. If $p=1$, the optimal partition $\{T^{-1}(x_i)\}_{i=1}^M$ in an Apollonius diagram. See, e.g., \cite[Sec.~6.4]{AurenhammerKleinLee}, \cite[Chap.~5]{Galichon2016}, \cite{KitagawaMerigotThibert,MerigotThibertOT}, \cite[Chap.~5]{PeyreCuturi2018}, and \cref{sec:BackgroundSemiDisc} below, where we summarize the main results from classical semi-discrete optimal transport theory.
Applications of semi-discrete transport include
fluid mechanics \cite{GallouetMerigot18,Gallouet:2022}, microstructure modelling \cite{BKRS20,BFRSB24}, optics \cite{MeyronMerigotThibert2018}, 
 and the Lagrangian discretization of Wasserstein gradient flows \cite{LeclercFlows2020} and mean field games \cite{SarrazinMFG2020}.

In \cref{sec:semiDiscreteUnbalanced} we extend these results to unbalanced transport, where $\mu$ and $\nu$ no longer need to have the same total mass, and the Wasserstein-$p$ metric is replaced by the unbalanced transport metric $W$ from \cref{def:UnbalancedTransport}. We prove that, also in the unbalanced case, the optimal partition is a type of generalized Laguerre diagram and it can be found by solving a concave maximization problem for a set of weights $w_1,\ldots,w_M$; see \cref{thm:UnbalancedTessellationDual,thm:UnbalancedOptimalityConditions}. This problem is natural from a modelling perspective, for example to describe a mismatch between the demand of a population $\mu$ and the supply of a resource $\nu$, and to model the prioritization of high-density regions at the expense of areas with a low population density.

For unbalanced transport, there is no one, definitive transport cost, but many models are conceivable. As a first application of our theory of semi-discrete unbalanced transport, in \cref{exm:modelComparison,ex:LengthScales}, we use it to compare different unbalanced transport models.
As a second application, in \cref{sec:quantization}, we apply it to the quantization problem.

\subsection{Quantization}
Quantization of measures refers to the problem of finding the best approximation of a diffuse measure by a discrete measure \cite{GrafLuschgy2000}, \cite[Sec.~33]{Gruber07}. For example, the classical quantization problem with respect to the Wasserstein-$p$ metric, $p \in [1,\infty)$, is the following: Given $\mu \in L^1(\Omega)$, $\Omega \subset \R^d$, $\int_\Omega \mu(x) \, \d x = 1$, find a discrete probability measure $\nu = \sum_{i=1}^M m_i \delta_{x_i}$ that gives the best approximation of $\mu$ in the Wasserstein-$p$ metric,
\begin{equation}
\label{eq:ClassicalQuantization}
Q^M_p(\mu) = \min \left\{ W_p^p(\mu,\nu) \, \Bigg| \, \nu = \sum_{i=1}^M m_i \delta_{x_i}, \; x_1,\ldots,x_M \in \Omega, \; m_i >0, \; \sum_{i=1}^M m_i=1 \right\}.
\end{equation}
We call $Q^M_p$ the \emph{quantization error}.
Problems of this form arise in a wide range of applications including economic planning and optimal location problems \cite{BollobasStern72,BouchitteJimenezMahadevan2011,ButtazzoSantambrogio2009}, finance \cite{PagesPhamPrintems2004}, numerical integration \cite[Sec.~2.2]{DuFaberGunzburger99}, \cite[Sec.~2.3]{PagesPhamPrintems2004}, energy-driven pattern formation \cite{BournePeletierRoper,LarssonChoksiNave2016}, and approximation of initial data for particle (meshfree) methods for PDEs. 
 An approach to quantization using gradient flows is given in
\cite{CagliotiGolseIacobelli2015,Iacobelli2018}.
We mention a few important variations on the classical quantization problem. The case where the masses $m_1,\ldots,m_M$ are fixed and the minimisation in \eqref{eq:ClassicalQuantization} is only taken over $x_1,\ldots,x_M$ is considered for example in  
\cite{BKRS20,MerigotSantambrogioSarrazin,LevyCentroidalPower}. 
The case where $\mu$ is a discrete measure, with support of cardinality $N \gg M$, has applications in image and signal compression \cite{DuGunzburgerJuWang2006,GershoGray1992} and data clustering ($k$-means clustering) \cite{MacQueen67,ThorpeTheilJohansenCade15}.
If $\nu$ is a one-dimensional measure (supported on a set of Hausdorff dimension $1$), then the quantization problem is known as the \emph{irrigation problem}
\cite{LuSlepcev2013,MosconiTilli2005}. In this paper we consider the variation where the Wasserstein-$p$ metric in \eqref{eq:ClassicalQuantization} is replaced by an unbalanced transport metric.

It can be shown that the quantization problem \eqref{eq:ClassicalQuantization} can be rewritten as an optimization problem in terms of the particle locations $\{ x_i \}_{i=1}^M$ and their Voronoi tessellation:
\begin{equation}
\label{eq:ClassicalQuantizationTessellation}
Q^M_p(\mu) = \min \left\{ J(x_1,\ldots,x_M) \, | \,  x_1,\ldots,x_M \in \Omega \right\}
\end{equation}
where
\[
J(x_1,\ldots,x_M) = \sum_{i=1}^M \int_{V_i(x_1,\ldots,x_M)} |x-x_i|^p \mu(x) \, \d x
\]
and where $\{ V_i \}_{i=1}^M$ is the Voronoi diagram generated by $\{x_i\}_{i=1}^M$,
\[
V_i = V_i(x_1,\ldots,x_M) = \left\{ x \in \Omega\,\middle|\, |x-x_i| \leq |x-x_j| \tn{ for all } j\in\{1,\ldots,M\} \right\}.
\]
If $(x_1, \ldots , x_M )$ is a global minimizer of $J$, then $\sum_{i=1}^M \left(  \int_{V_i} \mu \, \d x \right) \delta_{x_i}$ is an optimal quantizer of $\mu$ with respect to the Wasserstein-$p$ metric.
See for instance \cite[Sec.~4.1]{BournePeletierRoper}, \cite[Sec.~7]{Kloeckner2012} and \cref{thm:quantizationTessellation}.
In the vector quantization (electrical engineering) literature $J$ is known as the \emph{distortion} of the quantizer \cite{GershoGray1992}.

The quantization problem with respect to the Wasserstein-2 metric is particularly well studied. In this case it can be shown that critical points of $J$ are generators of \emph{centroidal Voronoi tessellations} (CVTs) of $M$ points \cite{DuFaberGunzburger99}; this means that $\nabla J(x_1,\ldots,x_M)=0$ if and only if $x_i$ is the centre of mass of its own Voronoi cell $V_i$ for all $i$,
\begin{equation}
\label{eq:CVTs}
x_i = \frac{\displaystyle \int_{V_i(x_1,\ldots,x_M)} x \mu(x) \, \d x}{\displaystyle \int_{V_i(x_1,\ldots,x_M)} \mu(x) \, \d x}, \quad i \in \{1,\ldots,M\}.
\end{equation}
In general there does not exist a unique CVT of $M$ points, as illustrated in \cref{fig:CVTs}, and $J$ is non-convex with many local minimisers for large $M$.
Equation \eqref{eq:CVTs} is a nonlinear system of equations for $x_1,\ldots,x_M$.
A simple and popular method for computing CVTs is Lloyd's algorithm \cite{DuFaberGunzburger99,EmelianenkoJuRand08,Lloyd82,SabinGray86}, which is a fixed point method for solving the Euler--Lagrange equations \eqref{eq:CVTs}.

\begin{figure}[hbt]
	\centering
	\includegraphics[width=0.4\textwidth,trim=35 15 0 0,clip]{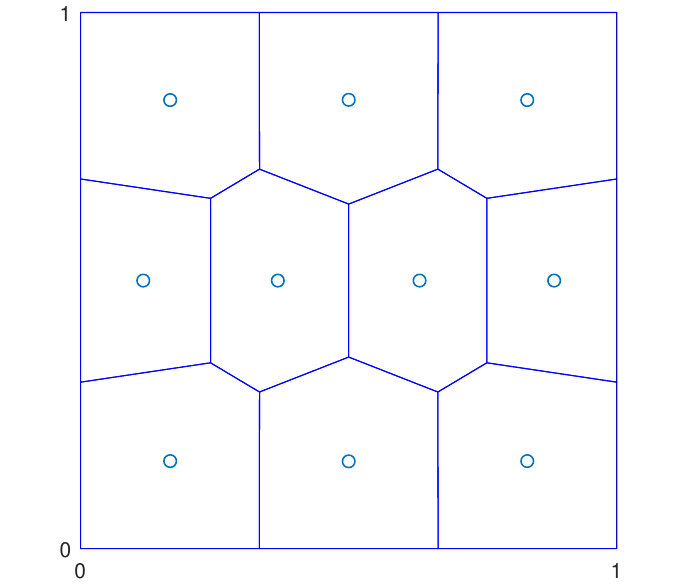}
  \includegraphics[width=0.4\textwidth,trim=35 15 0 0,clip]{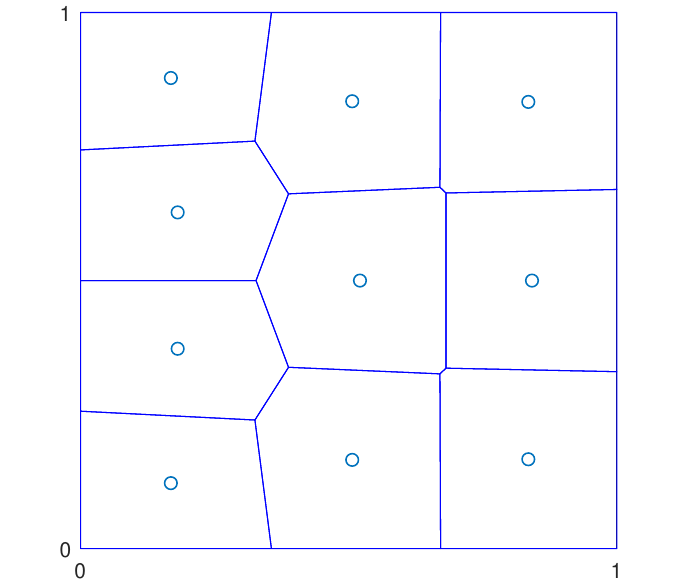}
	\caption{Two (approximate) centroidal Voronoi tessellations (CVTs) of 10 points for the uniform density $\mu=1$ on a unit square. The polygons are the centroidal Voronoi cells $V_i$ and the circles are the generators $x_i$. The CVTs were computed using Lloyd's algorithm. The CVT on the left has a lower energy $J$ than the CVT on the right. The corresponding quantizer $\nu=\sum_{i=1}^{10} m_i \delta_{x_i}$ of $\mu$ is reconstructed from the CVT by taking $m_i$ as the areas of the centroidal Voronoi cells and $x_i$ as their generators.}
	\label{fig:CVTs}
\end{figure}

In \cref{sec:Quantization,sec:QuantizationNumerics} we extend these results to unbalanced quantization, where the Wasserstein-$p$ metric in \eqref{eq:ClassicalQuantization} is replaced by the unbalanced transport metric $W$ (defined in equation \eqref{eq:UnbalancedProblem}) and where $\mu$ and $\nu$ need not have the same total mass.  In \cref{thm:quantizationTessellation} we prove an expression of the form $\eqref{eq:ClassicalQuantizationTessellation}$, which states that the unbalanced quantization problem can be reduced to an optimization problem for the locations $x_1,\ldots,x_M$ of the Dirac masses. This optimization problem is again formulated in terms of the Voronoi diagram generated by $x_1,\ldots,x_M$. In \cref{sec:QuantizationNumerics} we solve the unbalanced quantization problem numerically, which includes extending Lloyd's algorithm to the unbalanced case.

We conclude the paper in \cref{sec:Crystallization} by studying the asymptotic unbalanced quantization problem: What is the optimal configuration of the particles $x_1,\ldots,x_M$ as $M \to \infty$, and how does the quantization error scale in $M$?
Consider for example the classical quantization problem \eqref{eq:ClassicalQuantization} with $p=2$, $|\Omega|=1$, $\mu=1$ (i.e., $\mu$ is the Lebesgue measure on $\Omega$), and $M$ fixed. From above, we know that an optimal quantizer $\nu$ corresponds to an optimal CVT of $M$ points, where \emph{optimal} means that the CVT has lowest energy $J$ amongst all CVTs of $M$ points. Gersho \cite{Gersho79} conjectured that, as $M \to \infty$, the Voronoi cells of the optimal CVT asymptotically have the same shape, i.e., asymptotically they are translations and rescalings of a single polytope. In two dimensions ($d=2$) various versions of Gersho's Conjecture have been proved independently by several authors
\cite{BollobasStern72,Gruber99,MorganBolton02,Newman1982,GFejesToth01,FejesToth72}.
Roughly speaking, it has been shown that
the hexagonal tiling is optimal as $M \to \infty$. In other words, arranging the seeds $x_1,\ldots,x_M$ in a regular triangular lattice is asymptotically optimal.
This crystallization result can be stated more precisely as follows: If $\Omega$ is a convex polygon with at most 6 sides, then
\begin{equation}
\label{eq:classical_crystallizationLB}
J(x_1,\ldots,x_M) \ge \frac{5 \sqrt{3}}{54} \frac{1}{M}
\end{equation}
where the right-hand side is the energy of a regular triangular lattice of $M$ points such that the Voronoi cells $V_i$ are regular hexagons of area $1/M$.
In general this lower bound is not attained for finite $M$ (unless $\Omega$ is a regular hexagon and $M=1$), but it is attained in limit $M \to \infty$: 
\begin{equation}
\label{eq:classical_crystallization_limit}
\lim_{M \to \infty} M \cdot Q_2^M(1) = \lim_{M \to \infty} M \cdot \min_{x_i \in \Omega} J(x_1,\ldots,x_M) = \frac{5 \sqrt{3}}{54}.
\end{equation}
See the references above or \cite[Thm.~5]{BournePeletierTheil14}.
We generalise \eqref{eq:classical_crystallizationLB} and \eqref{eq:classical_crystallization_limit} to the unbalanced quantization problem in \cref{thm:lowerBound} and \cref{thm:asymptoticQuantization}, respectively. Roughly speaking, we show that again for $\mu=1$ the triangular lattice is optimal in the limit $M \to \infty$. For general $\mu \in L^1(\Omega)$, 
it is asymptotically optimal for the particles to locally form a triangular lattice with density determined by a nonlocal function of $\mu$.

While our quantization results are limited to two dimensions, this is also largely true for the classical quantization problem. In three dimensions it
is not known whether Gersho's Conjecture holds, although there is some numerical evidence for the case $p=2$ that optimal CVTs of $M$ points tend as $M \to \infty$ to the Voronoi diagram of the body-centered cubic (BCC) lattice, where each Voronoi cell is congruent to a truncated octahedron \cite{DuWang2005}. See also \cite{ChoLu}. For $p=2$ it has been proved that, \emph{amongst lattices}, the BCC lattice is optimal \cite{BarnesSloane1983}. 

For general $p$, $d$ and $\mu$, the scaling of the quantization error is known even if the optimal quantizer is not; Zador's Theorem \cite{Za82}, \cite[Cor.~33.3]{Gruber07} states that 
\begin{equation}
\label{eq:Zador}
\lim_{M \to \infty}
M^{\frac pd} \cdot Q^M_p(\mu) = 
c(p,d)
\,
\| \mu \|_{L^{\frac{d}{d+p}}(\Omega)}
\end{equation}
where the constant $c(p,d)$ is characterised by 
\[
c(p,d) = 
\lim_{M \to \infty} M^{\frac pd} \cdot Q^M_p(\Lebesgue \restr{[0,1]^d}),
\]
and where $\Lebesgue$ is the $d$-dimensional Lebesgue measure. For a modern proof using $\Gamma$-convergence see \cite{BouchitteJimenezRajesh2002} and \cite[Proposition 7.21]{Santambrogio-CV}.
For generalisations to quantization on Riemannian manifolds see \cite{Gr04}, \cite[Thm.~1.2]{Kloeckner2012} and \cite{AydinIacobelli}. It is an open problem to compute the optimal constant $c(p,d)$ except for $d=1$ and $d=2$, where 
\begin{equation}
\label{eq:Cp2}    
c(p,1) = \int_{-1/2}^{1/2} |x|^p \, \d x, \qquad
c(p,2) = \int_{\Hex(1)} |x|^p \, \d x,
\end{equation}
where $\Hex(1)$ is a regular hexagon of area $1$ centred at the origin $0$.
We recover Zador's Theorem for the case $d=2$, along with the optimal constant $c(p,2)$, as a special case of \cref{thm:asymptoticQuantization}; see \cref{example:Zador}.

\subsection{Outline and contribution}
\Cref{sec:background} collects relevant results from classical, unbalanced, and semi-discrete transport,
which will be generalized in \cref{sec:semiDiscreteUnbalanced} to the case of semi-discrete unbalanced transport. Finally, \Cref{sec:quantization} considers the unbalanced quantization problem.

In more detail, the contributions of this article are the following.
\begin{itemize}
\item \textbf{\Cref{sec:Tessellation}:}
We extend semi-discrete transport theory to the unbalanced case, most importantly a simple, geometric tessellation formulation (\cref{thm:UnbalancedTessellationDual}), optimality conditions that fully characterize primal and dual solutions (\cref{thm:UnbalancedOptimalityConditions}), and additional different primal and dual convex formulations.
Unlike in the balanced case, the dual potentials associated with the discrete mass locations do not only determine the tessellation of the continuous measure, but also the density of the optimal transport plan.
Particular attention needs to be paid to areas where the ground transport cost function becomes infinite.
Special cases of these results were derived in \cite{LeclercFlows2020,SarrazinMFG2020} to study a Lagrangian discretization of Wasserstein gradient flows and variational mean field games.
\item \textbf{\Cref{sec:UnbalancedNumerics}:}
We develop numerical algorithms for solving the semi-discrete unbalanced transport problem
and numerically illustrate novel phenomena of unbalanced transport (\cref{exm:modelComparison}).
In particular, we show qualitative differences between different unbalanced transport models
and examine the effect of changing the length scale, which typically is intrinsic to unbalanced transport models.
\item \textbf{\Cref{sec:Quantization,sec:QuantizationNumerics}:}
We extend the theory of optimal transport-based quantization of measures to unbalanced transport,
deriving in particular an equivalent Voronoi tessellation problem (\cref{thm:quantizationTessellation}),
which turns out to be a natural generalization of the known corresponding formulation in classical transport.
The interesting fact here is that the simple geometric Voronoi tessellation structure survives when passing from balanced to unbalanced transport, but the mass of the generating points now depends in a more complex way on the mass within their cells.
We also illustrate unbalanced quantization numerically, extending the standard algorithms (including Lloyd's algorithm) to the unbalanced case.
\item \textbf{\Cref{sec:Crystallization}:}
In two spatial dimensions, where crystallization results from discrete geometry are available,
we derive the optimal asymptotic quantization cost and the optimal asymptotic point density for quantizing a given measure $\mu$ using unbalanced transport (\cref{thm:asymptoticQuantization}).
Our result includes Zador's Theorem for classical, balanced quantization as a special case; see \cref{example:Zador}.
As is common in asymptotic quantization, we consider a spatial rescaling of the domain as the number of points increases and the most interesting regime is where the rescaled point density converges to a non-zero, finite limit.
While in the balanced case, the rescaled asymptotic cost only depends on the growth behaviour of the transport ground cost function, in the unbalanced setting we now observe an interplay between the rescaled point density and the intrinsic length scale of unbalanced transport.
An interesting, novel effect in this unbalanced setting is that the optimal point density depends nonlocally on the global mass distribution
in such a way that whole regions with positive measure may be completely neglected in favour of regions with higher mass.
\end{itemize}

\subsection{Setting and notation}
\label{sec:setting}
Throughout this article we work in a domain $\Omega = \overline{U}$ for $U \subset \R^d$ open.
(In principle, the results could be extended to more general metric spaces such as Riemannian manifolds.)
The Euclidean distance on $\R^d$ is denoted $d(\cdot,\cdot)$,
and we will write $\pi_i:\Omega \times \Omega \to \Omega$, for the projections $\pi_i(x_1,x_2)=x_i$, $i=1,2$.
The ($d$-dimensional) Lebesgue measure of a measurable set $A\subset\R^d$ will be indicated by $\Lebesgue(A)$ or $|A|$ for short, its diameter by $\diam(A)$.

By $\measp(\Omega)$ we denote the set of nonnegative Radon measures on $\Omega$,
and $\prob(\Omega)\subset\measp(\Omega)$ is the subset of probability measures.
The notation $\mu\ll\nu$ for two measures $\mu,\nu\in\measp(\Omega)$ indicates absolute continuity of $\mu$ with respect to $\nu$,
and the corresponding Radon--Nikodym derivative is written as $\RadNik{\mu}{\nu}$.
The restriction of $\mu\in\measp(\Omega)$ to a measurable set $A\subset\R^d$ is denoted $\mu\restr A$, and its support is denoted $\spt\mu$.
For a Dirac measure at a point $x\in\R^d$ we write $\delta_x$.
The pushforward of a measure $\mu$ under a measurable map $T$ is denoted $\pushforward{T}\mu$.

The spaces of Lebesgue integrable functions on $U$ or of $\mu$-integrable functions with $\mu\in\measp(\Omega)$ are denoted $L^1(U)$ and $L^1(\mu)$, respectively. Continuous functions on $\Omega$ are denoted by $\mathcal{C}(\Omega)$.

\section{Background}\label{sec:background}
The purpose of this section is a short introduction to classical, unbalanced, and semi-discrete transport.

\subsection{Optimal transport}
Here we briefly recall the basic setting of optimal transport. For a thorough introduction we refer, for instance, to \cite{Santambrogio-OTAM,Villani-OptimalTransport-09}.
For $\mu$, $\nu \in \prob(\Omega)$ the set
\begin{align}
\label{eq:Coupling}
\Gamma(\mu,\nu) = \{ \gamma \in \prob(\Omega \times \Omega)\,|\,\pushforward{\pi_1}{\gamma}=\mu,\, \pushforward{\pi_2}{\gamma}=\nu\}
\end{align}
is called the \emph{couplings} or \emph{transport plans} between $\mu$ and $\nu$. A measure $\gamma \in \Gamma(\mu,\nu)$ can be interpreted as a rearrangement of the mass of $\mu$ into $\nu$ where $\gamma(x,y)$ intuitively describes how much mass is taken from $x$ to $y$.
The total cost associated to a coupling $\gamma$ is given by
\begin{align}
	\label{eq:OTCost}
	\int_{\Omega \times \Omega} \c(x,y)\,\d\gamma(x,y)
\end{align}
where $\c : \Omega \times \Omega \to [0,\infty]$ and $\c(x,y)$ specifies the cost of moving one unit of mass from $x$ to $y$.
The \emph{optimal transport problem} asks for finding a $\gamma$ that minimizes \eqref{eq:OTCost} among all couplings $\Gamma(\mu,\nu)$,
\begin{align}
	\label{eq:OTProblem}
	\WOT(\mu,\nu) = \inf \left\{ \int_{\Omega \times \Omega} \c\,\d\gamma\,\middle|\,\gamma \in \Gamma(\mu,\nu) \right\}\,.
\end{align}
Under suitable regularity assumptions on $c$, existence of minimizers follows from standard compactness and lower semi-continuity arguments.
\begin{theorem}[\protect{\cite[Thm.\,4.1]{Villani-OptimalTransport-09}}]
	\label{thm:OTProblemExistence}
	If $c : \Omega \times \Omega \to [0,\infty]$ is lower semi-continuous, then minimizers of \eqref{eq:OTProblem} exist.
	The minimal value may be $+\infty$.
\end{theorem}

\subsection{Unbalanced transport}
The optimal transport problem \eqref{eq:OTProblem} only allows the comparison of measures $\mu$, $\nu$ with equal mass. Otherwise, the feasible set $\Gamma(\mu,\nu)$ is empty.
Therefore, so-called unbalanced transport problems have been studied, where mass may be created or annihilated during transport and thus measures of different total mass can be compared in a meaningful way. See \cref{sec:Motivation} for context and references.

Throughout this article we consider unbalanced optimal entropy-transport problems as studied in \cite{LieroMielkeSavare-HellingerKantorovich-2015a}.
The basic idea is to replace the hard marginal constraints $\pushforward{\pi_1}{\gamma}=\mu$, $\pushforward{\pi_2}{\gamma}=\nu$ in \eqref{eq:Coupling} with soft constraints where the deviation between the marginals of $\gamma$ and the measures $\mu$ and $\nu$ is penalized by a marginal discrepancy function.
This allows more flexibility for feasible $\gamma$.
We focus on a subset of the family of marginal discrepancies considered in \cite{LieroMielkeSavare-HellingerKantorovich-2015a}.
\begin{definition}[Marginal discrepancy]
	\label{def:MarginalDiscrepancy}
	Let $\Floc : [0,\infty) \to [0,\infty]$ be proper, convex, and lower semi-continuous with $\lim_{s \to \infty} \tfrac{\Floc(s)}{s}=\infty$.
	For a given measure $\mu \in \measp(\Omega)$, the function $\Floc$ induces a \emph{marginal discrepancy} $\Fint(\cdot|\mu) : \measp(\Omega) \to [0,\infty]$ via
	\begin{align}
		\Fint(\rho|\mu) = \begin{cases}
			\displaystyle \int_\Omega \Floc\big(\RadNik{\rho}{\mu}\big)\,\d \mu & \tn{if } \rho \ll \mu, \\
			+ \infty & \tn{otherwise.}
			\end{cases}
	\end{align}
	Note that the integrand is only defined $\mu$-almost everywhere.
	$\Fint$ is (sequentially) weakly-$\ast$ lower semi-continuous \cite[Thm.\,2.34]{AFP00}.

	We extend the domain of definition of $\Floc$ to $\R$ by setting $\Floc(s)=\infty$ for $s<0$. The Fenchel--Legendre conjugate of $\Floc$ is then the convex function $\Floc^\ast : \R \to (-\infty,+\infty]$ defined by
\[
\Floc^\ast(z) = \sup_{s \in \R} \left( z \cdot s - \Floc(s) \right) = \sup_{s \geq 0} \left( z \cdot s - \Floc(s) \right).
\]
\end{definition}

\begin{example}[Kullback--Leibler divergence]
	\label{exp:KL}
	The Kullback--Leibler divergence is an example of \cref{def:MarginalDiscrepancy} for the choice $\FlocKL : [0,\infty) \to [0,\infty)$,
\[
\FlocKL(s)= \begin{cases}
			s \log s -s + 1 & \tn{if } s>0, \\
			1 & \tn{if } s=0.
		\end{cases}
\]	
The Fenchel--Legendre conjugate is given by $\FlocKL^\ast(z)=e^z-1$.
\end{example}

\begin{definition}[Unbalanced optimal transport problem]
	\label{def:UnbalancedTransport}
	Let $\Floc$ be as in \cref{def:MarginalDiscrepancy} and let $\Fint$ be the induced marginal discrepancy.
	Let $\mu$, $\nu \in \measp(\Omega)$ and $\c : \Omega \times \Omega \to [0,\infty]$ be lower semi-continuous.
	The corresponding unbalanced transport cost $\E : \measp(\Omega \times \Omega) \to [0,\infty]$ is given by
	\begin{align}
		\label{eq:UnbalancedEnergy}
		\E(\gamma) = \int_{\Omega \times \Omega} \c\,\d\gamma + \Fint(\pushforward{\pi_1}{\gamma}|\mu) + \Fint(\pushforward{\pi_2}{\gamma}|\nu)
	\end{align}
	and induces the optimization problem
	\begin{align}
		\label{eq:UnbalancedProblem}
		\W(\mu,\nu) = \inf \left\{ \E(\gamma)\,\middle|\,\gamma \in \measp(\Omega \times \Omega)\right\}.
	\end{align}
\end{definition}

The interaction between the terms in \eqref{eq:UnbalancedEnergy} that penalise transport and mass change introduces an intrinsic length scale for transport that you do not see in classical balanced transport. This is discussed in Example \ref{ex:LengthScales}.

\begin{theorem}[\protect{\cite[Thm.\,3.3]{LieroMielkeSavare-HellingerKantorovich-2015a}}]
	\label{thm:UnbalancedProblemExistence}
	Minimizers of \eqref{eq:UnbalancedProblem} exist.
  The minimal value may be $+\infty$.
\end{theorem}

\begin{remark}
Observe that $\Fint(\rho|\mu)=\infty$ whenever $\rho \not \ll \mu$ and $\Fint(\rho|\nu) = \infty$ whenever $\rho \not \ll \nu$. This guarantees that $\pushforward{\pi_1}{\gamma} \ll \mu$  and $\pushforward{\pi_2}{\gamma} \ll \nu$ for all feasible $\gamma$, where \emph{feasible} means that $\mathcal{E}(\gamma) < \infty$.
Thus, when $\mu \ll \Lebesgue$ and $\nu$ is discrete, as in the semi-discrete setting (which will be discussed in the following section), then the first and second marginal of any feasible $\gamma$ will share these properties.
\end{remark}

\begin{remark}
For simplicity we assume that the same marginal discrepancy is applied to both marginals in \eqref{eq:UnbalancedEnergy}, but of course in some cases it may be more appropriate to consider two different discrepancies. All results in this article generalize to this case in a canonical way.
\end{remark}

In this article we focus on cost functions $\c$ that can be written as increasing functions of the distance between $x$ and $y$.
\begin{definition}[Radial cost]
	\label{def:RadialCost}
	A cost function $\c : \Omega \times \Omega \to [0,\infty]$ is called \emph{radial} if it can be written as $\c(x,y) = \cRad(d(x,y))$ for a strictly increasing function $\cRad : [0,\infty) \to [0,\infty]$, continuous on its domain with $\cRad(0)=0$.
\end{definition}

Note that the cost $c$ need not be \emph{twisted} (twistedness means that $y\mapsto\nabla_xc(x,y)$ is injective for all $x$, see \cite[Definition 1.16]{Santambrogio-OTAM}, which leads to some technical complications.
The following examples shall be used throughout for illustration.
They all feature a \emph{radial} transport cost $c$ in the sense of \cref{def:RadialCost}.
\begin{example}[Unbalanced transport models]\label{exm:models}\hfill
\begin{enumerate}[label=(\alph*),ref=\alph*]
	\item\label{item:ModelOT}
	\textbf{Standard Wasserstein-2 distance (W2).}
	Classical balanced optimal transport can be recovered as a special case of \cref{def:UnbalancedTransport} by choosing $\Fint(\rho|\mu)=0$ if $\rho=\mu$ and $\infty$ otherwise. This corresponds to
	\begin{align*}
	\Floc(s) = \iota_{\{1\}}(s) & = \begin{cases}
		0 & \tn{if } s = 1, \\
		\infty & \tn{otherwise,}
		\end{cases}
	&
	\Floc^\ast(z) & = z\,.
	\end{align*}
	Then $\E(\gamma)<\infty$ only if $\gamma \in \Gamma(\mu,\nu)$, and therefore \eqref{eq:UnbalancedProblem} reduces to \eqref{eq:OTProblem}.
	In particular, the Wasserstein-2 setting is obtained for $\c(x,y)=d(x,y)^2$, and the Wasserstein-2 distance is defined by $\W_2(\mu,\nu)=\sqrt{\W(\mu,\nu)}$.
	\item\label{item:ModelGHK}
	\textbf{Gaussian Hellinger--Kantorovich distance (GHK).} This distance is introduced in \cite[Thm.\,7.25]{LieroMielkeSavare-HellingerKantorovich-2015a} using
	\begin{align*}
	\Floc(s) & = \FlocKL(s) = \begin{cases} s\log s-s+1 & \tn{if } s >0, \\ 1 & \tn{if } s=0,\end{cases}
	&
	\Floc^\ast(z) & = e^z-1\,,
	&
	\c(x,y)& =d(x,y)^2\,.
	\end{align*}
	\item\label{item:ModelHK}
	\textbf{Hellinger--Kantorovich distance (HK).}
	This important instance of unbalanced transport was introduced in different formulations in \cite{KMV-OTFisherRao-2015,ChizatOTFR2015,LieroMielkeSavare-HellingerKantorovich-2015a} whose mutual relations are described in \cite{ChizatDynamicStatic2018}. In \cref{def:UnbalancedTransport} one chooses
	\begin{gather*}
	\Floc(s) = \FlocKL(s)\,,\quad
	\Floc^\ast(z) = e^z-1\,, \\
	\c(x,y)=\cHK(x,y) = \begin{cases} -2\log\big[\cos\big(d(x,y)\big)\big] & \tn{if } d(x,y) < \tfrac{\pi}{2}, \\ \infty & \tn{otherwise,} \end{cases}\quad
	\end{gather*}
	and the Hellinger--Kantorovich distance is defined by $\HK(\mu,\nu)=\sqrt{\W(\mu,\nu)}$.
	The distance $\HK$ is actually a geodesic distance on the space of non-negative measures over a metric base space. From $\cHK(x,y)=\infty$ for $d(x,y)\geq \tfrac{\pi}{2}$, we learn that mass is never transported further than $\tfrac{\pi}{2}$ in this setting.
	\item\label{item:ModelQuadratic}
	\textbf{Quadratic regularization (QR).}
	The Kullback--Leibler discrepancy $\FlocKL$ used in both previous examples has an infinite slope at $0$, which in \cref{def:UnbalancedTransport} leads to a strong incentive to achieve $\pushforward{\pi_1}{\gamma}\gg\mu$ and $\pushforward{\pi_2}{\gamma}\gg\nu$.
	The following mere quadratic discrepancy does not have this property,
	\begin{align*}
	\Floc(s) & = (s-1)^2\,, &
	\Floc^\ast(z) & = \begin{cases} \tfrac{z^2}{4}+z & \tn{if } z \geq -2, \\ -1 & \tn{otherwise,}\end{cases}
	&
	\c(x,y) & =d(x,y)^2\,.
	\end{align*}
\end{enumerate}
\end{example}

Unsurprisingly, the structure of the function $\Floc$ has a great influence on the behaviour of the unbalanced optimization problem \eqref{eq:UnbalancedProblem}. Often it is helpful to analyze corresponding dual problems where the conjugate function $\Floc^\ast$ appears. We gather some properties of $\Floc^\ast$, implied by the assumptions on $\Floc$ in \cref{def:MarginalDiscrepancy} and on some additional assumptions that we will occasionally make in this article.
\begin{lemma}[Properties of $\Floc^\ast$]
	\label{lem:ConjugateProperties}
	Let $\Floc$ satisfy the assumptions given in \cref{def:MarginalDiscrepancy}. Then
	\begin{enumerate}[label=(\roman*),ref=\roman*]
		\item{} $\Floc^\ast(z)>-\infty$ for $z\in\R$;
			\label{item:FlocAstNotNegInf}
		\item{} $\Floc^\ast$ is increasing;
			\label{item:FlocAstIncreasing}
		\item{} $\Floc^\ast(z)\leq 0$ for $z\leq 0$;
			\label{item:FlocAstNegNeg}
		\item{} $\Floc^\ast(z)<\infty$ for $z \in (0,\infty)$;
			\label{item:FlocAstNotInf}
		\item{} $\Floc^\ast$ is real-valued and continuous on $\R$;
			\label{item:FlocAstFinite}
		\item{} if $\Floc$ is strictly convex on its domain, then $\Floc^\ast$ is continuously differentiable on $\R$;
			\label{item:FlocAstDiffble}
		\item{} $\Floc^\ast(z)\ge-\Floc(0)$ for all $z \in \R$;
			\label{item:FlocAstBound}
		\item{} if $\inf\{x \geq 0 | \Floc(x)<\infty\}=0$ (which holds in particular when $\Floc(0)<\infty$), then
		\begin{align*}
		\lim_{z\to-\infty}\min\partial\Floc^\ast(z)=\lim_{z\to-\infty}\max\partial\Floc^\ast(z)=0.
		\end{align*}
			\label{item:FlocZeroFinite}
	\end{enumerate}
\end{lemma}
\begin{proof}
\textbf{(\ref{item:FlocAstNotNegInf})} Since $\Floc$ is proper, we can find $s \in (0,\infty)$ with $\Floc(s)<\infty$. Then for all $z \in \R$, $\Floc^\ast(z) = \sup_{x \geq 0} (z \cdot x -\Floc(x)) \geq z \cdot s -\Floc(s) > -\infty$.

\textbf{(\ref{item:FlocAstIncreasing})} Let $z_1 \leq z_2$. Then $\Floc^\ast(z_2) = \sup_{x \geq 0} (z_2 \cdot x - \Floc(x)) \geq \sup_{x \geq 0} (z_1 \cdot x - \Floc(x)) = \Floc^\ast(z_1)$.

\textbf{(\ref{item:FlocAstNegNeg})} Let $z \leq 0$. Since $F \geq 0$, then $\Floc^\ast(z) = \sup_{x \geq 0} (z \cdot x - \Floc(x)) \leq \sup_{x \geq 0} z \cdot x = 0$.

\textbf{(\ref{item:FlocAstNotInf})} Let $z \in (0,\infty)$. Since $\Floc \geq 0$, $\Floc^\ast(z)=\infty$ is only possible if any maximizing sequence $x_1,x_2,\ldots$ for $\Floc^\ast(z) = \sup_{x \geq 0} (z \cdot x - \Floc(x))$ diverges to $\infty$. However, $\lim_{n \to \infty} (z \cdot x_n - \Floc(x_n)) = \lim_{x \to \infty} x \big(z-\tfrac{\Floc(x)}{x}\big) = -\infty$ since $\lim_{s \to \infty} \tfrac{F(s)}{s} = \infty$. So $\Floc^\ast(z)<\infty$.

\textbf{(\ref{item:FlocAstFinite})} \eqref{item:FlocAstNotNegInf}, \eqref{item:FlocAstNotInf}, and \eqref{item:FlocAstNegNeg} imply $\dom(\Floc^\ast)=\R$. By convexity, $\Floc^\ast$ is therefore continuous.

\textbf{(\ref{item:FlocAstDiffble})} This is a special case of a classical result in convex analysis, which can be found, for instance, in \cite[Thm.~26.3]{Rockafellar1972Convex}.

\textbf{(\ref{item:FlocAstBound})} Let $z \in \R$. Then $\Floc^\ast(z) = \sup_{x \geq 0} (z \cdot x - \Floc(x)) \geq -\Floc(0)$.

\textbf{(\ref{item:FlocZeroFinite})} Let $z_1,z_2,\ldots$ and $u_1,u_2,\ldots$ be sequences with $z_n \to -\infty$ as $n \to \infty$ and $u_n \in \partial \Floc^\ast(z_n)$. By monotonicity of $\Floc^\ast$, \eqref{item:FlocAstIncreasing}, we have $u_n \geq 0$. By \eqref{item:FlocAstNegNeg} and convexity one finds for any $a \geq 0$ with $\Floc(a)<\infty$ that
	$0 \geq \Floc^\ast(0) \geq \Floc^\ast(z_n) + u_n \cdot (0-z_n) \geq a \cdot z_n -\Floc(a) + u_n \cdot |z_n|$,
	which implies that $u_n \leq \Floc(a)/|z_n|+a$. This implies that $\limsup_n u_n \leq a$. Sending now $a \to 0$ yields the claim.
\end{proof}

\begin{remark}[Feasibility for finite $\Floc(0)$]
Note that for $\Floc(0)<\infty$ the trivial transport plan $\gamma=0$ leads to a finite cost in \eqref{eq:UnbalancedEnergy} so that $\W(\mu,\nu)<\infty$ for all $\mu,\nu\in\measp(\Omega)$.
\end{remark}

\subsection{Semi-discrete transport}
\label{sec:BackgroundSemiDisc}

An important special case of the classical balanced optimal transport problem \eqref{eq:OTProblem} is the case where $\mu$ is absolutely continuous with respect to the Lebesgue measure,
\begin{subequations}
\label{eqn:conditionMarginals}
\begin{equation}\label{eqn:conditionMu}
\mu \ll \Lebesgue\,,
\end{equation}
and $\nu$ is a discrete measure,
\begin{equation}\label{eqn:conditionNu}
\nu = \sum_{i=1}^M m_i \delta_{x_i}\,,
\end{equation}
\end{subequations}
with $m_i >0$, $x_i \in\Omega$, and $x_i \ne x_j$ for $i \ne j$.
See \cref{sec:Motivation} for context and references.
In this section we review the special structure of problem \eqref{eq:OTProblem} that follows from \eqref{eqn:conditionMarginals}.
For instance, optimal couplings for \eqref{eq:OTProblem} turn out to have a very particular form: the domain $\Omega$ is partitioned into cells, one cell for each discrete point $x_i$, and mass will only be transported from each cell to its corresponding discrete point.
The shape of the cells is determined by $\mu$, $\nu$ and the cost function $\c$ and can be expressed with the aid of \cref{def:Cells}. Problem \eqref{eq:OTProblem} can be rewritten explicitly as an optimization problem in terms of the cells. This tessellation formulation is given in \cref{thm:OTSemiDiscrete}, and its optimality conditions are described in \cref{thm:OTSemiDiscreteOptimality}.

\begin{definition}[Generalized Laguerre cells]
Given a transportation cost $\c$ and points $x_1,\ldots,x_M \in \Omega$, we define the \emph{generalized Laguerre cells} corresponding to the weight vector $w\in\R^M$ by
\label{def:Cells}
\begin{align}
	\C_i(w) = \left\{ x \in \Omega\,\middle|\,\c(x,x_i) < \infty, \, \c(x,x_i) - w_i \leq \c(x,x_j) - w_j \tn{ for all } j\in\{1,\ldots,M\} \right\}
\end{align}
for $i \in \{1,\ldots, M \}$.
The residual of $\Omega$, the set not covered by any of the cells $\C_i$, is defined by
\begin{align}
	\label{eq:Resid}
	\resid = \left\{x \in \Omega\,\middle|\,c(x,x_i)=\infty \tn{ for all } i \in\{1,\ldots,M\} \right\}.
\end{align}
Note that $\resid$ can also be written as $\resid = \Omega \setminus \big(\bigcup_{i=1}^M C_i(w)\big)$, which does not depend on $w \in \R^M$.
Note also that, if $a=\lambda(1,1,\ldots,1) \in \R^M$ is a vector with all components equal, then $C_i(w+a)=C_i(w)$ for all $i \in \{ 1, \ldots, M \}$.
\end{definition}

\begin{example}[Generalized Laguerre cells \cite{AurenhammerKleinLee}] \hfill
\begin{enumerate}[label=(\alph*),ref=\alph*]
\item\textbf{Voronoi diagrams.} If $c$ is radial (see \cref{def:RadialCost}) and finite, then the collection of generalized Laguerre cells with weight vector $0 \in \R^M$, $\{ C_i(0) \}_{i=1}^M$, is just the Voronoi diagram generated by the points $x_1,\ldots,x_M$. The residual set $\resid=\emptyset$.

\item\textbf{Laguerre diagrams or power diagrams.} If $c(x,y)=|x-y|^2$, then the collection of generalized Laguerre cells $\{ C_i(w) \}_{i=1}^M$ is known as the \emph{Laguerre diagram} or \emph{power diagram} generated by the weighted points $(x_1,w_1), \ldots, (x_M,w_M)$. The cells $C_i$ are the intersection of convex polytopes with $\Omega$. The residual set $\resid=\emptyset$.

\item\textbf{Apollonius diagrams.} If $c(x,y)=|x-y|$, then the collection of generalized Laguerre cells $\{ C_i(w) \}_{i=1}^M$ is known as the \emph{Apollonius diagram} generated by the weighted points $(x_1,w_1), \ldots, (x_M,w_M)$. The cells $C_i$ are the intersection of star-shaped sets with $\Omega$, and in two dimensions the boundaries between cells are arcs of hyperbolas. Again, $\resid=\emptyset$.
\end{enumerate}
\end{example}

\begin{theorem}[Dual tessellation formulation for semi-discrete transport]
	\label{thm:OTSemiDiscrete}
	Assume that $\mu$ and $\nu$ satisfy \eqref{eqn:conditionMarginals} and $\mu(\Omega)=\nu(\Omega)$.
	Let the cost function $c$ be radial (see \cref{def:RadialCost}) and $\WOT(\mu,\nu)<\infty$.
	Then
	\begin{align}
		\label{eq:SemiDiscreteUnconstrainedDual}
		\WOT(\mu,\nu) = \sup \left\{
			\sum_{i=1}^M \int_{\C_i(w)} \c(x,x_i)\,\d\mu(x)
			+ \sum_{i=1}^M \big(m_i - \mu(\C_i(w))\big) \cdot w_i
			\,\middle|\, w \in \R^M \right\}.
	\end{align}
\end{theorem}
\begin{remark}[Existence of optimal weights]\label{rem:existenceSemiDiscreteDual}
	Maximizers for \eqref{eq:SemiDiscreteUnconstrainedDual} do not always exist, even when $\WOT(\mu,\nu)<\infty$. A simple sufficient condition for existence is that $\c$ is bounded from above on $\Omega \times \Omega$. More details can be found, for instance, in \cite[Thm.~5.10]{Villani-OptimalTransport-09}.
\end{remark}

\begin{theorem}[Optimality conditions]
	\label{thm:OTSemiDiscreteOptimality}
	Under the conditions of \cref{thm:OTSemiDiscrete},
	a coupling $\gamma \in \Gamma(\mu,\nu)$ and a vector $w \in \R^M$ are optimal for $\WOT(\mu,\nu)$ in \eqref{eq:OTProblem} and \eqref{eq:SemiDiscreteUnconstrainedDual} respectively, if and only if
	\begin{align}
		\label{eq:SemiDiscreteOptimalGamma}
		\gamma & = \sum_{i=1}^M \left( \mu \restr{\C_i(w)} \otimes \delta_{x_i} \right), &
		\mu(\C_i(w)) & = m_i \tn{ for } i \in \{1,\ldots,M\}.
	\end{align}
\end{theorem}

Proofs of \cref{thm:OTSemiDiscrete} and \cref{thm:OTSemiDiscreteOptimality} can be found below and for example in \cite{KitagawaMerigotThibert} and \cite[Section 4]{MerigotThibertOT} for twisted costs $c$. 
We provide proofs of \cref{thm:OTSemiDiscrete,thm:OTSemiDiscreteOptimality} for two reasons.
 They serve as preparation for the proof of \cref{thm:UnbalancedTessellationDual,thm:UnbalancedOptimalityConditions} in the case of semi-discrete \emph{unbalanced} transport, which generalize \cref{thm:OTSemiDiscrete,thm:OTSemiDiscreteOptimality}.
In addition, they deal with the technical aspect that our cost function $c$ is not necessarily twisted and may take the value $+\infty$ at finite distances. In particular, $c$ does not satisfy the assumptions in \cite{KitagawaMerigotThibert,MerigotThibertOT}. We rely on the following lemma, which essentially provides the existence of a Monge map in the semi-discrete setting (\cref{cor:SemiDiscreteMongeMap}). For twisted costs this result can be found in \cite[Proposition 37]{MerigotThibertOT}.

\begin{lemma}[Laguerre cell boundaries]
	\label{lem:CellNullOverlap}
	Let the cost function $\c$ be radial in the sense of \cref{def:RadialCost} and let $\{x_i\}_{i=1}^M$ be $M$ \emph{distinct} points in $\Omega$. The induced generalized Laguerre cells satisfy $|\C_i(w) \cap \C_j(w)| = 0$ for $i \neq j$ and any $w\in\R^M$.
\end{lemma}
\begin{proof}
	Fix $i \neq j$ and $w \in \R^M$ and recall that $\c(x,y)=\cRad(d(x,y))$.
	We have
	\begin{equation*}
	\C_i(w)\cap\C_j(w)=\bigcup_{n\in\N}A_n
	\quad\text{for }
	A_n=\{x\in\Omega\,|\,\c(x,x_i)-w_i=\c(x,x_j)-w_j,\,\c(x,x_i)\leq n\}\,,
	\end{equation*}
	and we will show that the $d$-dimensional Hausdorff measure of each $A_n$ is zero, $\hd^d(A_n)=0$, which implies $|A_n|=0$ and thus also $|\C_i(w) \cap \C_j(w)| = 0$.
	Indeed, abbreviating $f=d(\cdot,x_i)$ and noting that the Jacobian of $f$ equals $1$ almost everywhere, the coarea formula \cite[Thm.\,2.93]{AFP00} yields
	\begin{equation*}
	\hd^d(A_n)
	=\int_{A_n}1\,\d\hd^d
  =\int_\R\hd^{d-1}(A_n\cap f^{-1}(t))\,\d t
  =\int_0^{\cRad^{-1}(n)}\hd^{d-1}(A_n\cap f^{-1}(t))\,\d t\,.
	\end{equation*}
	Now, for $t \in [0,\cRad^{-1}(n)]$,
\[
A_n\cap f^{-1}(t)=\{x\in\Omega\,|\,d(x,x_i)=t\text{ and }d(x,x_j)\in\cRad^{-1}(\cRad(d(x,x_i))+w_j-w_i)\},
\]
	where $\cRad^{-1}(\cRad(d(x,x_i))+w_j-w_i)$ is either empty or single-valued due to the strict monotonicity of $\cRad$.
	Hence, $A_n\cap f^{-1}(t)$ is contained in the intersection of two non-concentric $(d-1)$-dimensional spheres and thus is $\hd^{d-1}$-negligible.
\end{proof}

\begin{proof}[Proof of \cref{thm:OTSemiDiscrete}]
	By Kantorovich duality \cite[Thm.\,5.10]{Villani-OptimalTransport-09} we can write
	\begin{multline}
		\label{eq:OTDuality}
		\WOT(\mu,\nu) = \sup \left\{
			\int_\Omega \phi\,\d \mu + \int_\Omega \psi\,\d \nu \,\middle|\,
				\phi \in L^1(\mu),\, \psi \in L^1(\nu), \right. \\
			 \left. \vphantom{\int_\Omega} \phi(x) + \psi(y) \leq c(x,y) \,\forall\,(x,y) \in \Omega\times\Omega \right\}\,.
	\end{multline}
	Since $\nu$ is discrete, $L^1(\nu)$ is isomorphic to $\R^M$ under the isomorphism $I:L^1(\nu)\to\R^M$, $\psi\mapsto(\psi(x_1),\ldots,\psi(x_M))$. The above dual problem thus becomes
  \begin{multline*}
    \WOT(\mu,\nu) = \sup \left\{
      \int_\Omega \phi\,\d \mu + \sum_{i=1}^M w_i\,m_i \,\middle|\,
        \phi \in L^1(\mu),\, w\in\R^M, \right. \\
       \left. \vphantom{\int_\Omega} \phi(x) + w_i \leq c(x,x_i) \,\forall\,x\in\Omega,i\in\{1,\ldots,M\}\right\}\,.
  \end{multline*}
	Next, for fixed $w$, one can explicitly maximize over $\phi$, which corresponds to pointwise maximization subject to the constraint. We denote the maximizer by $\phi_{w}$ to emphasize the dependency on $w$,
	\begin{align}
		\label{eq:PhiW}
		\phi_{w}(x) = \min\big\{ c(x,x_i)-w_i\,|\,i =1,\ldots,M \big\}\,.
	\end{align}
	Since $\infty>\WOT(\mu,\nu)-\sum_{i=1}^Mw_i\,m_i\geq\int_\Omega\phi_w\,\d\mu$ and $\phi_w$ is bounded from below (as $c$ is so in our setting) one must have $\phi_{w} \in L^1(\mu)$ for all $w \in \R^M$, and we find
	\begin{align}
		\label{eq:ESDPre}
		\WOT(\mu,\nu) = \sup\{ \E_{\tn{SD}}(w)\,|\,w \in \R^M\}
		\quad \tn{with} \quad
		\E_{\tn{SD}}(w) = \int_\Omega \phi_{w}(x)\,\d \mu(x) + \sum_{i=1}^M w_i\,m_i\,.
	\end{align}
	Since $\phi_w\in L^1(\mu)$ for any $w\in\R^M$, the residual set $\resid$ must be $\mu$-negligible; likewise, the intersection of generalized Laguerre cells is $\mu$-negligible by \cref{lem:CellNullOverlap}.
	Consequently,
	\begin{equation*}
	\E_{\tn{SD}}(w)
  = \sum_{i=1}^M\int_{\C_i(w)} \phi_{w}(x)\,\d \mu(x) + \sum_{i=1}^M w_i\,m_i
  = \sum_{i=1}^M\int_{\C_i(w)} [\c(x,x_i)-w_i]\,\d \mu(x) + \sum_{i=1}^M w_i\,m_i\,,
	\end{equation*}
	which leads to the desired result.
\end{proof}

\begin{proof}[Proof of \cref{thm:OTSemiDiscreteOptimality}]
	The condition $\gamma \in \Gamma(\mu,\nu)$ implies that $\gamma$ can be written as $\gamma=\sum_{i=1}^M \gamma_i \otimes \delta_{x_i}$ where $\gamma_i \in \measp(\Omega)$, $\gamma_i(A) := \gamma(A \times \{ x_i\})$. Observe that $\sum_{i=1}^M \gamma_i = \mu$ and $\gamma_i(\Omega)=m_i$.
	We obtain
	\begin{align}
		\label{eq:SemiDiscOptPrimal}
		\WOT(\mu,\nu) \leq \int_{\Omega \times \Omega} c\,\d\gamma
		= \sum_{i=1}^M \int_\Omega c(x,x_i)\,\d\gamma_i(x)\,,
	\end{align}
	where the inequality is an equality if and only if $\gamma$ is optimal. Let $w \in \R^M$.
	From \eqref{eq:ESDPre} with $\phi_{w}$ given by \eqref{eq:PhiW} we find
	\begin{align}
		\label{eq:SemiDiscOptDual}
		\WOT(\mu,\nu) \geq \int_\Omega \phi_{w}(x)\,\d \mu(x) + \sum_{i=1}^M w_i\,m_i
		= \sum_{i=1}^M \int_\Omega \left[\phi_{w}(x)+w_i\right]\,\d\gamma_i(x)\,,
	\end{align}
	where the inequality is an equality if and only if $w$ is optimal.
	Subtracting \eqref{eq:SemiDiscOptDual} from \eqref{eq:SemiDiscOptPrimal} yields
	\begin{align}
		\label{eq:SemiDiscOptPrimalDual}
		0 \leq \sum_{i=1}^M \int_\Omega \left[c(x,x_i)-w_i-\phi_{w}(x)\right]\,\d\gamma_i(x).
	\end{align}
	 with equality if and only if $\gamma$ and $w$ are optimal.
	 By definition of $\phi_{w}$ the integrand in each term of the sum is nonnegative and strictly positive for $x \notin \C_i(w)$. Therefore \eqref{eq:SemiDiscOptPrimalDual} is an equality if and only if $\gamma_i$ is concentrated on $\C_i(w)$ for all $i \in \{1,\ldots,M\}$.
	 Combining absolute continuity with respect to the Lebesgue measure of $\mu$ and $\gamma_i$ and \cref{lem:CellNullOverlap} implies that the unique choice is $\gamma_i = \mu \restr{C_i(w)}$.
	 Due to the second marginal constraint this implies $\mu(C_i(w))=\gamma_i(\Omega)=m_i$.	
\end{proof}

The above results imply the existence of an optimal Monge map for the semi-discrete problem.
\begin{corollary}[Existence of Monge map]
	\label{cor:SemiDiscreteMongeMap}
	If a maximizer $w\in\R^M$ of \eqref{eq:SemiDiscreteUnconstrainedDual} exists (cf.~\cref{rem:existenceSemiDiscreteDual}), then the optimal coupling $\gamma$ in \eqref{eq:OTProblem} is induced by a transport map $T : \Omega \to \{x_i\}_{i=1}^M \subset \Omega$, $\gamma = \pushforward{(\Id \times T)}{\mu}$, defined by $T(x)=x_i$ when $x \in C_i(w)$. By virtue of \cref{lem:CellNullOverlap} and since $\mu \ll \Lebesgue$, $T$ is well-defined $\mu$-almost everywhere.
\end{corollary}

\begin{example}[Optimal tessellations for Wasserstein distances] Let $\mu$ and $\nu$ satisfy \eqref{eqn:conditionMarginals}.
\begin{enumerate}[label=(\alph*),ref=\alph*]
\item\textbf{Wasserstein-2 distance.}
Let $c(x,y)=|x-y|^2$. If $T$ is an optimal Monge map, then the optimal transport cells $T^{-1}(\{ x_i \})$ are the Laguerre cells (or power cells) $C_i(w)$ with weight vector $w = (\psi(x_1),\ldots,\psi(x_M))$, where $\psi:\Omega \to \R$ is an optimal Kantorovich potential for the dual transport problem \eqref{eq:OTDuality}.
\item\textbf{Wasserstein-1 distance.} Let $c(x,y)=|x-y|$. If $T$ is an optimal Monge map, then the optimal transport cells $T^{-1}(\{ x_i \})$ are the Apollonius cells $C_i(w)$ with weight vector $w = (\psi(x_1),\ldots,\psi(x_M))$, where $\psi$ is an optimal Kantorovich potential.
\end{enumerate}
\end{example}

\section{Semi-discrete unbalanced transport}\label{sec:semiDiscreteUnbalanced}
In this section we consider \emph{semi-discrete unbalanced transport}. That is, we study \eqref{eq:UnbalancedProblem} for the cases where $\mu$ is absolutely continuous with respect to the Lebesgue measure and $\nu$ is discrete, as stated in \eqref{eqn:conditionMarginals}, and we do not require that $\mu(\Omega) = \nu(\Omega)$.
Semi-discrete unbalanced transport models the situation where there is a mismatch between the capacity of a discrete resource $\nu$ and the demand of a population $\mu$.

\subsection{Tessellation formulation}\label{sec:Tessellation}

The main results of this Section are \cref{thm:UnbalancedTessellationDual,thm:UnbalancedOptimalityConditions}, which generalize \cref{thm:OTSemiDiscrete,thm:OTSemiDiscreteOptimality} to the unbalanced setting.
Furthermore, in \cref{cor:UnbalancedTessellationPrimal} we state a `primal' counterpart of \cref{thm:UnbalancedTessellationDual} which is somewhat pathological in the classical, balanced optimal transport setting, but quite natural in the unbalanced case.

The following result generalizes \cref{thm:OTSemiDiscrete} to unbalanced transport.
\begin{theorem}[Tessellation formulation for semi-discrete unbalanced transport]
	\label{thm:UnbalancedTessellationDual}
	Let the cost function $c$ be radial (see \cref{def:RadialCost}).
	Given $\mu,\nu\in\measp(\Omega)$ satisfying \eqref{eqn:conditionMarginals}, define $\G : \R^M \to (-\infty,\infty]$ by
	\begin{subequations}
	\label{eq:UnbalancedTessellationDual}
	\begin{gather}
		\label{eq:UnbalancedTessellationDualObjective}
		\G(w) = -\sum_{i=1}^M \left(
				\int_{\C_i(w)} \Floc^\ast\big(-c(x,x_i)+w_i\big)\,\d\mu(x)
				+\Floc^\ast(-w_i) \cdot m_i \right) + \Floc(0) \cdot \mu(\resid) \\
		\intertext{with the convention $\infty\cdot0=0$. Then the unbalanced optimal transport distance can be obtained via}
\label{eq:UnbalancedTessellationDualObjective_b}
		\W(\mu,\nu) = \sup\left\{ \G(w)\,\middle|\,w \in \R^M\right\}.
	\end{gather}
	This is a concave maximization problem.
	\end{subequations}
\end{theorem}

\begin{proof}
	In analogy to the Kantorovich duality \eqref{eq:OTDuality} for the classical optimal transport problem \eqref{eq:OTProblem} we make use of a corresponding duality result for the unbalanced transport problem \eqref{eq:UnbalancedProblem},
	\begin{multline*}
		\W(\mu,\nu) = \sup\left\{
			-\int_\Omega \Floc^\ast(-\phi(x))\,\d\mu(x)
			-\int_\Omega \Floc^\ast(-\psi(x))\,\d\nu(x)
			\,\middle|\, \phi, \psi \in \cont(\Omega), \right. \\
		\left. \vphantom{\int_\Omega}
			\phi(x)+\psi(y) \leq c(x,y) \,\forall\,(x,y) \in \Omega\times\Omega
		\right\}.
	\end{multline*}
	This follows from \cite[Thm.\,4.11 and Cor.\,4.12]{LieroMielkeSavare-HellingerKantorovich-2015a}, where the former establishes the duality formula with $\phi$ and $\psi$ ranging over all lower semi-continuous simple functions and the latter allows us to use continuous functions instead, exploiting the fact that $\Floc^\ast$ is continuous on $\R$ by \cref{lem:ConjugateProperties}\allowbreak\eqref{item:FlocAstFinite}.
	Analogously to the proof of \cref{thm:OTSemiDiscrete} (dual tessellation formulation)
    we now parameterize the function $\psi$ on the set $\{x_i\}_{i=1}^M$ by a vector $w \in \R^M$, $w_i=\psi(x_i)$, and obtain
	\begin{multline}
		\label{eq:UnbalancedDuality}
		\W(\mu,\nu) = \sup\left\{
			-\int_\Omega \Floc^\ast(-\phi(x))\,\d\mu(x)
			-\sum_{i=1}^Mm_i\Floc^\ast(-w_i)
			\,\middle|\, \phi \in \cont(\Omega), w\in\R^M,\right. \\
		\left. \vphantom{\sum_{i=1}^M}
			\phi(x)+w_i \leq c(x,x_i) \,\forall\,x \in \Omega, i\in\{1,\ldots,M\}
		\right\}.
	\end{multline}
	Next, given $w\in\R^M$ we would like to optimize for $\phi$ as we did in \eqref{eq:PhiW}.
	Note though that $\phi_w=\infty$ on the residual set $\resid$, which in unbalanced transport may be nonnegligible despite finite $W(\mu,\nu)$.
	For this reason we argue by truncation:
	For given $w \in \R^M$ and $n \in \N$, the function $\phi=\phi_{w,n}$ with
	\begin{align*}
		\phi_{w,n}:\Omega\to\R,\quad\phi_{w,n}(x) = \min\{n, \min\{c(x,x_i)-w_i\,|\,i \in \{1,\ldots,M\} \} \}
	\end{align*}
	lies in $\cont(\Omega)$ and is feasible in \eqref{eq:UnbalancedDuality}.
 Moreover, for fixed $w$ the sequence $(\phi_{w,n})_{n \in \N}$ is a maximizing sequence for the maximization over $\phi$,
	and it converges pointwise monotonically to the function $\phi_{w}$ defined in \eqref{eq:PhiW}.
	By \cref{lem:ConjugateProperties}\allowbreak\eqref{item:FlocAstIncreasing} and \eqref{item:FlocAstFinite}, $z \mapsto -\Floc^\ast(-z)$ is real-valued, continuous, and increasing.
	Since $\phi_{w,n}(x) \geq -\max_i w_i$, $-\Floc^\ast(-\phi_{w,n}(x))$ is uniformly bounded from below with respect to $n$ and $x$. Consequently, also $\phi_w$ is bounded from below.
	Therefore the monotone convergence theorem implies that
	\begin{align*}
		\lim_{n \to \infty} -\int_\Omega \Floc^\ast(-\phi_{w,n}(x))\,\d\mu(x)
		= -\int_\Omega \Floc^\ast(-\phi_{w}(x))\,\d\mu(x),
	\end{align*}
	where by convention $\Floc^\ast(-\infty)=\lim_{z\to-\infty}\Floc^\ast(z)=-\Floc(0)$ (see \cref{lem:ConjugateProperties}).
	With this, \eqref{eq:UnbalancedDuality} finally becomes
	\begin{align}
		\label{eq:EUnbaPre}
		\W(\mu,\nu) = \sup \left\{
			-\int_\Omega \Floc^\ast(-\phi_{w}(x))\,\d\mu(x)-\sum_{i=1}^M m_i\Floc^\ast(-w_i)
			\,\middle|\,
			w \in \R^M \right\}\,.
	\end{align}
	Note that this objective is $>-\infty$ for all $w \in \R^M$ by the lower bound on $\phi_w$ and the monotonicity and real-valuedness of $\Floc^*$.
	Now we decompose the integration domain $\Omega$ into $\{\C_i(w)\}_{i=1}^M$ and $\resid$ (using once more $\mu \ll \Lebesgue$ and \cref{lem:CellNullOverlap}; it is only here that we use the radiality of the cost function $c$ so that the cells $C_i(w)$ are well-defined with negligible overlap). For $x \in \C_i(w)$ one finds
	$\phi_{w}(x)=c(x,x_i)-w_i$, while for $x \in \resid$ one obtains $\phi_{w}(x)=\infty$ and therefore $\Floc^\ast(-\phi_{w}(x))=-\Floc(0)$.
	This leads to expression \eqref{eq:UnbalancedTessellationDualObjective} and also implies that $\G(w)>-\infty$.
	
	For fixed $x \in \Omega$ the map $w \mapsto \phi_{w}(x)$ is concave (since it is a minimum over affine functions). Moreover, the map $z \mapsto -\Floc^\ast(-z)$ is concave and increasing (cf.~\cref{lem:ConjugateProperties}\allowbreak\eqref{item:FlocAstIncreasing}). Therefore, the objective function in \eqref{eq:EUnbaPre} and consequently $\G$ are concave functions of $w$.
\end{proof}

\begin{remark}[Finiteness of $\G$]\label{rem:finitenessG}
$\G\not\equiv\infty$ if and only if $\G$ is finite everywhere.
Indeed, the last summand in \eqref{eq:UnbalancedTessellationDualObjective} is independent of $w$, and $-\sum_{i=1}^M\Floc^\ast(-w_i)\,m_i$ is finite by \cref{lem:ConjugateProperties}\eqref{item:FlocAstFinite}.
Finally, if $\tilde\G(w)<\infty$ for some $w\in\R^M$, where $\tilde\G(w)=-\sum_{i=1}^M\int_{\C_i(w)}\Floc^\ast\big(-c(x,x_i)+w_i\big)\,\d\mu(x)$,
then for $g_i(x)\in\partial\Floc^\ast(-c(x,x_i)+w_i)$ and $\tilde g_i\in\partial\Floc^\ast(w_i)$ we have
\begin{multline*}
\tilde\G(\tilde w)
\leq-\sum_{i=1}^M\int_{\C_i(\tilde w)}\Floc^\ast\big(-c(x,x_i)+w_i\big)+g_i(x)(\tilde w_i-w_i)\,\d\mu(x)\\
\leq\tilde\G(w)+\sum_{i=1}^M|\tilde w_i-w_i|\int_{\C_i(w)}\tilde g_i\,\d\mu(x)
<\infty
\end{multline*}
for all $\tilde w\in\R^M$ by the convexity and monotonicity of $\Floc^\ast$ from \cref{lem:ConjugateProperties}\eqref{item:FlocAstNotNegInf}.

$\G\not\equiv\infty$ is for instance ensured by $\Floc(0)<\infty$ or by boundedness of $\cRad$ (which implies $\resid=\emptyset$).
Indeed, the former implies boundedness of $\Floc^\ast$ from below by \cref{lem:ConjugateProperties}\eqref{item:FlocAstBound}, while the latter implies uniform boundedness of $-c(x,x_i)+w_i$ from below
so that in either case the integrand of $\tilde\G$ is uniformly bounded from below.
\end{remark}

The following result generalizes the optimality conditions of \cref{thm:OTSemiDiscreteOptimality} to unbalanced transport.
\begin{theorem}[Optimality conditions]
\label{thm:UnbalancedOptimalityConditions}
	Let $\gamma \in \measp(\Omega \times \Omega)$, $w \in \R^M$, and set $\rho = \pushforward{\pi_1}{\gamma}$.
	If $\W(\mu,\nu) < \infty$ and $\gamma$ and $w$ are optimal for $\W(\mu,\nu)$ in \eqref{eq:UnbalancedProblem} and \eqref{eq:UnbalancedTessellationDual}, respectively, then
	\begin{subequations}
	\label{eq:UnbalancedOptimalityConditions}
	\begin{gather}
		\label{eq:UnbalancedOptimalityConditionsGamma}
		\gamma = \sum_{i=1}^M \rho \restr{\C_i(w)} \otimes \delta_{x_i}, \\
		\label{eq:UnbalancedOptimalityConditionsDensity}
		\RadNik{\rho}{\mu}(x) \in \partial \Floc^\ast(-c(x,x_i)+w_i)
			\tn{ for $\mu$-a.e.~} x \in \C_i(w), \quad
		\RadNik{\rho}{\mu}(x) = 0 \tn{ for } x \in \resid, \\
		\label{eq:UnbalancedOptimalityConditionsCells}
		\tfrac{\rho(\C_i(w))}{m_i} \in \partial \Floc^\ast(-w_i)\tn{ for }i \in \{1,\ldots,M\}.
	\end{gather}
	\end{subequations}
	Conversely, if $\gamma$ and $w$ satisfy \eqref{eq:UnbalancedOptimalityConditions}, then they are optimal in \eqref{eq:UnbalancedProblem} and \eqref{eq:UnbalancedTessellationDual}, respectively.
\end{theorem}
\begin{proof}
	Let $\gamma \in \measp(\Omega \times \Omega)$ be such that $\E(\gamma)$ in \eqref{eq:UnbalancedEnergy} is finite. This implies that $\gamma$ can be written as $\gamma=\sum_{i=1}^M \gamma_i \otimes \delta_{x_i}$ for $\gamma_i \in \measp(\Omega)$, $\sum_{i=1}^M \gamma_i = \pushforward{\pi_1}{\gamma}= \rho \ll \mu$ and $\rho(\resid)=0$.
	(Note that the same holds true if \eqref{eq:UnbalancedOptimalityConditions} is assumed instead of $\E(\gamma)<\infty$.)
	We obtain
	\begin{multline*}
		\E(\gamma) = \int_{\Omega \times \Omega} c\,\d\gamma
			+ \Fint(\rho|\mu)
			+ \Fint(\pushforward{\pi_2}{\gamma}|\nu) \\
		= \sum_{i=1}^M \int_{\Omega \setminus \resid} \c(x,x_i)\,\d\gamma_i(x)
			+ \int_{\Omega \setminus \resid} \Floc\big(\RadNik{\rho}{\mu}(x)\big)\,\d\mu(x)
			+ \Floc(0) \cdot \mu(\resid)
			+ \sum_{i=1}^M \Floc\big(\tfrac{\gamma_i(\Omega)}{m_i}\big) \cdot m_i
	\end{multline*}
	so that the duality gap between the primal and dual formulations \eqref{eq:UnbalancedProblem} and \eqref{eq:UnbalancedTessellationDual} reads
	\begin{multline*}
	 \E(\gamma)-\G(w)= \sum_{i=1}^M \int_{\Omega \setminus \resid} \c(x,x_i)\,\d\gamma_i(x)
	 + \int_{\Omega \setminus \resid} \left[
	 	\Floc\big(\RadNik{\rho}{\mu}(x)\big)+\Floc^\ast(-\phi_{w}(x))
	 	\right] \d\mu(x) \\
	 + \sum_{i=1}^M \left(
	 	\Floc\big(\tfrac{\gamma_i(\Omega)}{m_i}\big)
	 	+\Floc^\ast(-w_i)
	 	\right) \cdot m_i\,.
	 \end{multline*}
	Using the Fenchel--Young inequality, which states that $\Floc(s)+\Floc^\ast(z) \geq s \cdot z$ with equality if and only if $z \in \partial \Floc(s)$ or equivalently $s \in \partial \Floc^\ast(z)$ \cite[Prop.\,13.13 and Thm.\,16.23]{BauschkeCombettes11}, we obtain the lower bound
	\begin{align*}
	 \E(\gamma)-\G(w)&\geq\sum_{i=1}^M \int_{\Omega \setminus \resid} \c(x,x_i)\,\d\gamma_i(x)
	 - \int_{\Omega \setminus \resid} \phi_{w}(x) \, \d\rho(x)
	 -\sum_{i=1}^M w_i \cdot \gamma_i(\Omega)\\
	 &= \sum_{i=1}^M \int_{\Omega \setminus \resid}
		 \left[ \c(x,x_i)-w_i-\phi_{w}(x) \right] \d\gamma_i(x)\geq0\,,
	\end{align*}
        where the first inequality is an equality if and only if $\RadNik{\rho}{\mu}(x) \in \partial \Floc^\ast(-\phi_{w}(x))$ for $\mu$-almost every $x \in \Omega \setminus \resid$ and $\tfrac{\gamma_i(\Omega)}{m_i} \in \partial \Floc^\ast(-w_i)$ for $i=1,\ldots,M$,
        and where the second inequality is an equality if and only if $\spt\gamma_i\subset\C_i(w)$ and thus $\gamma_i=\rho\restr\C_i(w)$ for $i=1,\ldots,M$.
        As a consequence, we have $\E(\gamma)-\G(w)=0$ if and only if \eqref{eq:UnbalancedOptimalityConditions} holds.
	
	Now let $\W(\mu,\nu)<\infty$ and $\gamma$ and $w$ be optimal in \eqref{eq:UnbalancedProblem} and \eqref{eq:UnbalancedTessellationDual} so that $\W(\mu,\nu)=\E(\gamma)=\G(w)<\infty$. Then necessarily $\E(\gamma)-\G(w)=0$ and so \eqref{eq:UnbalancedOptimalityConditions} holds.
	Conversely, if \eqref{eq:UnbalancedOptimalityConditions} holds, then if $\E(\gamma)<\infty$ or $\G(w)<\infty$ (so that the difference $\E(\gamma)-\G(w)$ is well-defined), the above argument shows that $\E(\gamma)-\G(w)=0$, which due to $\E(\gamma)\geq\W(\mu,\nu)\geq\G(w)$ implies $\W(\mu,\nu)=\E(\gamma)=\G(w)$ and thus the optimality of $\gamma$ and $w$.
	If on the other hand $\E(\gamma)=\G(w)=\infty$, then $\W(\mu,\nu)=\infty$ so that $\gamma$ and $w$ are trivially optimal.
\end{proof}

\begin{corollary}[Uniqueness of coupling]
	\label{cor:UnbalancedOptimalityConditionsII}
	Let $\W(\mu,\nu)<\infty$ and $w$ be optimal for \eqref{eq:UnbalancedTessellationDual}. Then the unique minimizer $\gamma$ for \eqref{eq:UnbalancedProblem} is given by \eqref{eq:UnbalancedOptimalityConditionsGamma}, where $\rho$ is uniquely determined by \eqref{eq:UnbalancedOptimalityConditionsDensity} and automatically satisfies \eqref{eq:UnbalancedOptimalityConditionsCells}.
\end{corollary}
\begin{proof}
	We first show that \eqref{eq:UnbalancedOptimalityConditionsDensity} fully specifies $\rho$.
	Let $S$ be the set where $\partial \Floc^\ast$ is not a singleton.
	By convexity, $S$ is countable.
	In analogy to \cref{lem:CellNullOverlap}, for any $s \in S$ the set $\{ x \in \R\,|\,-\c(x,x_i)+w_i=s\}$ is Lebesgue negligible. Since $S$ is countable, the set $\{ x \in \R\,|\,-\c(x,x_i)+w_i \in S \}$ is Lebesgue-negligible and thus also $\mu$-negligible. Consequently, $\RadNik{\rho}{\mu}$ is uniquely defined by \eqref{eq:UnbalancedOptimalityConditionsDensity} on $\Omega$ up to a $\mu$-negligible set.
	
	For $\W(\mu,\nu)<\infty$, conditions \eqref{eq:UnbalancedOptimalityConditions} are necessary and must therefore be satisfied by any minimizer $\gamma$ (which exists by \cref{thm:UnbalancedProblemExistence}).
	Therefore, as $\rho$ is uniquely determined by \eqref{eq:UnbalancedOptimalityConditionsDensity}, so is $\gamma$ by \eqref{eq:UnbalancedOptimalityConditionsGamma}.
	Optimality of $\gamma$ and $w$ ensures that \eqref{eq:UnbalancedOptimalityConditionsCells} also holds.
\end{proof}

To gain some intuition we will illustrate the previous results with numerical examples in the next section.
Here we just spell out consistency with the balanced transport setting.
\begin{remark}[Balanced transport]
	\label{exp:BalancedTransportOptimality}
	For classical optimal transport with $\Floc=\iota_{\{1\}}$ (such as the Wasserstein-2 distance from \cref{exm:models}\allowbreak\eqref{item:ModelOT}) one obtains $-\Floc^\ast(-z)=z$.
	Then \eqref{eq:UnbalancedTessellationDual} becomes \eqref{eq:SemiDiscreteUnconstrainedDual} (and finiteness of $\WOT(\mu,\nu)$ implies that $\mu(\resid)=0$).
	Furthermore, with $\partial \Floc^\ast(z)=1$ for all $z$, equation \eqref{eq:UnbalancedOptimalityConditionsDensity} implies $\rho=\mu \restr{(\Omega \setminus \resid)}=\mu$. Then \eqref{eq:UnbalancedOptimalityConditionsGamma} and \eqref{eq:UnbalancedOptimalityConditionsCells} become \eqref{eq:SemiDiscreteOptimalGamma}.
\end{remark}

From the derivation of \eqref{eq:UnbalancedTessellationDual} we learned that it can be interpreted as a variant of the dual problem to \eqref{eq:UnbalancedProblem}, where one of the dual variables is parametrized by $w$.
Given the form of primal optimizers $\gamma$ according to \cref{thm:UnbalancedOptimalityConditions}, we can formulate a corresponding variant of the primal problem.
\begin{corollary}[Primal tessellation formulation of semi-discrete unbalanced transport]
\label{cor:UnbalancedTessellationPrimal}
Assume $W(\mu,\nu)<\infty$ and that optimizers of the unbalanced primal and dual problems \eqref{eq:UnbalancedProblem} and \eqref{eq:UnbalancedTessellationDual} exist. Then
\begin{multline}
	\label{eq:UnbalancedTessellationPrimal}
	\W(\mu,\nu) = \min \left\{
		\sum_{i=1}^M \int_{\C_i(w)} \c(x,x_i)\,\d \rho(x) + \Fint(\rho|\mu) \right. \\
		\left. + \sum_{i=1}^M \Floc\big(\tfrac{\rho(\C_i(w))}{m_i}\big) \cdot m_i
		\,\middle|\,
		w \in \R^M,\,\rho \in \measp(\Omega),\, \rho \restr{\resid}=0 \right\}.
\end{multline}
If $\gamma$ and $w$ are optimal in \eqref{eq:UnbalancedProblem} and \eqref{eq:UnbalancedTessellationDual}, respectively, then $w$ and $\rho = \pushforward{\pi_1}{\gamma}$ are optimal in \eqref{eq:UnbalancedTessellationPrimal}.
Conversely, if $w$ and $\rho$ are optimal in \eqref{eq:UnbalancedTessellationPrimal}, then \eqref{eq:UnbalancedOptimalityConditionsGamma} defines an optimal $\gamma$ for \eqref{eq:UnbalancedProblem}.
\end{corollary}
\begin{proof}
For any $w \in \R^M$ and $\rho \in \measp(\Omega)$ with $\rho \restr{\resid}=0$, the objective function in \eqref{eq:UnbalancedTessellationPrimal} is equal to $\E(\gamma)$ for $\gamma=\sum_{i=1}^M \rho \restr{\C_i(w)} \otimes \delta_{x_i}$.
Therefore, minimizing \eqref{eq:UnbalancedTessellationPrimal} corresponds to minimizing $\E$ over a particular subset of $\measp(\Omega \times \Omega)$, which implies that the right-hand side of \eqref{eq:UnbalancedTessellationPrimal} is no smaller than $\W(\mu,\nu)$.
Now, if $\gamma$ and $w$ are a pair of optimizers for \eqref{eq:UnbalancedProblem} and \eqref{eq:UnbalancedTessellationDual}, then by \eqref{eq:UnbalancedOptimalityConditions}, the objective function in \eqref{eq:UnbalancedTessellationPrimal} for $w$ and $\rho=\pushforward{\pi_1}{\gamma}$ becomes $\E(\gamma)=\W(\mu,\nu)$ so that the right-hand side of \eqref{eq:UnbalancedTessellationPrimal} actually equals $\W(\mu,\nu)$ and $w$ and $\rho$ are minimizers of \eqref{eq:UnbalancedTessellationPrimal}.

Conversely, if $w$ and $\rho$ minimize \eqref{eq:UnbalancedTessellationPrimal}, the induced $\gamma$ must minimize $\E$.
\end{proof}
\begin{remark}[Optimality of dual variable]\label{rem:optimalityDual}
	The converse conclusion that optimal $w$ in \eqref{eq:UnbalancedTessellationPrimal} are optimal in \eqref{eq:UnbalancedTessellationDual} is in general not true.
	Indeed, \eqref{eq:UnbalancedTessellationPrimal} only depends on $w$ via the cells $\{\C_i(w)\}_{i=1}^M$ and therefore is invariant with respect to adding the same constant to all components of $w$, which does not change the cells. In the balanced case, when $\Floc^*(z)=z$, then \eqref{eq:UnbalancedTessellationDual} is also invariant under such transformations (as long as $\mu(\Omega)=\sum_{i=1}^M m_i$). However, for general $\Floc$, $\Floc^*$ is nonlinear and the objective function of \eqref{eq:UnbalancedTessellationDual} is no longer invariant.
	
	Similarly, when the support of the optimal $\rho$ in \eqref{eq:UnbalancedTessellationPrimal} is bounded away from the boundary of some $\C_i(w)$ (see \cref{fig:SemiDisc_model_comparison}, right), then slightly changing the corresponding $w_i$ will not affect the value of \eqref{eq:UnbalancedTessellationPrimal}, whereas \eqref{eq:UnbalancedTessellationDual} will usually not exhibit this invariance.
	
	Finally, if $c(x,x_i)$ becomes infinite for sufficiently small $d(x,x_i)$, then there exists an isolated cell $\C_i(w)$ that is strictly bounded away from any other cell (see \cref{fig:SemiDisc_HK_base}, right). In that case, none of the cells $\{\C_j(w)\}_{j=1}^M$ depend on $w_i$, and so neither does \eqref{eq:UnbalancedTessellationPrimal}.
	However, the objective function of \eqref{eq:UnbalancedTessellationDual} in general still depends on $w_i$ via $\Floc^\ast$.
\end{remark}

\begin{remark}[Primal tessellation formulation for classical optimal transport]
	For classical optimal transport with $\Floc=\iota_{\{1\}}$, the term $\Fint(\rho|\mu)$ in \eqref{eq:UnbalancedTessellationPrimal} is finite (and zero) if and only if $\rho=\mu$. Likewise, $\sum_{i=1}^M \Floc\big(\tfrac{\rho(\C_i(w))}{m_i}\big) \cdot m_i$ is finite (and zero) if and only if $\rho(\C_i(w))=m_i$.
	These are the optimality conditions given in \cref{thm:OTSemiDiscreteOptimality}.
	Thus, the objective function in \eqref{eq:UnbalancedTessellationPrimal} is finite only where it is optimal, making it somewhat pathological.
\end{remark}

Even though \eqref{eq:UnbalancedTessellationPrimal} is less pathological for more general unbalanced transport problems, we focus on \eqref{eq:UnbalancedTessellationDual} for numerical optimization.

\subsection{Numerical examples and different models}
\label{sec:UnbalancedNumerics}
Depending on the choice of the cost function $\c$ and the marginal discrepancy $\Fint$, the semi-discrete unbalanced transport problem exhibits several qualitatively different regimes which we will illustrate in this section. The discussion will be complemented with numerical examples.

Problem \eqref{eq:UnbalancedTessellationDual} is an unconstrained, finite-dimensional maximization problem over a concave objective.
For simplicity, throughout this section we shall assume that the cost $\c$ is radial and $\Floc^\ast$ is differentiable or equivalently $\Floc$ is strictly convex (those assumptions are satisfied for the models from \cref{exm:models}).
This allows us to derive the objective function gradient in \cref{thm:UnbalancedTessellationGradient} and to treat the optimization problem with methods of smooth (as opposed to nonsmooth) optimization. A simple discretization scheme is given in \cref{rem:Discretization}. The resulting discrete problem is solved with an L-BFGS quasi-Newton method \cite{Nocedal-LBFGS-TOMS-97}. As stated in \cref{rem:PDGap}, the quality of the obtained solution can easily be verified via the primal-dual gap between \eqref{eq:UnbalancedTessellationDual} and \eqref{eq:UnbalancedTessellationPrimal}.
The special case of balanced optimal transport is discussed in \cref{exp:BalancedTransportGradient}.
Afterwards we provide numerical illustrations for several examples of different unbalanced models.

To calculate the gradient of $\G$ we make use of the following lemma.
\begin{lemma}[Derivative of integral functionals]\label{thm:derivativeIntegral}
Let $f:\Omega\times\R^M\to\R$ be uniformly Lipschitz in its second argument, and let $\mu\in\measp(\Omega)$ and $u\in\R^M$ be such that $\R^M \ni \tilde u \mapsto f(x,\tilde u)$ is differentiable at $\tilde u=u$ for $\mu$-almost all $x\in\Omega$.
Define $\mathcal{H}:\R^M \to \R$ by $\mathcal H(\tilde u) = \int_\Omega f(x,\tilde u)\,\d\mu(x)$. If $|\mathcal H(u)|<\infty$, then $\mathcal{H}$ is differentiable at $\tilde u=u$ with
\begin{equation*}
\frac{\partial \mathcal{H}}{\partial\tilde u}(u)=\int_\Omega \frac{\partial f}{\partial\tilde u}(x,u)\,\d\mu(x)\,.
\end{equation*}
\end{lemma}
\begin{proof}
This is similar to the proof of Theorem 2.27 in \cite{Folland}, which however only covers the case where $M=1$ and $f(x,\cdot)$ is differentiable for all $x \in \Omega$.
We show that the directional derivative of $\mathcal H$ in an arbitrary direction $\hat u\in\R^M$ exists and is of the desired form.
Indeed, let $L>0$ be the Lipschitz constant of $f$ in its second argument.
By assumption there exists $S\subset\Omega$ Lebesgue-negligible such that $f(x,\cdot)$ is differentiable at $u$ for all $x\in\Omega\setminus S$.
Now for $t\neq0$,
\begin{equation*}
\frac{\mathcal H(u+t\hat u)-\mathcal H(u)}t
=\int_{\Omega\setminus S}\frac{f(x,u+t\hat u)-f(x,u)}t\,\d\mu(x)\,.
\end{equation*}
Since the integrand is bounded in absolute value by $L\|\hat u\|$ and since it converges pointwise to $\frac{\partial f}{\partial u}(x,u)\cdot\hat u$ as $t\to0$, by the Dominated Convergence Theorem we have
\begin{equation*}
\lim_{t\to0}\frac{\mathcal H(u+t\hat u)-\mathcal H(u)}t
=\int_{\Omega\setminus S}\frac{\partial f}{\partial\tilde u}(x,u)\cdot\hat u\,\d\mu(x)
=\int_{\Omega}\frac{\partial f}{\partial\tilde u}(x,u) \,\d\mu(x)\cdot\hat u\,.
\end{equation*}
The arbitrariness of $\hat u$, the linearity of the directional derivative, and the Lipschitz continuity of $\mathcal H$ imply that $\mathcal{H}$ is differentiable and has the desired form.
\end{proof}

Note that one may also replace the Lipschitz condition by a convexity or concavity condition,
in which case dominated convergence would be replaced with monotone convergence;
such a procedure would allow one to treat different examples in the later application \cref{lem:gradJ}.

\begin{theorem}[Gradient of dual tessellation formulation]
	\label{thm:UnbalancedTessellationGradient}
	If $\Floc$ is strictly convex and $\G$ from \cref{thm:UnbalancedTessellationDual} is finite (e.g., if $\Floc(0)$ is finite or $\cRad$ is bounded, see \cref{rem:finitenessG})
then $\G$ is differentiable with
	\begin{align}
		\label{eq:UnbalancedTessellationGradient}
		\frac{\partial \G}{\partial w_i}(w) = (\Floc^\ast)'(-w_i)\cdot m_i
			- \int_{\C_i(w)} (\Floc^\ast)'(-c(x,x_i)+w_i)\,\d\mu(x)\,.
	\end{align}
\end{theorem}
\begin{proof}
Define
\[
f(x,w)=\min\{-\Floc^\ast(-\c(x,x_j)+w_j)\,|\,j=1,\ldots,M\}.
\]
Since $\Floc^\ast$ is increasing by \cref{lem:ConjugateProperties}\allowbreak\eqref{item:FlocAstIncreasing}, then $f(x,w) = -\Floc^\ast(-\c(x,x_i)+w_i)$ for $x \in \C_i(w)$. If $\cRad$ is bounded, the residual set $\resid$ is empty; otherwise we have $f(x,w)=-\Floc^\ast(-\infty)=\Floc(0)$ for $x \in \resid$.
Therefore
\[
\G(w)=\int_\Omega f(x,w)\,\d\mu(x)-\sum_{i=1}^M\Floc^\ast(-w_i)\cdot m_i.
\]
By \cref{rem:finitenessG}, $\G(w)$ is finite for all $w\in\R^M$.
Thus, there exists $\tilde\Omega\subset\Omega$ with $\mu(\Omega\setminus\tilde\Omega)=0$ such that $f(x,w)$ is finite for all $x\in\tilde\Omega$ and $w\in\R^M$.
Now consider the function $\Omega\times\R \ni (x,v) \mapsto f_i(x,v)=-\Floc^\ast(-\c(x,x_i)+v)$, where $i \in \{1 , \ldots ,M\}$. 
Furthermore, by \cref{lem:ConjugateProperties}\allowbreak\eqref{item:FlocAstDiffble} the strict convexity of $\Floc$ implies continuous differentiability of its conjugate $\Floc^\ast$ so that $f_i(x,\cdot)$ is differentiable for any $x\in\Omega$ with $f_i(x,\cdot)$ finite.
Moreover, since $\Floc^\ast$ is convex and increasing, ${\partial f_i}/{\partial v}$ is nonpositive and decreasing so that $f_i(x,\cdot)$ (if finite) is Lipschitz on $(-\infty,\omega]$ for any $\omega\in\R$ with Lipschitz constant $L\leq-\frac{\partial f_i}{\partial v}(x,\omega)\leq(\Floc^\ast)'(\omega)$.
Consequently, $(-\infty,\omega]^M \ni \hat{w} \mapsto f(x,\hat{w})$ is Lipschitz with constant $\sqrt M(\Floc^\ast)'(\omega)$ for all $x\in\tilde\Omega$ and differentiable in $\hat{w}$ for all $x\in\tilde\Omega\setminus S$, where $S=\bigcup_{i=1}^M\partial\C_i(w)$ is Lebesgue-negligible and thus also $\mu$-negligible.
Thus, by the previous \namecref{thm:derivativeIntegral}, $\G$ is differentiable with
\begin{equation*}
\frac{\partial \G(w)}{\partial w_i}
= (\Floc^\ast)'(-w_i)\cdot m_i
+ \int_{\Omega} \frac{\partial f}{\partial w_i}(x,w)\,\d\mu(x)\,,
\end{equation*}
where $\frac{\partial f}{\partial w_i}(x,w)=-(\Floc^\ast)'(-c(x,x_i)+w_i)$ for $\mu$-almost all $x\in\C_i(w)$ and $\frac{\partial f}{\partial w_i}(x,w)=0$ for $\mu$-almost all $x\notin\C_i(w)$.
\end{proof}

\begin{remark}[Balanced transport]
	\label{exp:BalancedTransportGradient}
	For classical optimal transport with $\Floc=\iota_{\{1\}}$ as in \cref{exp:BalancedTransportOptimality}, \cref{thm:UnbalancedTessellationGradient} reduces to well-known results. In particular, \eqref{eq:UnbalancedTessellationGradient} becomes
	\begin{align}
		\label{eq:SemiDiscOTGrad}
		\frac{\partial \G(w)}{\partial w_i} = m_i-\mu(\C_i(w))\,.
	\end{align}
	For more details we refer, for example, to 
 \cite[Thm.~1.1]{KitagawaMerigotThibert} or \cite[Thm.~40]{MerigotThibertOT}. For marginals $\mu=\tilde\mu\Lebesgue$ with $\tilde\mu\in\cont(\Omega)$ the Hessian
	\begin{align}
		\label{eq:SemiDiscOTHess}
		\frac{\partial^2 \G(w)}{\partial w_i \partial w_j} =- \frac{\partial \mu(\C_i(w))}{\partial w_j}
  	\end{align}
        can also be computed explicitly in terms of face integrals (see, for instance, \cite[Thm.~1.3]{KitagawaMerigotThibert} and \cite[Thm.~45]{MerigotThibertOT}). Therefore  \eqref{eq:SemiDiscreteUnconstrainedDual} lends itself to efficient numerical optimization \cite{AHA98,MultiscaleTransport2011,KitagawaMerigotThibert,LevySemiDiscrete2015}.
	For special cost functions, most prominently for the squared Euclidean distance, the gradient \eqref{eq:SemiDiscOTGrad} and Hessian \eqref{eq:SemiDiscOTHess} can be evaluated numerically efficiently and with high precision, allowing the application of Newton's method \cite{KitagawaMerigotThibert}.

The semi-discrete unbalanced problem \eqref{eq:UnbalancedTessellationDual} is more complicated due to the influence of the marginal fidelity $\Fint$ and since we are often interested in non-standard cost functions such as $\cHK$. Generalizing the above methods for balanced transport to the unbalanced case is therefore beyond the scope of this article.
\end{remark}

\begin{remark}[Discretization]
	\label{rem:Discretization}
	Problem \eqref{eq:UnbalancedTessellationDual} is already finite-dimensional. We must however evaluate the integrals over $\C_i(w)$.
	For classical optimal transport and special cost functions $c$, these integrals can be evaluated essentially in closed form (see \cref{exp:BalancedTransportGradient}).
	For simplicity, in this section we approximate $(\Omega,\mu)$ with Dirac masses on a fine Cartesian grid.
The cells $\{\C_i(w)\}_{i=1}^M$ are approximated using brute force by computing $c(x,x_i)-w_i$ for each point $x$ in the Cartesian grid for each $i \in \{1,\ldots,M\}$. Points $x$ on the common boundaries of several cells $\{\C_i(w)\}_{i=1}^M$ are arbitrarily assigned to one of those cells. (An efficient GPU-implementation of this brute force method can be found in \cite{BFRSB24}.)
Note that for the special cost $c(x,y)=|x-y|^2$, the Laguerre diagram $\{\C_i(w)\}_{i=1}^M$ can be computed exactly, up to machine precision, and much more efficiently using, e.g., the lifting method \cite[Sec.~6.2.2]{AurenhammerKleinLee}, which has complexity $\mathcal{O}(M \log M)$ in $\R^2$ and $\mathcal{O}(M^2)$ in $\R^3$.
Our discretization yields an approximation of $\G(w)$ from \eqref{eq:UnbalancedTessellationDualObjective} and of $\nabla \G(w)$ from \eqref{eq:UnbalancedTessellationGradient}, as required for the quasi-Newton method.
	In the numerical examples below we use $\Omega=[0,L]^2$ for some $L>0$ and approximate it by a regular Cartesian grid with $1000$ points along each dimension.
\end{remark}

\begin{remark}[Primal-dual gap]
	\label{rem:PDGap}
	The sub-optimality of any vector $w \in \R^M$ for \eqref{eq:UnbalancedTessellationDual} can be bounded by the primal-dual gap between \eqref{eq:UnbalancedTessellationDualObjective} and the objective of \eqref{eq:UnbalancedTessellationPrimal}. We avoid the remaining optimization over $\rho$ in \eqref{eq:UnbalancedTessellationPrimal} by generating a feasible candidate via \eqref{eq:UnbalancedOptimalityConditionsDensity}. \Cref{cor:UnbalancedOptimalityConditionsII,cor:UnbalancedTessellationPrimal} guarantee that the primal-dual gap vanishes for optimal $w$.
\end{remark}

In the remainder of the section we illustrate semi-discrete unbalanced transport by numerical examples.
In particular, we showcase qualitative differences between different models as well as phenomena due to model-inherent length scales, which do not occur in classical, balanced transport.

\begin{figure}[hbt]
	\centering
	\includegraphics{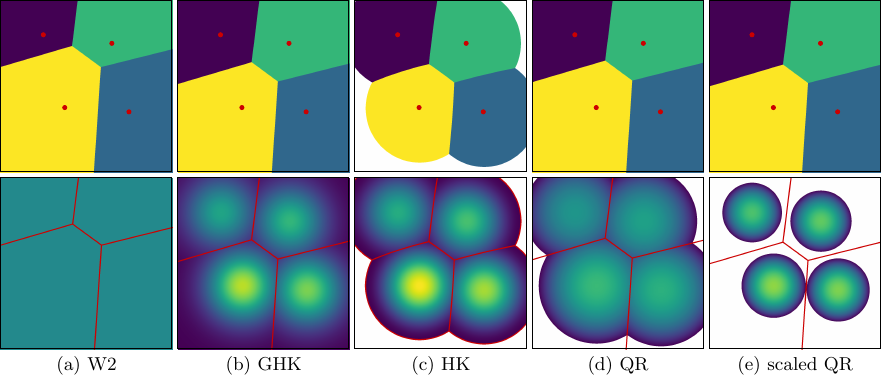}
	\caption{Semi-discrete transport between the Lebesgue measure on $\Omega=[0,L]^2$, $L=5$ and a discrete measure with $M=4$ Dirac masses of locations $(x_1,x_2,x_3,x_4)=L \cdot ((0.375,0.375),(0.75,0.35),(0.65,0.75),(0.25,0.8))$ and weights $(m_1,m_2,m_3,m_4)=|\Omega|\cdot(0.38,0.29,0.19,0.14)$.
	\textit{Top row:} optimal cells $\{\C_i(w)\}_{i=1}^M$; the residual set $\resid$ is represented by white; location of the discrete points $(x_1,\ldots,x_M)$ is indicated with red dots. %
	\textit{Bottom row:} optimal marginal $\rho$ (identical colour scale in all figures; regions with $\RadNik{\rho}{\mu}(x)=0$ are white for emphasis) and boundaries of cells $\{\C_i(w)\}_{i=1}^M$ (red) are shown for models \eqref{item:ModelOT}--\eqref{item:ModelQuadratic} from \cref{exm:models,exm:modelComparison}.
	Figure (e) shows the same model as \eqref{item:ModelQuadratic}, only with $\c(x,y)=[d(x,y)/2]^2$ instead of $\c(x,y)=d(x,y)^2$; on some cells $\spt \rho$ is now strictly bounded away from the boundaries of $\C_i(w)$. %
	}
	\label{fig:SemiDisc_model_comparison}
\end{figure}

\begin{example}[Comparison of unbalanced transport models]\label{exm:modelComparison}
The structure of the optimal unbalanced coupling $\gamma$ in \eqref{eq:UnbalancedProblem} and its first marginal $\rho = \pushforward{\pi_1}{\gamma}$ can vary substantially, depending on the choices for $\c$ and $\Floc$.
Below we discuss the models from \cref{exm:models} with a corresponding numerical illustration in \cref{fig:SemiDisc_model_comparison}.
\begin{enumerate}[label=(\alph*),ref=\alph*]
	\item
	\textbf{Standard Wasserstein-2 distance (W2, \cref{fig:SemiDisc_model_comparison}(\ref{item:ModelOT})).}
	Since this is an instance of balanced transport, necessarily we have $\rho=\mu$. Furthermore, the cells $\{\C_i(w)\}_{i=1}^M$ are standard, polygonal Laguerre cells, and $\resid=\emptyset$.
	\item
	\textbf{Gaussian Hellinger--Kantorovich distance (GHK, \cref{fig:SemiDisc_model_comparison}(\ref{item:ModelGHK})).}
	The cells are still standard polygonal Laguerre cells with $\resid=\emptyset$. This time, however, we usually have $\rho \neq \mu$. Nevertheless, we find $\spt \rho = \spt \mu$ since \eqref{eq:UnbalancedOptimalityConditionsDensity} with $(\FlocKL^\ast)'(z)=e^z>0$ implies $\RadNik\rho\mu>0$. This behaviour essentially originates from the infinite slope of $\FlocKL$ in $0$.
	Since $\c(x,y)=d(x,y)^2$, the density $\RadNik\rho\mu$ is piecewise Gaussian.
	\item
	\textbf{Hellinger--Kantorovich distance (HK, \cref{fig:SemiDisc_model_comparison}(\ref{item:ModelHK})).}
	This time, the generalized Laguerre cells have curved boundaries, and also $\resid$ is in general no longer empty, as $\cHK(x,y)=+\infty$ for $d(x,y)\geq \tfrac{\pi}{2}$.
	Thus, $\rho=0$ on $\resid$ by \eqref{eq:UnbalancedOptimalityConditionsDensity}, independent of $\mu$. However, similarly to \eqref{item:ModelGHK} we have $\RadNik{\rho}{\mu}(x)>0$ on the complement of $\resid$, the union of all generalized Laguerre cells.
	\item
	\textbf{Quadratic regularization (QR, \cref{fig:SemiDisc_model_comparison}(\ref{item:ModelQuadratic})-(e)).}
	Since $\c(x,y)=d(x,y)^2$, once more the cells are polygonal Laguerre cells and $\resid=\emptyset$. However, \eqref{eq:UnbalancedOptimalityConditionsDensity} together with $(\Floc^\ast)'(z)=0$ for $z \leq -2$ implies $\RadNik{\rho}{\mu}(x)=0$ whenever $\phi_{w}(x)= \min\big\{ c(x,x_i)-w_i\,|\,i =1,\ldots,M \big\}\geq 2$, even on $\Omega \setminus \resid$. Intuitively, this is possible since $\Floc$ and its right derivative are finite at $z=0$ so that, for large transport costs $c(x,x_i)$, mass removal may be more profitable than transport.
\end{enumerate}
We emphasize that the reasons for $\RadNik{\rho}{\mu}(x)=0$ between models \eqref{item:ModelHK} and \eqref{item:ModelQuadratic} are different: In the Hellinger--Kantorovich case, $c(x,x_i)=\infty$ for $x \in \resid$ prohibits any transport. In the quadratic case, despite finite transport cost and $\resid=\emptyset$, it may still be cheaper to remove and create mass via the fidelity $\Floc$, due to its behaviour at $z=0$.
Also the slope at which $\RadNik\rho\mu$ approaches zero is different for both models, as can be seen in the one-dimensional slice visualized in \cref{fig:SemiDisc_model_comparison_slice}.
\end{example}

\begin{figure}[H]
	\centering
	\includegraphics{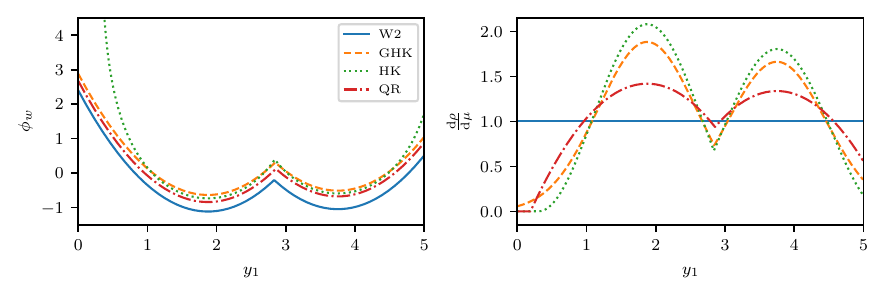}
	\caption{One-dimensional slices of computational results from \cref{fig:SemiDisc_model_comparison} along $[0,L]\times\{0.375\,L\}$ with $L=5$. %
	\textit{Left:} $\phi_w$ for optimal $w \in \R^M$. For models \eqref{item:ModelOT}, \eqref{item:ModelGHK}, and \eqref{item:ModelQuadratic}, $\phi_{w}$ is piecewise quadratic; for \eqref{item:ModelHK} the profile is determined by $\cHK$ and $\phi_w=\infty$ on $\resid \neq \emptyset$.
	\textit{Right:} Optimal density $\RadNik{\rho}{\mu}$, where $\RadNik{\rho}{\mu}=(\Floc^\ast)'(-\phi_{w})$ on $\Omega \setminus \resid$ and $0$ elsewhere by \eqref{eq:UnbalancedOptimalityConditionsDensity}. For \eqref{item:ModelOT} the density is constant, for \eqref{item:ModelGHK} it is piecewise Gaussian, for \eqref{item:ModelHK} it is piecewise given by $\cos(d(y,x_i))^2$ on $\Omega \setminus \resid$ and $0$ on $\resid$, and for \eqref{item:ModelQuadratic} it is given by truncated paraboloids.}
	\label{fig:SemiDisc_model_comparison_slice}
\end{figure}

\begin{example}[Varying transport length scales]
\label{ex:LengthScales}
As illustrated in the previous comparison of different models, unbalanced transport models typically have an intrinsic length scale which determines how far mass is optimally transported.
Varying this length scale for fixed $\mu$ and $\nu$ changes the behaviour of the semi-discrete transport.
For illustration we concentrate on the Hellinger--Kantorovich distance and vary its intrinsic length scale by replacing $\c(x,y)=\cHK(x,y)$ with
\begin{equation*}
\c(x,y)=\cHK^\varepsilon(x,y)=\cHK(\tfrac x\varepsilon,\tfrac y\varepsilon) = \begin{cases} -2\log\big[\cos\big(d(x,y)/\varepsilon\big)\big] & \tn{if } d(x,y) < \tfrac{\pi}{2}\varepsilon, \\ \infty & \tn{otherwise,} \end{cases}
\end{equation*}
that is, we set
\begin{multline*}
\HK^\varepsilon(\mu,\nu)^2=
\\
\inf\left\{\int_{\Omega\times\Omega}\cHK^\varepsilon\d\gamma
+\int_\Omega\FlocKL\left(\RadNik{\pushforward{\pi_1}{\gamma}}\mu\right)\d\mu
+\int_\Omega\FlocKL\left(\RadNik{\pushforward{\pi_2}{\gamma}}\nu\right)\d\nu
\,\middle|\,\gamma \in \measp(\Omega \times \Omega)\right\}\,.
\end{multline*}
Note that this is equivalent to rescaling the domain $\Omega$ by the factor $\frac1\varepsilon$ and simultaneously replacing the measures $\mu$ and $\nu$ by their pushforwards under $x\mapsto\frac x\varepsilon$.

For large $\varepsilon$, transport becomes very cheap relative to mass changes and thus asymptotically, as $\varepsilon \to \infty$, one recovers the Wasserstein-2 distance:
$\lim_{\varepsilon\to\infty}\varepsilon\HK^\varepsilon(\mu,\nu)=W_2(\mu,\nu)$  by \cite[Thm.\,7.24]{LieroMielkeSavare-HellingerKantorovich-2015a}. In particular the distance diverges when $\mu(\Omega) \neq \nu(\Omega)$.
Conversely, as $\varepsilon \searrow 0$, transport becomes increasingly expensive and mass change is preferred. Asymptotically one obtains $\lim_{\varepsilon \searrow 0} \HK^\varepsilon(\mu,\nu)=\Hell(\mu,\nu)$ \cite[Thm.\,7.22]{LieroMielkeSavare-HellingerKantorovich-2015a}, where $\Hell$ denotes the Hellinger distance
\begin{align*}
	\Hell(\mu,\nu)^2 = \int_{\Omega} \left(\sqrt{\RadNik{\mu}{\sigma}}-\sqrt{\RadNik{\nu}{\sigma}}\right)^2\d\sigma
\end{align*}
for $\sigma \in \measp(\Omega)$ an arbitrary reference measure with $\mu, \nu \ll \sigma$ (for instance $|\mu|+|\nu|$ with $|\cdot|$ indicating the total variation measure). By positive one-homogeneity of the function $(m_1,m_2) \mapsto (\sqrt{m_1}-\sqrt{m_2})^2$ the value of $\Hell(\mu,\nu)$ does not depend on the choice of $\sigma$.
In our semi-discrete setting, $\mu$ and $\nu$ are always mutually singular so that $\Hell(\mu,\nu)^2 = \mu(\Omega) + \nu(\Omega)$.

\begin{figure}[H]
	\centering
	\includegraphics{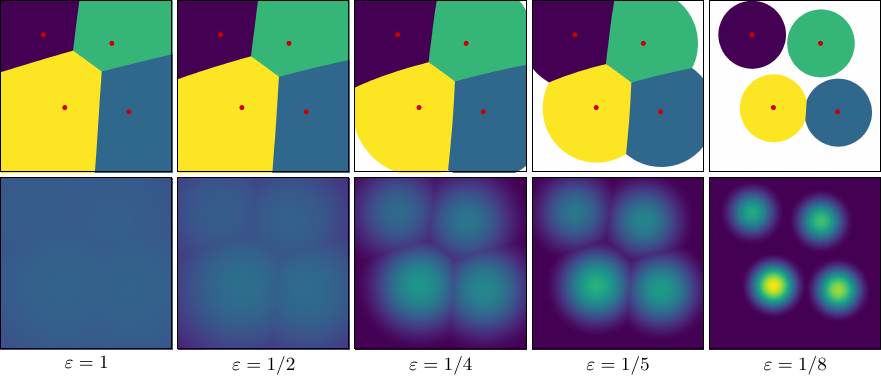}
	\caption{Semi-discrete Hellinger--Kantorovich transport on $\Omega=[0,1]^2$ (using the same values for $x_i/L$ and $m_i/|\Omega|$ as in \cref{fig:SemiDisc_model_comparison}) for different length scales $\varepsilon$. %
	\textit{Top row:} optimal cells $\{\C_i(w)\}_{i=1}^M$; the residual set $\resid$ is represented by white; location of the discrete points $(x_1,\ldots,x_M)$ is indicated with red dots. %
	\textit{Bottom row:} optimal marginal $\rho$ (using the same colour scale for all images).
	For large $\varepsilon$ the behaviour is similar to that of the standard semi-discrete Wasserstein-2 distance. As $\varepsilon$ decreases, the effects of unbalanced transport become increasingly prominent.}
	\label{fig:SemiDisc_HK_base}
\end{figure}

\begin{figure}[H]
	\centering
	\includegraphics{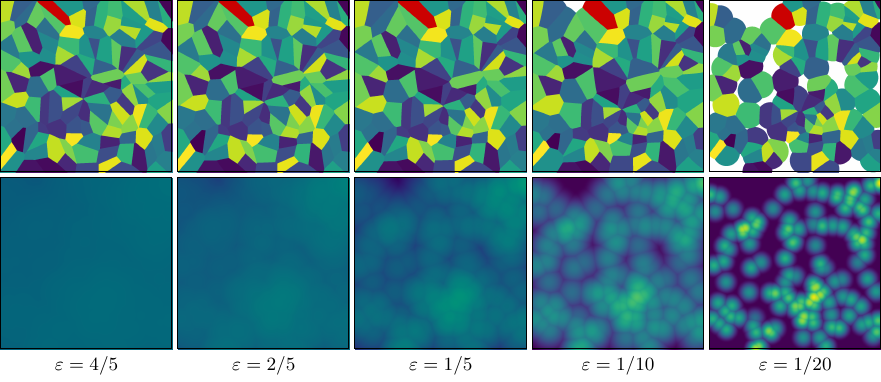}
	\caption{Semi-discrete Hellinger--Kantorovich distance on $\Omega=[0,1]^2$ for different length scales $\varepsilon$, as in \cref{fig:SemiDisc_HK_base}, but for $M=128$. The evolution of one cell $\C_i(w)$ for fixed $i$ is highlighted in red (\textit{top row}).
	For large $\varepsilon$, $\C_i(w)$ is essentially the standard Wasserstein-2 Laguerre cell, not necessarily containing $x_i$. For small $\varepsilon$, $\C_i(w)$ becomes (a fraction of) the open ball $B_{\varepsilon\pi/2}(x_i)$.}
	\label{fig:SemiDisc_HK_large}
\end{figure}

\Cref{fig:SemiDisc_HK_base} illustrates the optimal cells $\{\C_i(w)\}_{i=1}^M$ and marginal densities $\rho=\pushforward{\pi_1}{\gamma}$ between the uniform volume measure $\mu=\Lebesgue$ on $\Omega=[0,1]^2$ and a discrete measure $\nu=\sum_{i=1}^M m_i\,\delta_{x_i}$ for $M=4$, using different values of the intrinsic length scale $\varepsilon$
(the same experiment with $M=128$ discrete points is shown in \cref{fig:SemiDisc_HK_large}).
As expected, for large $\varepsilon$ the cells $\{\C_i(w)\}_{i=1}^M$ look very similar to standard, polygonal Laguerre cells for the squared Euclidean distance $\c(x,y)=d(x,y)^2$, and the residual set $\resid$ is empty. The optimal $\rho$ is essentially equal to $\mu$, as dictated by balanced transport.
As $\varepsilon$ decreases, the boundaries between the cells become curved. Eventually $\resid$ becomes non-empty, and finally the cells start to decompose into disjoint discs. In accordance, the density of the optimal marginal $\rho$ is given on each cell $\C_i(w)$ by $\cos(d(x,x_i)/\varepsilon)^2 \cdot e^{w_i}$.
The interpolatory behaviour of $\HK^\varepsilon$ between the Wasserstein-2 distance $W_2$ and the Hellinger distance $\Hell$ for $\varepsilon\to\infty$ and $\varepsilon \searrow 0$ is numerically verified in \cref{fig:SemiDisc_HK_energies}.

\begin{figure}[H]
	\centering
	\includegraphics{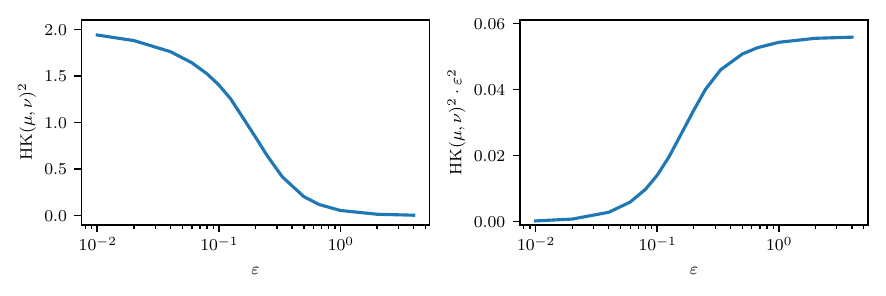}
	\caption{$\HK^\varepsilon(\mu,\nu)^2$ for different length scales $\varepsilon$ for the setup from \cref{fig:SemiDisc_HK_base}. %
	\textit{Left:} as $\varepsilon \searrow 0$, $\HK^\varepsilon(\mu,\nu)^2$ tends to $\Hell(\mu,\nu)^2=2$. %
	\textit{Right:} as $\varepsilon \to \infty$, $\varepsilon^2\HK^\varepsilon(\mu,\nu)^2$ tends to $W_2(\mu,\nu)^2$.}
	\label{fig:SemiDisc_HK_energies}
\end{figure}
\end{example}

\clearpage

\section{Unbalanced quantization}\label{sec:quantization}
In this section we study the unbalanced quantization problem:
we aim to approximate a given Lebesgue-continuous measure $\mu\in\measp(\Omega)$ by a discrete, quantized measure $\nu=\sum_{i=1}^Mm_i \cdot \delta_{x_i}$ with at most $M\in\N$ Dirac masses,
where the unbalanced transport cost serves as a measure of approximation quality.
To be precise, we consider the optimization problem
\begin{equation}\label{eqn:quantization}
\min\left\{\W(\mu,\nu)\,\middle|\,\nu=\sum_{i=1}^M m_i \delta_{x_i},\,x_1,\ldots,x_M\in\Omega,\,m_1,\ldots,m_M\geq0\right\}\,.
\end{equation}
Existence of minimizers will be established in \Cref{thm:quantizationTessellation}.
Applications include optimal location problems (economic planning), information theory (vector quantization) and particle methods for PDEs (approximation of continuous initial data by particles).
We first characterize optimal particle configurations in terms of Voronoi diagrams, then consider a corresponding numerical scheme, and finally prove the optimal energy scaling of the quantization problem in terms of $M$ for the case $d=2$.
The procedure essentially follows the one known for classical optimal transport; the important fact is that the Voronoi tessellation structure survives if mass changes are allowed.

The quantization cost will essentially depend on the function $-\Floc^*\circ(-\cRad)$ for $\cRad$ from \cref{def:RadialCost} (see \cref{thm:quantizationTessellation}), which is why we briefly list a few of its relevant properties. We will mention below when these properties are assumed or used.
\begin{lemma}[Properties of $-\Floc^*\circ(-\cRad)$]\label{thm:propertiesQuantizationIntegrand}
Let $\cRad$ define a radial cost and $\Floc$ a marginal discrepancy, and consider the following conditions on them,
\begin{gather}
\label{eq:QuantizationAsymptoticAssumptions_a}
\Floc(0)<\infty
\qquad\tn{or}\qquad
\cRad(s)<\infty\; \forall \, s \in [0,\infty),\\
\label{eq:QuantizationAsymptoticAssumptions_b}
\Floc(0) > 0\,,\\
\label{eq:QuantizationFlocAssumption_a}
\Floc(1)=0\,,\\
\label{eq:ellDiverges}
\lim_{s\to\infty}\cRad(s)=\infty\,,\\
\label{eq:ellLipschitz}
\ell\text{ is Lipschitz on }[0,\diam(\Omega)] \text{ or on } [0,\infty) \text{ when $\Omega$ is unbounded}.
\end{gather}
Under these conditions, $-\Floc^*\circ(-\cRad)$ satisfies the following properties.
\begin{enumerate}[label=\rm(P\arabic*),ref=P\arabic*]
\item\label{enm:quantIntMonotone}
$-\Floc^*\circ(-\cRad)$ is nondecreasing (increasing before it potentially reaches its maximum),
\item\label{enm:quantIntContinuous}
$-\Floc^*\circ(-\cRad)$ is continuous,
\item\label{enm:quantIntFinite}
$-\Floc^*\circ(-\cRad)<\infty$ on $(0,\infty)$,
\item\label{enm:quantIntPositive}
$-\Floc^*\circ(-\cRad)>0$ on $(0,\infty)$,
\item\label{enm:quantIntZero}
$-\Floc^*\circ(-\cRad)(0)=0$,
\item\label{enm:quantIntLimit}
$\lim_{s\to\infty}-\Floc^*\circ(-\cRad)(s)=\Floc(0)$,
\item\label{enm:quantIntLipschitz}
$-\Floc^*\circ(-\cRad)$ is Lipschitz on $[0,\diam(\Omega)]$ or on $[0,\infty)$ when $\Omega$ is unbounded.
\end{enumerate}
More specifically, \eqref{enm:quantIntMonotone} and \eqref{enm:quantIntContinuous} hold by properties of $\Floc$ and $\cRad$ according to \cref{def:RadialCost,def:MarginalDiscrepancy},
\eqref{eq:QuantizationAsymptoticAssumptions_a}$\Rightarrow$\eqref{enm:quantIntFinite},
\eqref{eq:QuantizationAsymptoticAssumptions_b}$\Rightarrow$\eqref{enm:quantIntPositive},
\eqref{eq:QuantizationFlocAssumption_a}$\Rightarrow$\eqref{enm:quantIntZero},
\eqref{eq:ellDiverges}$\Rightarrow$\eqref{enm:quantIntLimit}, and
\eqref{eq:ellLipschitz}$\Rightarrow$\eqref{enm:quantIntLipschitz}.
\end{lemma}

We leave the proof as a straightforward exercise, it essentially being a direct consequence of \cref{lem:ConjugateProperties,def:RadialCost,def:MarginalDiscrepancy}.
Note that the conditions are not necessary for the properties \eqref{enm:quantIntMonotone}-\eqref{enm:quantIntLipschitz} to hold,
which is why in the remainder of the section we will solely refer to these properties rather than to conditions on $\Floc$ and $\cRad$.
However, the above conditions on $\Floc$ and $\cRad$ are natural:
\eqref{eq:QuantizationAsymptoticAssumptions_a} expresses that it is always possible to either completely remove or transport mass at any given location with finite cost ($W(\mu,\nu)$ may still be infinite, if $\Floc(0)=\infty$, $\sup_z \cRad(z)=\infty$, and $\mu$ has strong tails),
\eqref{eq:QuantizationAsymptoticAssumptions_b} expresses that complete mass removal has a positive cost,
\eqref{eq:QuantizationFlocAssumption_a} expresses that there is zero cost for not changing the mass,
\eqref{eq:ellDiverges} expresses that the transport cost becomes infinite for infinite distances,
and \eqref{eq:ellLipschitz} ensures that the transport cost does not increase superlinearly.
Of course \eqref{eq:ellLipschitz} is not necessary for \eqref{enm:quantIntLipschitz}, since the growth of $\ell$ can be compensated by a sufficiently slow increase of $-\Floc^*(-\cdot)$, which corresponds to $\Floc$ growing only slowly (or being bounded) near zero. Examples for this are the GHK and HK distances; see \cref{ex:TessellationModels}.
Throughout this section we will assume that zero mass change induces zero cost, i.e.\ \eqref{eq:QuantizationFlocAssumption_a}.
This is the natural choice for approximating $\mu$, as it implies a preference for $\pushforward{\pi_1}{\gamma}=\mu$ in the first marginal fidelity term  $\Fint$ of \eqref{eq:UnbalancedEnergy}.
Since $\Floc(z)\geq 0$ by \cref{def:MarginalDiscrepancy}, a consequence is
\begin{gather}
\label{eq:QuantizationFlocAssumption_b}
	0 \in \partial \Floc(1) \qquad \Leftrightarrow \qquad 1 \in \partial\Floc^\ast(0)
\qquad \text{and} \qquad \Floc^\ast(0)=0.
\end{gather}
Also note that for the quantization problem only the behaviour of $\Floc$ on $[0,1]$ is relevant
since by a simple comparison argument one can see that any minimizer
$\nu$ in \eqref{eqn:quantization}
and any associated coupling $\gamma$ in \cref{def:UnbalancedTransport} satisfy $\pushforward{\pi_2}{\gamma}=\nu$ and $\pushforward{\pi_1}{\gamma}\leq\mu$.

\subsection{Unbalanced quantization as a Voronoi tessellation problem}
\label{sec:Quantization}
The following theorem shows that the quantization problem can equivalently be formulated as an optimization of the points $x_1,\ldots,x_M$ with a functional depending on the Voronoi tessellation induced by $(x_1,\ldots,x_M)$.
\begin{theorem}[Tessellation formulation of quantization problem]\label{thm:quantizationTessellation}
For $\Floc$ satisfying \eqref{eq:QuantizationFlocAssumption_a}, the unbalanced quantization problem \eqref{eqn:quantization} is equivalent to the minimization problem
\begin{equation}\label{eqn:quantizationTessellation}
\min\left\{J(x_1,\ldots,x_M)\,\middle|\,x_1,\ldots,x_M\in\Omega\right\}
\end{equation}
where
\[
J(x_1,\ldots,x_M)=\sum_{i=1}^M\int_{V_i(x_1,\ldots,x_M)} -\Floc^\ast(-\c(x,x_i))\,\d\mu(x)
\]
and
where $V_i(x_1,\ldots,x_M) = \{ x \in \Omega\,|\, d(x,x_i) \leq d(x,x_j) \text{ for } j=1,\ldots,M \, \}$ is the Voronoi cell associated with $x_i$ and we adopt the convention $-\Floc^\ast(-\infty)=\Floc(0)$ and $\partial\Floc^\ast(-\infty) = \{0\}$ (cf.~\cref{lem:ConjugateProperties}).
Indeed, the minimum values coincide and, if $(x_1,\ldots,x_M)$ minimizes \eqref{eqn:quantizationTessellation} and the minimal value is finite, then $(x_1,\ldots,x_M,m_1,\ldots,m_M)$ minimizes \eqref{eqn:quantization} for
\begin{equation} \label{eqn:quantizationTessellationMasses}
m_i=\int_{V_i(x_1,\ldots,x_M)}\partial\Floc^\ast(-\c(x,x_i))\,\d\mu(x)\,,\quad i=1,\ldots,M.
\end{equation}
(By the proof of \cref{cor:UnbalancedOptimalityConditionsII} the subgradient $\partial\Floc^\ast(-\c(x,x_i))$ contains a unique element for $\mu$-almost every $x$ and so the $m_i$ are well defined.)
Furthermore, the optimal transport plan $\gamma$ associated with $\W(\mu,\nu)$ only transports mass from each Voronoi cell $V_i(x_1,\ldots,x_M)$ to the corresponding point $x_i$.
\end{theorem}

\begin{example}[Tessellation formulation for unbalanced transport models]
\label{ex:TessellationModels}
The cost functional in \eqref{eqn:quantizationTessellation} and the masses in \eqref{eqn:quantizationTessellationMasses} for the models from \cref{exm:models} are
\begin{align*}
&\text{W2}:&&\begin{cases}
\displaystyle
J=\sum_{i=1}^M\int_{V_i}d(x,x_i)^2\,\d\mu(x)\,,\\
\displaystyle
m_i=\mu(V_i)\,,
\end{cases}\\
&\text{GHK}:&&\begin{cases}
\displaystyle
J=\sum_{i=1}^M\int_{V_i}\left[ 1-e^{-d(x,x_i)^2} \right] \,\d\mu(x)\,,\\
\displaystyle
m_i=\int_{V_i}e^{-d(x,x_i)^2}\,\d\mu(x)\,,
\end{cases}\\
&\text{HK}:&&\begin{cases}
\displaystyle
J=\sum_{i=1}^M\int_{V_i}\sin^2\left(\min\left\{d(x,x_i),\tfrac{\pi}{2}\right\}\right)\,\d\mu(x)\,,\\
\displaystyle
m_i=\int_{V_i}\cos^2\left(\min\left\{d(x,x_i),\tfrac{\pi}{2}\right\}\right)\,\d\mu(x)\,,
\end{cases}\\
&\text{QR}:&&\begin{cases}
\displaystyle
J=\sum_{i=1}^M\int_{V_i\cap B_{\sqrt2}(x_i)}\left[ d(x,x_i)^2-\frac{d(x,x_i)^4}4 \right] \,\d\mu(x)
+\mu(V_i\setminus B_{\sqrt2}(x_i))\,,\\
\displaystyle
m_i=\int_{V_i}\max\left\{1-\frac{d(x,x_i)^2}2,0\right\}\,\d\mu(x).
\end{cases}
\end{align*}
\end{example}
\begin{remark}
An intuitive strategy for proving \cref{thm:quantizationTessellation} could be as follows. One starts from the primal tessellation formulation in \cref{cor:UnbalancedTessellationPrimal} and in addition minimizes over masses $(m_1,\ldots,m_M)$ and positions $(x_1,\ldots,x_M)$. By \eqref{eq:QuantizationFlocAssumption_a} we find that minimizing masses are given by $m_i=\rho(\C_i(w))$.
Next, only the transport term depends on the weights $w$, and since the cost $\c$ is a strictly increasing function of distance, the term is minimized for $w=0$, thus essentially reducing the generalized Laguerre cells $\C_i(w)$ into (truncated) Voronoi cells. Finally, the remaining minimization over $\rho$ can be handled with arguments from convex analysis, similar to those of \cref{thm:UnbalancedOptimalityConditions}, thus arriving at \eqref{eqn:quantizationTessellation}.
We give a shorter proof, using results from the dual tessellation formulation and its optimality conditions.
\end{remark}
\begin{proof}[Proof of \cref{thm:quantizationTessellation}]
First we consider minimization over the masses $(m_i)_i$ for fixed positions $(x_i)_i$.
Let $\nu=\sum_{i=1}^Mm_i \cdot \delta_{x_i}$ be any admissible measure for \eqref{eqn:quantization}.
From \eqref{eq:UnbalancedTessellationDualObjective_b} we find $\W(\mu,\nu) \geq \G(0)$ for any positions $x_1,\ldots,x_M$ and masses $m_1,\ldots,m_M$.
Note that $\G(0)$ does not depend on $m_1,\ldots,m_M$ since we assume $\Floc^\ast(0)=0$, \eqref{eq:QuantizationFlocAssumption_b}. Consider now the case where $\G(0)<\infty$ (and hence $\G$ is finite for all values of $w$, see \Cref{rem:finitenessG}).
We now show $\W(\mu,\nu)=\G(0)$ for a particular choice of $m_1,\ldots,m_M$, which therefore must be optimal (for given locations $x_1,\ldots,x_M$).
We first define $\rho$ via \eqref{eq:UnbalancedOptimalityConditionsDensity} and then $\gamma$ via \eqref{eq:UnbalancedOptimalityConditionsGamma} for $w=0$ ($\rho$ and $\gamma$ are fully determined, see \cref{cor:UnbalancedOptimalityConditionsII}). Furthermore, since $1 \in \partial \Floc^\ast(0)$ by \eqref{eq:QuantizationFlocAssumption_b}, equation \eqref{eq:UnbalancedOptimalityConditionsCells} is satisfied by the choice $m_i=\rho(\C_i(0))$.
(Note that had we chosen $\Floc(\hat z)=0$ instead of $\Floc(1)=0$, one would simply use $m_i=\rho(\C_i(0))/\hat z$ so that \eqref{eqn:quantizationTessellationMasses} would change by the factor $\hat z$.)
By \cref{thm:UnbalancedOptimalityConditions} (optimality conditions), $\gamma$ and $w$ are optimizers of $\E$ and $\G$ for these mass coefficients, which implies that $\W(\mu,\nu)=\G(0)$. Using $\Floc^\ast(0)=0$ from \eqref{eq:QuantizationFlocAssumption_b}, we have
\begin{align*}
	\min_{(m_1,\ldots,m_M)} \W(\mu,\nu) = \G(0) =
	-\sum_{i=1}^M
				\int_{\C_i(0)} \Floc^\ast\big(-c(x,x_i)\big)\,\d\mu(x)
				+ \Floc(0) \cdot \mu(\resid)\,.
\end{align*}
Since $\c(x,y)$ is a strictly increasing function of the distance $d(x,y)$, for $w=0$ we find $\C_i(0) \subset V_i(x_1,\ldots,x_M)$. With the convention $-\Floc^\ast(-\infty)=\Floc(0)$ (cf.~\cref{lem:ConjugateProperties}), the term $\Floc(0) \cdot \mu(\resid)$ becomes $\int_{\resid}-\Floc^\ast(-\phi_0(x))\,\d \mu(x)$, where $\phi_0$ was defined in equation \eqref{eq:PhiW}. Since $\mu \ll \Lebesgue$, integrating over $\resid$ and $\Omega \setminus \resid$ is equivalent to integrating over all Voronoi cells $\{V_i(x_1,\ldots,x_M)\}_{i=1}^M$, and we arrive at
\begin{align}
\label{eq:QuantTessellMassMin}
	\min_{(m_1,\ldots,m_M)} \W(\mu,\nu) = -\sum_{i=1}^M
				\int_{V_i(x_1,\ldots,x_M)} \Floc^\ast\big(-c(x,x_i)\big)\,\d\mu(x)\,,
\end{align}
which establishes equivalence between \eqref{eqn:quantization} and \eqref{eqn:quantizationTessellation}.

Finally, with $m_i=\rho(\C_i(0))$ and $\rho$ given by \eqref{eq:UnbalancedOptimalityConditionsDensity} one obtains \eqref{eqn:quantizationTessellationMasses}, where the integral runs over $\C_i(0)$ instead of $V_i(x_1,\ldots,x_M)$. If the minimum is finite, then either $\mu(\resid)=0$ or $\Floc(0)$ is finite, which implies the convention $(\Floc^\ast)'(-\infty)=0$ (cf.~\cref{lem:ConjugateProperties}\allowbreak\eqref{item:FlocZeroFinite}). In both cases we can extend the area of integration to $V_i(x_1,\ldots,x_M)$ without changing its value.
Equation\,\eqref{eq:UnbalancedOptimalityConditionsGamma} implies that mass is only transported from each Voronoi cell $V_i(x_1,\ldots,x_M) \supset \C_i(0)$ to the corresponding point $x_i$.

If on the other hand $\G(0)=\infty$, then by \eqref{eq:UnbalancedTessellationDualObjective_b} for all $\nu$ concentrated on the positions $(x_i)_i$ one has $W(\mu,\nu)=\infty$. This establishes the equality of the minimal values in both the finite and infinite cases.

Now we consider minimization over the positions $(x_i)_i$. We merely need to consider the case where the minimal value is finite, as otherwise any configuration is optimal. Let $((x_i^k)_{i=1}^M)_{k \in \N}$ be a minimizing sequence of points for the right-hand side of \eqref{eq:QuantTessellMassMin}. After selection of a subsequence and relabeling the order of the points, there will be an integer $N \in \{0,\ldots,M\}$ and points $(x_i)_{i=1}^N$ such that $\lim_k x_i^k=x_i$ for $i \in \{1,\ldots,N\}$ and $\lim_k |x_i^k|=\infty$ for $i>N$.
Let $(\rho^k)_{k \in \N}$ be the sequence induced via \eqref{eq:UnbalancedOptimalityConditionsDensity} from $((x_i^k)_i)_k$ (for $w=0$). Using non-negativity of $c$, convexity of $\Floc^\ast$, and the fact that $1 \in \partial\Floc^\ast(0)$, we get that $\rho^k \leq \mu$ for all $k$ and hence the sequence is tight, allowing extraction of a weak (i.e.~in duality with bounded continuous functions) cluster point $\rho$.
Let $m_i^k=\rho^k(V_i((x_j^k)_j))$ be the corresponding optimal masses, as above, and let $m_i=\rho(V_i((x_j)_{j=1}^N))$ be the masses induced by $\rho$ for $i \in \{1,\ldots,N\}$. Using weak convergence of $\rho^k \to \rho$ and absolute continuity of $\rho$ one finds that $m^k_i \to m_i$ for $i \in \{1,\ldots,N\}$ and $m_i^k \to 0$ for $i>N$. One then has that $\nu^k := \sum_{i=1}^M m_i^k \cdot \delta_{x_i^k}$ converges weakly to $\nu := \sum_{i=1}^N m_i \cdot \delta_{x_i}$. By \cite[Lemma 3.9]{LieroMielkeSavare-HellingerKantorovich-2015a}, $W$ is weakly lower-semicontinuous and therefore the points $(x_i)_{i={1}}^N$ are candidates for a minimizer in \eqref{eq:QuantTessellMassMin}. However, if $N<M$ they are too few. But since adding arbitrary points on the right-hand side of \eqref{eq:QuantTessellMassMin} will not increase the objective, and thus any such extension must yield a minimizer.
\end{proof}

\subsection{A numerical method: Lloyd's algorithm and quasi-Newton variant}
\label{sec:QuantizationNumerics}
Formulation \eqref{eqn:quantizationTessellation} has the advantage over \eqref{eqn:quantization} that it does not contain an inner minimization to find the optimal transport coupling.
Thus we aim to solve \eqref{eqn:quantizationTessellation} numerically. To this end we compute the gradient $\partial_{x_j} J$ (see also analogous derivatives for similar functionals as for instance in \cite{BourneRoper15}).

\begin{lemma}[Derivative of the cost functional $J$]
\label{lem:gradJ}
Let $\mu\in\measp(\Omega)$ be absolutely continuous, and let \eqref{enm:quantIntMonotone} and \eqref{enm:quantIntLipschitz} hold.
Then for $j=1,\ldots,M$,
\[
\partial_{x_j}J(x_1,\ldots,x_M)
=\int_{V_j(x_1,\ldots,x_M)} r(d(x,x_j))(x_j-x)\,\d\mu(x)
\]
where
\[
r(s)=\frac{[-\Floc^\ast\circ(-\cRad)]'(s)}s
\]
(note that $-\Floc^\ast\circ(-\cRad)$ is differentiable for almost every $s\in[0,\diam(\Omega)]$ so that $r$ and $r(d(\cdot,x_j))$ are well-defined almost everywhere).
\end{lemma}
\begin{example}[Cost derivative for unbalanced transport models]
For the models from \cref{exm:models} one can readily check
\begin{align*}
&\text{W2}:&&
r(s)=2\,,\\
&\text{GHK}:&&
r(s)=2e^{-s^2}\,,\\
&\text{HK}:&&
r(s)=\sin(2s)/s\text{ if }s\leq\tfrac\pi2\text{ and }0\text{ otherwise,}\\
&\text{QR}:&&
r(s)=\max\{2-s^2,0\}\,.
\end{align*}
Note that W2 only satisfies assumption \eqref{enm:quantIntLipschitz} on bounded domains.
\end{example}
\begin{proof}[Proof of \Cref{lem:gradJ}]
Note that $J(x_1,\ldots,x_M)=\int_\Omega f(x,(x_1,\ldots,x_M))\,\d\mu(x)$ with
\begin{equation*}
f(x,(x_1,\ldots,x_M))=\min\{-\Floc^\ast(-\cRad(d(x,x_i)))\,|\,i=1,\ldots,M\}\,
\end{equation*}
since $-\Floc^\ast\circ(-\ell)$ is nondecreasing by \eqref{enm:quantIntMonotone}.
By the Lipschitz assumption \eqref{enm:quantIntLipschitz} on $\Floc^\ast\circ(-\cRad)$, $f$ is Lipschitz in its second argument.
Furthermore, $\Floc^\ast\circ(-\cRad)$ is differentiable almost everywhere, and $d(x,x_i)$ is differentiable in its second argument for all $x\neq x_i$.
Therefore, $f$ is differentiable in its second argument at $(x_1,\ldots,x_M)$ for almost all $x\in\Omega$ (thus for $\mu$-almost all $x\in\Omega$) with
\begin{equation*}
\partial_{x_j}f(x,(x_1,\ldots,x_M))=\begin{cases}
r(d(x,x_j))(x_j-x)&\text{if }x\in V_j(x_1,\ldots,x_M),\\0&\text{otherwise.}
\end{cases}
\end{equation*}
\Cref{thm:derivativeIntegral} now implies the desired result.
\end{proof}

To find a minimizer of $J$ and thus a solution to the optimality condition $\partial_{x_j}J=0$ for $j=1,\ldots,M$, one can perform the following fixed point iteration associated with the optimality conditions,
\[
x_i^{(k+1)} =
\frac{\displaystyle \int_{V_i(x_1^{(k)},\ldots,x_M^{(k)})} xr(d(x_i^{(k)},x)) \, \d\mu(x)}
{\displaystyle \int_{V_i(x_1^{(k)},\ldots,x_M^{(k)})} r(d(x_i^{(k)},x)) \, \d\mu(x)}\,,
\quad i=1,\ldots,M,
\]
starting from some initialization $x_1^{(0)},\ldots,x_M^{(0)}\in\Omega$.
This iteration is well-defined as long as the denominator is nonzero, for instance if $\mu$ is strictly positive on $\Omega$ (recall that $-\Floc^\ast(-\cRad(s))$ is strictly increasing for small $s$ by \eqref{enm:quantIntMonotone}).
This is a generalisation of Lloyd's algorithm for computing Centroidal Voronoi Tessellations \cite{DuFaberGunzburger99}, which are critical points of the function
\[
\tilde J(x_1,\ldots,x_M) = \sum_{i=1}^M \int_{V_i(x_1,\ldots,x_M)} |x-x_i|^2 \, \d\mu(x)\,.
\]
Its convergence has been proven in a number of settings \cite{SabinGray86,EmelianenkoJuRand08,BourneRoper15} which also cover many possible choices for our $\mu$, $\c$, and $\Floc$.
Since the algorithm is based solely on the first variation one can expect linear convergence.
To achieve faster convergence one may use a quasi-Newton method for the minimization of $J$ instead, which seems particularly well-suited since the optimization is performed over a finite-dimensional space.

Our numerical implementation is performed in Matlab.
The integrals over a Voronoi cell $V_i(x_1,\ldots,x_M)$ are evaluated using Gaussian quadrature on the triangulation which is obtained by connecting each vertex of $V_i(x_1,\ldots,x_M)$ with $x_i$.
The Voronoi cells themselves are computed using the built-in function \texttt{voronoin}.
\Cref{fig:quantizationEnergyDecrease} shows a slightly faster convergence of the BFGS method compared to Lloyd's algorithm,
while \cref{fig:quantizationModelComparison} shows quantization results for the same models as in \cref{fig:SemiDisc_model_comparison}, resulting in different point distributions.
Similarly, \cref{fig:quantizationL} shows quantization results for the same input marginal $\mu$ and the Hellinger--Kantorovich model, but for varying length scales.

\begin{figure}[H]
\centering
\includegraphics[]{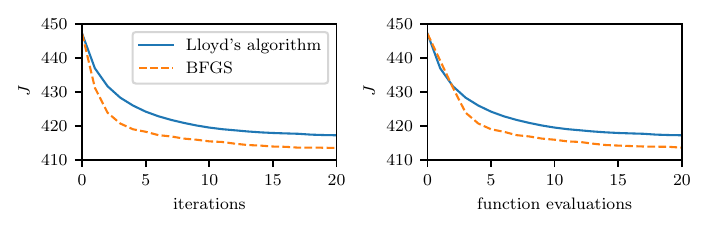}
\hfill
\includegraphics[]{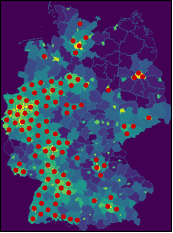}
\hfill
\caption[]{Quantization energy decrease of Lloyd's algorithm and a BFGS method versus number of iterations (\textit{left}) and function evaluations (\textit{centre}; for the BFGS method function evaluations differ from iterations due to additional evaluations in the stepsize control) for the example shown on the right. \textit{Right:} Input density $\mu$ and optimal locations $(x_1,\ldots,x_M)$ for $M=100$, where $\mu$ is population density in Germany 2015 (published by the Federal Statistical Office of Germany in the ``Regional Atlas'' \url{http://www.destatis.de/regionalatlas}).
The computations use the Hellinger--Kantorovich model.}
\label{fig:quantizationEnergyDecrease}
\end{figure}

\begin{figure}[H]
	\centering
	\includegraphics{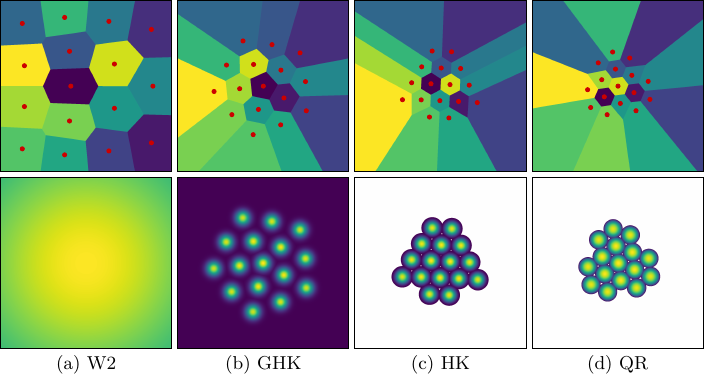}
	\caption{Quantization results for $\mu=(1+\exp(-\frac{|x|^2}{2(4\pi)^2})) \cdot \Lebesgue \restr\Omega$ and $\Omega=[-4\pi,4\pi]^2$, $M=16$ on the same models as in \cref{fig:SemiDisc_model_comparison}.
	\textit{Top row:} optimal locations $x_1,\ldots,x_M$ and Voronoi cells $\{V_i(x_1,\ldots,x_M)\}_{i=1}^M$. %
	\textit{Bottom row:} optimal marginal $\rho=\pushforward{\pi_1}{\gamma}$ (identical colour scale in all figures; regions with $\RadNik{\rho}{\mu}(x)=0$ are white for emphasis). For (a) we have $\rho=\mu$.
	}
	\label{fig:quantizationModelComparison}
\end{figure}

\begin{figure}[H]
	\centering
	\includegraphics{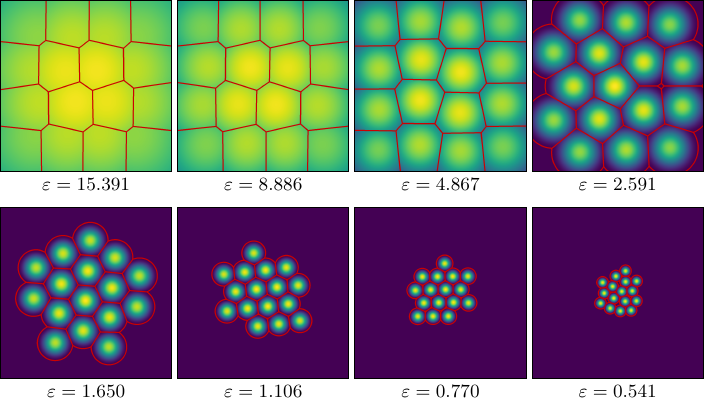}
	\caption{Quantization results for the Hellinger--Kantorovich model and different length scales, showing the optimal Laguerre cells $\C_i(0)$ (which coincide with the optimal Voronoi cells up to the set $\resid$ from \eqref{eq:Resid})
 and the optimal marginals $\rho=\pushforward{\pi_1}{\gamma}$ (same domain and $\mu$ as in \cref{fig:quantizationModelComparison}; identical colour scale in all figures).
	}
	\label{fig:quantizationL}
\end{figure}

\begin{figure}[H]
	\centering
	\includegraphics{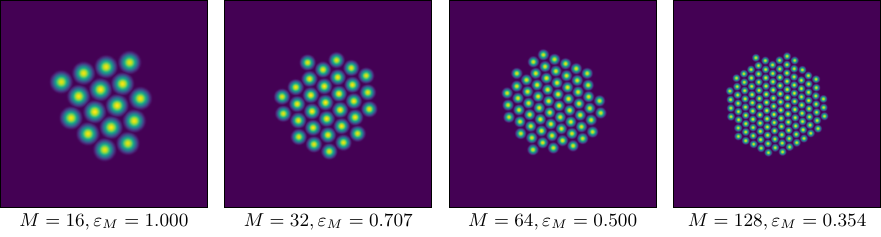}
	\caption{Quantization results for the Hellinger--Kantorovich model using different length scales and numbers of discrete points, with constant total point density $\varepsilon_M^2  M$.
	The optimal marginals $\rho=\pushforward{\pi_1}{\gamma}$ are shown (same domain and $\mu$ as in \cref{fig:quantizationModelComparison}; identical colour scale in all figures).
	}
	\label{fig:quantizationML}
\end{figure}

\subsection{Crystallization in two dimensions}
\label{sec:Crystallization}

In this section we consider the asymptotic behaviour of the unbalanced quantization problem in the limit of infinitely many points, $M \to \infty$, in two dimensions, $\Omega \subset \R^2$, in which case crystallization results from discrete geometry are available.

To simplify the exposition in this section we assume \eqref{enm:quantIntFinite} so that the unbalanced transport cost is always finite.
Additionally we assume \eqref{enm:quantIntPositive}, which simply ensures that the quantization problem is not trivially degenerate.
The situation without these conditions can in principle be treated similarly, but requires a number of technical case distinctions (such as whether the domain of $(-\Floc^\ast\circ(-\cRad))$ is open or closed).

As we increase $M$, the average distance between points of $\Omega$ and their nearest discrete point $x_i$ decreases so that the (balanced) transport cost from $\mu$ onto $\nu$ vanishes in the limit, whereas the cost for changing mass remains unchanged. Therefore, in the limit $M \to \infty$ the interplay of transport and mass change in \emph{unbalanced} transport would not be visible. To avoid this, we will rescale the metric on the domain $\Omega$ as $M$ grows and study the resulting different regimes, depending on the scaling.
Consequently, in this section we consider the scaled cost
\begin{equation}
\label{eq:QuantizationScaledCost}
J_{\varepsilon}^M(x_1,\ldots,x_M)
=\sum_{i=1}^M\int_{V_i(x_1,\ldots,x_M)}-\Floc^\ast\left(-\cRad\left(\tfrac{d(x,x_i)}{\varepsilon}\right)\right)\,\d\mu(x)
\end{equation}
for $M \in \N$, $\varepsilon \in (0,\infty)$.

We first prove a lower bound on the quantization cost $J_{\varepsilon}^M$ for the Lebesgue measure, which corresponds to a perfect triangular lattice. Then a corresponding upper bound is derived.
Finally, for $\mu$ with Lipschitz continuous Lebesgue density, we show that these two bounds imply that asymptotically a locally regular triangular lattice becomes an optimal quantization configuration, where the local density of points depends on the density of $\mu$.

\begin{theorem}[Lower bound for quantization of the Lebesgue measure]\label{thm:lowerBound}
Let $\Omega \subset \R^2$ be a convex polygon with at most six sides, let $\mu$ be the Lebesgue measure on $\Omega$, and let \eqref{enm:quantIntMonotone} hold.
A lower bound on \eqref{eq:QuantizationScaledCost} is given by
\begin{align} \label{eq:QuantizationLowerBound}
\min_{x_1,\ldots,x_M\in \mathbb{R}^2} J_{\varepsilon}^M(x_1,\ldots,x_M)
\geq M \int_{\Hex(|\Omega|/M)}-\Floc^\ast\left(-\cRad\left(\tfrac{d(x,0)}{\varepsilon}\right)\right) \, \d x\,,
\end{align}
where $\Hex(|\Omega|/M)$ is a regular hexagon of area $|\Omega|/M=\Lebesgue(\Omega)/M$ centred at the origin $0$.
\end{theorem}
\begin{remark}[Cost of the triangular lattice]\label{rem:costTriangularLattice}
Comparing with \cref{thm:quantizationTessellation},
the lower bound is exactly the unbalanced transportation cost $\W(\mu,\nu)$ from a regular triangular lattice $\nu$ of $M$ Dirac measures of mass
\[
m = \int_{\Hex(|\Omega|/M)} \partial\Floc^\ast\left(-\cRad\left(\tfrac{d(x,0)}{\varepsilon}\right)\right)\, \d x\,,
\]
whose Voronoi cells are translations of $\Hex(|\Omega|/M)$,
onto $\mu$ the Lebesgue measure on the union of these Voronoi cells.
\end{remark}
\begin{proof}[Proof of \cref{thm:lowerBound}]
Since $-\Floc^\ast(-\cRad(\cdot/\varepsilon))$ is increasing by \eqref{enm:quantIntMonotone}, for $x_1,\ldots,x_M\in\Omega$ we have
\begin{align*}
	J_\varepsilon^M(x_1,\ldots,x_M)
	& = \sum_{i=1}^M \int_{V_i(x_1,\ldots,x_M)}
	-\Floc^\ast\left(-\cRad\left(\tfrac{d(x,x_i)}{\varepsilon}\right)\right)\,\d x \\
	& =\int_\Omega\min_{i=1,\ldots,M}-\Floc^\ast
		\left(-\cRad\left(\tfrac{d(x,x_i)}{\varepsilon}\right)\right)
	\,\d x\,,
\end{align*}
and the result follows immediately from L.~Fejes T\'oth's Theorem on Sums of Moments \cite{Gruber99} (see also \cite{FejesToth72,MorganBolton02}).
\end{proof}

\begin{remark}[Degeneracy of minimizers]\label{rem:degeneracyMinimizers}
As opposed to the quantization problem for classical optimal transport,
the set of minimizers in the unbalanced transport case can exhibit strong degeneracies.
As an example, consider the case of Hellinger--Kantorovich transport with $M\ll4\,|\Omega|/(\pi^3 \varepsilon^2)$.
Let $x_1,\ldots,x_M$ be any arrangement of the point masses such that the balls $B_{\varepsilon\,\pi/2}(x_i)$ are pairwise disjoint and included in $\Omega$ (which necessarily implies $M\leq4\,|\Omega|/(\pi^3 \varepsilon^2)$).
Then $(x_1,\ldots,x_M)$ achieves the lower bound since
\begin{align*}
J_\varepsilon^M(x_1,\ldots,x_M)
&= \sum_{i=1}^M \int_{V_i(x_1,\ldots,x_M)}
-\Floc^\ast\left(-\cRad\left(\tfrac{d(x,x_i)}{\varepsilon}\right)\right)\,\d x \\
&=\sum_{i=1}^M \int_{V_i(x_1,\ldots,x_M)}\sin^2\left(\min\left\{d(x,x_i)/\varepsilon,\tfrac{\pi}{2}\right\}\right)\,\d x\\
&=|\Omega|-M\varepsilon^2\,\pi^3/4+\sum_{i=1}^M\int_{B_{\varepsilon\,\pi/2}(x_i)}\sin^2\left(d(x,x_i)/\varepsilon\right)\,\d x\\
&=M\int_{\Hex(|\Omega|/M)}-\Floc^\ast(-\cRad(d(x,0)/\varepsilon))\,\d x\,,
\end{align*}
where we used $\Hex(|\Omega|/M) \supset B_{\varepsilon\,\pi/2}(0)$.
Indeed, the energy does not discriminate between different solutions, because $-\Floc^\ast\circ(-\cRad)$ is constant for distances larger than $\frac{\epsilon\pi}2$.
\end{remark}

\begin{theorem}[Upper bound for quantization of the Lebesgue measure]\label{thm:upperBound}
Let $\Omega \subset \R^2$ be a square domain, let $\mu$ be the Lebesgue measure, and let \eqref{enm:quantIntMonotone} hold.
Let $x_1,\ldots,x_M$ be a regular triangular arrangement of points in the following sense:
Let $G \subset \R^2$ be a regular triangular lattice with lattice spacing $\sqrt{\frac{2\,|\Omega|}{\sqrt{3}\,M}}$, such that the corresponding Voronoi cells are regular hexagons with area $|\Omega|/M$ and side length $L=\sqrt{\frac{2\,|\Omega|}{3\sqrt{3} M}}$.
Let $\{x_1,\ldots,x_{\hat{M}}\} \subset G$ be those points for which the corresponding hexagon $\Hex_i$ is fully contained in $\Omega$.
Assume that $M$ is sufficiently large so that $\hat{M} \ge 1$.
If $\hat{M}<M$, pick $\{x_{\hat{M}+1},\ldots,x_M\}$ arbitrarily from $\Omega$.
Then
\begin{equation}
\label{eq:QuantizationUpperBound}
J_\varepsilon^M(x_1,\ldots,x_M)\leq M\!\!\int_{\Hex(\frac{|\Omega|}M)}\!\!-\Floc^\ast\left(-\cRad\left(\tfrac{d(x,0)}{\varepsilon}\right)\right) \d x
- |\partial\Omega|\sqrt{\tfrac{8|\Omega|}{3\sqrt3M}}\,\Floc^\ast\!\!\left(-\cRad\left(
\sqrt{\tfrac{C_Q\,|\Omega|}{\varepsilon^2M}}
\right)\right),
\end{equation}
where $C_Q = 2(2 \sqrt{2} + 1)^2 / (3 \sqrt{3})$, $|\partial \Omega|$ denotes the one-dimensional Hausdorff measure of $\partial \Omega$ and $|\Omega|=\Lebesgue(\Omega)$.
\end{theorem}
\begin{proof}
Let $S = \Omega \setminus \bigcup_{i=1}^{\hat{M}} \Hex_i$ be those points in $\Omega$ that are not covered by any hexagon $\Hex_i$. Note that all $x\in S$ lie no further away from $\partial\Omega$ than the diameter of a hexagon, $2L=\sqrt{\frac{8\,|\Omega|}{3\sqrt3M}}$.
Since $\Omega$ is convex we thus have $|S|\leq | \Omega \cap \bigcup_{x \in \partial \Omega} B_{2L}(x)| \leq 2\,L\,|\partial\Omega|$.
Likewise, any point $x\in S$ lies no further away from the union of all hexagons $H_i$ than $2 \sqrt{2}L$
and thereby no further away from $\{x_1,\ldots,x_M\}$ than $(2 \sqrt{2}+1)L$,
thus 
$\min_{i}d(x,x_i)\leq
(2 \sqrt{2}+1) L$.

Note that $V_i(x_1,\ldots,x_M) \setminus S \subseteq \Hex_i$ for $i=1,\ldots,\hat{M}$
and that $-\Floc^\ast\circ(-\cRad)$ is monotonously increasing so that we find
\begin{align*}
	J_\varepsilon^M(x_1,\ldots,x_M) & \leq \sum_{i=1}^{\hat{M}} \int_{\Hex_i} -\Floc^\ast\left(-\cRad\left(\tfrac{d(x,x_i)}{\varepsilon}\right)\right)\,\d x 
	+\int_S-\Floc^\ast\left(-\cRad\left(\tfrac{\min_id(x,x_i)}{\varepsilon}\right)\right)\,\d x \\
	& \leq M \cdot \int_{\Hex(|\Omega|/M)} -\Floc^\ast\left(-\cRad\left(\tfrac{d(x,0)}{\varepsilon}\right)\right)\,\d x
	-|S|\Floc^\ast\left(-\cRad\left(\tfrac{
(2 \sqrt{2}+1)
    L}{\varepsilon}\right)\right).
\end{align*}
Substituting the value of $L$ and the above bound for $|S|$ proves the claim.
\end{proof}

\begin{remark}[A priori estimate] \label{rem:APrioriEstimate}
From \eqref{enm:quantIntMonotone}
we also have the estimate
\begin{equation*}
\min J_\varepsilon^M
\leq\int_\Omega-\Floc^\ast(-\cRad(\diam(\Omega)/\varepsilon))\,\d\mu
\leq \mu(\Omega) \cdot (-\Floc^\ast(-\cRad(\diam(\Omega)/\varepsilon))),
\end{equation*}
whose right-hand side may be further bounded by the potentially infinite $\mu(\Omega)\Floc(0)=\W(\mu,0)$
(the latter bound is directly obtained by choosing $\nu=0$ as a quantization candidate in \eqref{eqn:quantization}).
\end{remark}

Let now $(\varepsilon_M)_{M \in \N}$ be a positive, decreasing sequence of scaling factors.
We use \cref{thm:lowerBound,thm:upperBound} to study the asymptotic quantization behaviour of the sequence of functionals $(J_{\varepsilon_M}^M)_M$ as $M \to \infty$ for a non-uniform mass distribution $\mu$ with Lipschitz Lebesgue density $m$.
We identify three different regimes, depending on the behaviour of the sequence $\varepsilon_M^2\,M$ (the quantity $\varepsilon_M^2\,M$ indicates something like the average point density).
A corresponding numerical illustration for the case of constant average point density is provided in \cref{fig:quantizationML}.

\begin{theorem}[Asymptotic quantization]\label{thm:asymptoticQuantization}
Let $\Omega\subset\R^2$ be a closed Lipschitz domain (a domain whose boundary is locally the graph of a Lipschitz function with the domain lying on one side) and $\mu=m \cdot (\Lebesgue \restr\Omega)$ for $\Lebesgue$ the Lebesgue measure and $m:\Omega\to[0,\infty)$ a Lipschitz-continuous density.
Let \eqref{enm:quantIntMonotone}-\eqref{enm:quantIntPositive} hold.
For any sequence $\varepsilon_1,\varepsilon_2,\ldots>0$ with $\varepsilon_M\searrow0$ as $M\to\infty$ the following holds:
\begin{enumerate}
\item\label{enm:manyMasses}
If $\displaystyle\lim_{M\to\infty}\!\!\varepsilon_M^2M\!=\!\infty$, then $\displaystyle\lim_{M\to\infty}\!\!\min J_{\varepsilon_M}^M=\!-\!\Floc^\ast(-\cRad(0))\!\cdot\!\mu(\Omega)$, which under \eqref{enm:quantIntZero} equals $0$.
\item\label{enm:fewMasses}
If $\displaystyle\lim_{M\to\infty}\!\!\varepsilon_M^2M\!=\!0$, then $\displaystyle\lim_{M\to\infty}\!\!\min J_{\varepsilon_M}^M=\mu(\Omega)\lim_{s\to\infty}-\Floc^\ast(-\cRad(s))$, which under \eqref{enm:quantIntLimit} equals $\mu(\Omega)\Floc(0)=\W(\mu,0)$.
\item\label{enm:normalMasses}
If $\displaystyle\lim_{M\to\infty}\!\!\varepsilon_M^2M\!=\!P\in(0,\infty)$, then
$$\displaystyle\lim_{M\to\infty}\!\!\min J_{\varepsilon_M}^M
=\left[\kappa\mapsto\int_\Omega \EDens^\ast(\kappa/m(x))\,\d\mu(x)\right]^\ast(P),
$$
where $\EDens:(-\infty,\infty)\to(0,\infty]$,
\begin{equation*}
\EDens(z)=z \cdot \int_{\Hex(1/z)}-\Floc^\ast(-\cRad(d(x,0)))\,\d x\, \quad \tn{for } z>0,
\end{equation*}
$\EDens(0)=\lim_{s\to\infty}-\Floc^\ast(-\cRad(s))$, and $\EDens(z)=\infty$ for $z<0$.
Furthermore, there exists a unique constant $\lambda<0$ and some measurable function $D:\Omega\to[0,\infty)$ such that
$$
\displaystyle\lim_{M\to\infty}\min J_{\varepsilon_M}^M
=\int_\Omega \EDens(D(x))\,\d\mu(x),
$$
and
\begin{equation}\label{eqn:optCond}
D(x)\in\partial \EDens^\ast(\lambda/m(x))
\text{ for almost all }x\in\Omega,
\qquad
P=\int_\Omega D(x)\,\d x
\end{equation}
(by convention, for $m(x)=0$ we set $D(x)=0$). That is, $D$ can be interpreted as (being proportional to) the spatially varying point density of the asymptotically optimal local triangular grid.
\end{enumerate}
\end{theorem}

\begin{remark}[Limit cases]
\Cref{thm:asymptoticQuantization}\allowbreak\eqref{enm:manyMasses} and \eqref{enm:fewMasses} can in fact be recovered as the special cases $P=\infty$ and $P=0$ of \cref{thm:asymptoticQuantization}\allowbreak\eqref{enm:normalMasses}
if we set $(\lambda,D)\equiv(0,\infty)$ or $(\lambda,D)\equiv(-\infty,0)$, respectively.
However, it is simpler to treat them separately.
\end{remark}

Before stating a few more remarks and proving the asymptotic result we analyse the function $B$, which represents the cell problem of quantizing a hexagon by a single Dirac mass.
\begin{lemma}[Properties of the cell problem]\label{lem:cellProblem}
Assume \eqref{enm:quantIntMonotone}-\eqref{enm:quantIntPositive}.
On $(0,\infty)$, the function $\EDens$ from \cref{thm:asymptoticQuantization} is finite, positive, decreasing, and convex with continuous derivative
\begin{equation*}
\EDens'(z)=\frac1z\left[\EDens(z)-\tfrac{1}{|\partial \Hex(1/z)|} \int_{\partial \Hex(1/z)}-\Floc^\ast(-\cRad(d(x,0)))\,\d x\right]\, =:G(z).
\end{equation*}
Further, $\EDens(z) \to \EDens(0)$ as $z \searrow 0$, $G(z)\to0$ as $z\to\infty$,
and there exists some $Z\geq0$ such that $G$ is constant on $(0,Z]$ and strictly increasing on $(Z,\infty)$,
where $Z>0$ if $-\Floc\circ(-\cRad)$ achieves a maximum.
With $r = \lim_{z \searrow Z} B'(z)\in[-\infty,0)$ we can summarize $r< G(z) < 0$ for $z >Z$ and
\[
\partial \EDens(z) =
\begin{cases}
\emptyset & \tn{ for } z <0,
\\
(-\infty,r] & \tn{ for } z=0,
\\
\{ r \} & \tn{ for } z \in (0,Z],
\\
\{ G(z) \} & \tn{ for } z > Z,
\end{cases}
\qquad
\partial (\EDens^\ast)(s) =
\begin{cases}
\{ 0 \} & \tn{ for } s < r,
\\
[0,Z] & \tn{ for } s=r,
\\
\{ G^{-1}(s) \} & \tn{ for } s \in (r,0),
\\
\emptyset & \tn{ for } s \geq 0.
\end{cases}
\]
\end{lemma}
\begin{example}[Balanced quantization]
\label{ex:BforBalancedQuantization}
We consider the case of the standard Wasserstein-2 distance, where $\cRad(t)=t^2$, $\Floc^\ast(z)=z$ and $\Floc(0)=\infty$. Then for $z>0$,
\[
\EDens(z) = z \int_{\Hex(1/z)} |x|^2 \d x = \frac{5 \sqrt{3}}{54} \frac{1}{z}, \qquad \EDens'(z) = G(z) = - \frac{5 \sqrt{3}}{54} \frac{1}{z^2}.
\]
and so $Z=0$, $r=-\infty$. For $s < 0$,
\[
\EDens^\ast(s) = - 2 \sqrt{ - \frac{5 \sqrt{3}}{54} \, s}, \qquad (\EDens^\ast)'(s) = \sqrt{ - \frac{5 \sqrt{3}}{54} \frac{1}{s}} = G^{-1}(s).
\]
\end{example}
\begin{remark}[Bounds in terms of cell problem] \label{rem:CellProblem}
	$\EDens(z)$ can be interpreted as energy density associated with a regular triangular lattice with point density $z$ (that is, each Voronoi cell occupies an area of $1/z$). The energy of such a lattice with $M$ cells with total area $|\Omega|$ will be given by $\EDens(M/|\Omega|) \cdot |\Omega|$.
Taking into account the scaling factor $\varepsilon$, we can restate the bounds \eqref{eq:QuantizationLowerBound} and \eqref{eq:QuantizationUpperBound} as
\begin{align*}
\min_{x_1,\ldots,x_M\in\Omega}J_{\varepsilon}^M(x_1,\ldots,x_M) &
	\geq |\Omega| \cdot \EDens\big(\tfrac{\varepsilon^2\,M}{|\Omega|}\big)
	\qquad \tn{and} \\
J_\varepsilon^M(x_1,\ldots,x_M) &
	\leq |\Omega| \cdot  \EDens\big(\tfrac{\varepsilon^2\,M}{|\Omega|}\big)
	- |\partial\Omega|\sqrt{\tfrac{8|\Omega|}{3\sqrt3M}}\,\Floc^\ast\!\!\left(-\cRad\left(\sqrt{
    \tfrac{C_Q \,|\Omega|}{\varepsilon^2M}
    }\right)\right).
\end{align*}
\end{remark}
\begin{proof}[Proof of \cref{lem:cellProblem}]
By \eqref{enm:quantIntPositive}, $\EDens(z)>0$ for $z>0$. Likewise, $\EDens(z)<\infty$ by \eqref{enm:quantIntMonotone} and \eqref{enm:quantIntFinite}.
Now observe that $\EDens$ yields the average value of $-\Floc^\ast(-\cRad(d(\cdot,0)))$ over $\Hex(1/z)$.
Hence, the monotonicity \eqref{enm:quantIntMonotone} of $-\Floc^\ast\circ(-\cRad)$ implies that $\EDens$ is decreasing. Further,\[
\lim_{z \searrow 0} \, \EDens(z) = \lim_{z \searrow 0} \, \int_{\Hex(1)} - \Floc^\ast(-\cRad(d(y/\sqrt{z},0))) \, \d y =
-\Floc^\ast(-\cRad(\infty))=B(0).
\]
Since $-\Floc^\ast\circ(-\cRad)$ is continuous by \eqref{enm:quantIntContinuous}, the integral in the definition of $\EDens$ is differentiable with respect to $z$ by the Leibniz integral rule, and we have
\begin{align*}
\EDens'(z)
&=\frac\d{\d z}\left[z\int_{\Hex(1/z)}-\Floc^\ast(-\cRad(d(x,0)))\,\d x\right]\\
&=\int_{\Hex(1/z)}-\Floc^\ast(-\cRad(d(x,0)))\,\d x+z\int_{\partial \Hex(1/z)}-\Floc^\ast(-\cRad(d(x,0)))v_n(z)\,\d x\cdot\left(-\frac1{z^2}\right)\,,
\end{align*}
where $v_n(z)=1/|\partial \Hex(1/z)|$ is the normal velocity of the hexagonal boundary as the area of the hexagon is increased at rate $1$.
This coincides with the formula provided in the statement.
To check convexity we first assume that $-\Floc^\ast\circ(-\cRad)$ is differentiable. In the following we use the notation
\begin{align*}
	\avint_{\partial \Hex(1/z)} f\,\d x= \frac{1}{|\partial \Hex(1/z)|} \int_{\partial \Hex(1/z)} f\,\d x
\end{align*}
and calculate
\begin{align*}
&\EDens''(z) =
\\
&=
	-\frac1{z^2}\left[
		\EDens(z)-\avint_{\partial \Hex(1/z)}-\Floc^\ast(-\cRad(d(x,0)))\,\d x
		\right] \\
& \quad	+\frac1z\left[
		\frac1z\left(
			\EDens(z)-\avint_{\partial \Hex(1/z)}-\Floc^\ast(-\cRad(d(x,0)))\,\d x
			\right)
		-\frac\d{\d z}\left(\avint_{\partial \Hex(1/z)}-\Floc^\ast(-\cRad(d(x,0)))\,\d x\right)
		\right]\\
&	=-\frac1z\frac\d{\d z}\left(
		\avint_{\partial \Hex(1/z)}-\Floc^\ast(-\cRad(d(x,0)))\,\d x\right)
\\
&	\geq0\
\end{align*}
since $-\Floc^\ast\circ(-\cRad)$ is nondecreasing. Therefore $\EDens$ is convex.
By \eqref{enm:quantIntMonotone} there exists some $R\in(0,\infty]$ such that $-\Floc^\ast\circ(-\cRad)$ is strictly increasing on $[0,R)$ and constant on $[R,\infty)$.
Thus we see $\EDens''>0$ on $(Z,\infty)$ for some $Z\geq0$ and $\EDens''(z)=0$ for $z<Z$.
The monotonicity properties of $\EDens'$ without assuming differentiability of $-\Floc^\ast\circ(-\cRad)$ now follow by a standard approximation argument.
Note that positivity and monotonicity of $B$ imply $B'(z)\to0$ as $z\to\infty$.
We leave it as an easy exercise in convex analysis to check the expressions for the subdifferentials $\partial \EDens$ and $\partial(\EDens^\ast)$.
\end{proof}

\begin{remark}[Calculation of asymptotic density]\label{rem:calculation}
Given a density $m$, the asymptotically optimal point density $D$ can be computed numerically based on the function $\EDens$ using
\begin{equation*}
D(x)\in\partial \EDens^\ast(\lambda/m(x))=\begin{cases}
\{0\}&\text{if }\lambda/m(x)<r,\\
[0,Z]&\text{if }\lambda/m(x)=r,\\
\{(\EDens')^{-1}(\lambda/m(x))\}&\text{if }\lambda/m(x)\in(r,0),\\
\emptyset&\text{otherwise,}
\end{cases}
\end{equation*}
where $r$ was defined in \cref{lem:cellProblem}.
\end{remark}

\begin{example}[Zador's Theorem is a special case of  Theorem \ref{thm:asymptoticQuantization}]
\label{example:Zador}
We show that Zador's Theorem \cite{Gr04,Za82} in two dimensions (see equation \eqref{eq:Zador} with $d=2$) can be recovered from Theorem 
\ref{thm:asymptoticQuantization}
by taking $\cRad(t)=t^p$ and $\Floc(s) = \iota_{\{1\}}(s)$.
In this case $\min J_{\varepsilon_M}$ is just the standard (balanced) optimal quantization error with respect to the Wasserstein-$p$ distance.
Note that $\Floc(0)=+\infty$ but the transport cost $\cRad$ is finite and so assumption \eqref{eq:QuantizationAsymptoticAssumptions_a} is satisfied.
We have $\Floc^\ast(z)=z$ and
\[
B(z) = c_6(p) z^{- \frac p2} \quad \tn{for } z>0 \tn{ and } +\infty \tn{ otherwise,}
\]
where
\[
c_6(p) = \int_{H(1)} |y|^p \, \d y.
\]
Therefore, for $z>0$, $s<0$, 
\[
B'(z) = - \frac{p}{2} c_6(p) z^{-\frac p2 - 1}, \qquad
(B')^{-1}(s) = \left( - \frac{p  c_6(p)}{2s} \right)^{\frac{2}{p+2}}.
\]
Assume that we are in Regime 3: $\lim_{M \to \infty} \varepsilon_M^2 M = P$.
Note that $B'$ is nowhere constant and so $Z=0$ and $r = \lim_{z \to 0} B'(z) = -\infty$. By Remark 
\ref{rem:calculation}, if $m(x) > 0$,
\begin{equation}
\label{eq:Rem4.18_D}
D(x) = (B')^{-1}(\lambda/m(x)) = \left( - \frac{p  c_6(p) m(x)}{2 \lambda} \right)^{\frac{2}{p+2}}
\end{equation}
where $\lambda<0$ is the constant given by Theorem 
\ref{thm:asymptoticQuantization}. Then
\begin{equation}
\label{eq:Rem4.18_P}
P = \int_\Omega D(x) \, \d x = \left( - \frac{p  c_6(p)}{2 \lambda} \right)^{\frac{2}{p+2}} \int_{\Omega} m(x)^{\frac{2}{p+2}} \, \d x.
\end{equation}
We can eliminate $\lambda$ from \eqref{eq:Rem4.18_D} and \eqref{eq:Rem4.18_P} to write $D$ in terms of $P$:
\[
D(x) = P m^{\frac{2}{p+2}} \left( \int_{\Omega} m(x)^{\frac{2}{p+2}} \, \d x \right)^{-1}.
\]
Therefore, by Theorem \ref{thm:asymptoticQuantization},
\begin{align*}
\lim_{M \to \infty} \varepsilon_M^{-p} \min_{\{ x_i \}_{i=1}^M} \sum_{i=1}^M \int_{V_i} |x-x_i|^p \, m(x) \d x
& = \lim_{M \to \infty} \min J^M_{\varepsilon_M}
\\
& = \int_{\Omega} B(D(x)) \, m(x) \d x 
\\
& = P^{-\frac{p}{2}} c_6(p) \left( \int_{\Omega} m(x)^{\frac{2}{p+2}} \, \d x \right)^{\frac{p+2}{2}}.
\end{align*}
Since $\lim_{M \to \infty} \varepsilon_M^2 M = P$, we can rewrite this as
\[
\lim_{M \to \infty} M^{\frac{p}{2}} \min_{\{ x_i \}_{i=1}^M} \sum_{i=1}^M \int_{V_i} |x-x_i|^p \, m(x) \d x
= c_6(p) \left( \int_{\Omega} m(x)^{\frac{2}{p+2}} \, \d x \right)^{\frac{p+2}{2}}
\]
which is exactly Zador's Theorem in two dimensions; see equations \eqref{eq:Zador} and \eqref{eq:Cp2}.
\end{example}

\begin{proof}[Proof of \cref{thm:asymptoticQuantization}]
\textbf{Regime 1: $\lim_{M\to\infty}\varepsilon_M^2M=\infty$.}
We can pick a sequence of radii $r_M$ such that $r_M/\varepsilon_M\to0$, but still $r_M^2M\to\infty$ (for instance pick $r_M=\varepsilon_M^{1/2}/M^{1/4}$).
Since $\Omega$ is a Lipschitz domain, for $M$ large and thus $r_M$ small enough, $\Omega$ can be covered with $K_M\leq 2|\Omega|/r_M^2$ squares of side length $r_M$
(for instance cover $\Omega$ with a regular square grid of lattice spacing $r_M$).
Now position the points $x_1,\ldots,x_M$ arbitrarily with the only condition that each square contains at least one point (this is possible for $M$ large enough, since $K_M/M\to0$ as $M\to\infty$).
Obviously, $\Omega$ is covered by the balls $B_{2r_M}(x_i)$ of radius $2r_M$ centred at the $x_i$, as $B_{2r_M}(x_i)$ includes the whole square containing $x_i$.
Thus we find $V_i(x_1,\ldots,x_M) \subset B_{2r_M}(x_i)$ for $i =1,\ldots,M$.
Then
\begin{align*}
	\min J_{\varepsilon_M}^M & \leq J_{\varepsilon_M}^M(x_1,\ldots,x_M) = \sum_{i=1}^M
		\int_{V_i(x_1,\ldots,x_M)} -\Floc^\ast
			\left(-\cRad\left(\tfrac{d(x,x_i)}{\varepsilon_M}\right)\right)\,\d\mu(x) \\
	& \leq -\Floc^\ast
			\left(-\cRad\left(\tfrac{2r_M}{\varepsilon_M}\right)\right) \cdot \mu(\Omega)
			\to -\Floc^\ast(-\cRad(0))\cdot\mu(\Omega) \quad \tn{as } M \to \infty
\end{align*}
since $[0,\infty) \ni z \mapsto -\Floc^\ast(-\cRad(z))$ is continuous in $z=0$ by \eqref{enm:quantIntContinuous}.
$J_{\varepsilon_M}^M\geq-\Floc^\ast(-\cRad(0))\cdot\mu(\Omega)$ follows from \eqref{enm:quantIntMonotone}, the monotonicity of $-\Floc^\ast\circ(-\cRad)$.

\textbf{Regime 2: $\lim_{M\to\infty}\varepsilon_M^2M=0$.}
\Cref{rem:APrioriEstimate} yields $\min J_{\varepsilon_M}^M \leq \mu(\Omega) \cdot \lim_{s\to\infty}-\Floc^\ast(-\cRad(s))$ (which may be infinite). Let now $r_1,r_2,\ldots$ be a positive sequence such that $r_M^2 \cdot M \to 0$ and $r_M/\varepsilon_M \to \infty$ as $M \to \infty$. Let $x_1,\ldots,x_M$ be arbitrary distinct points in $\Omega$ and set $S=\Omega \cap \bigcup_{i=1}^M B_{r_M}(x_i)$. Note that, since $r_M^2 \cdot M \to 0$ and $\mu \ll \Lebesgue$, $\mu(S) \to 0$ as $M \to \infty$. Clearly,
\begin{align*}
	\min_{i \in \{1,\ldots,M\}} -\Floc^\ast
			\left(-\cRad\left(\tfrac{d(x,x_i)}{\varepsilon_M}\right)\right) \geq
	\begin{cases}
		-\Floc^\ast
			\left(-\cRad\left(\tfrac{r_M}{\varepsilon_M}\right)\right) & \tn{for } x \in \Omega \setminus S, \\
			0 & \tn{for } x \in S.
	\end{cases}
\end{align*}
Therefore,
\begin{align*}
J_{\varepsilon_M}^M(x_1,\ldots,x_M) \geq -\Floc^\ast
			\left(-\cRad\left(\tfrac{r_M}{\varepsilon_M}\right)\right) \cdot \mu(\Omega \setminus S) \to \lim_{s\to\infty}-\Floc^\ast(-\cRad(s)) \cdot \mu(\Omega) \quad \tn{as } M \to \infty.
\end{align*}

\textbf{Regime 3: $\lim_{M\to\infty}\varepsilon_M^2M=P \in (0,\infty)$. Part 3.1: lower bound.}\phantomsection\label{QuantProof:3.1}
For $\delta \in (0,1)$ cover $\Omega$ by a regular grid of squares with edge length $\delta+\delta^2$. Denote by $\{S_i\}_{i=1}^N$ the $N$ squares that are fully contained in $\Omega$. For each $S_i$ denote by $\hat{S}_i$ the square of edge length $\delta$ that lies centered within $S_i$ such that $\dist(\hat{S}_i,S_j) \geq \delta^2/2$ for $i \neq j$ (the boundary layers of width $\delta^2$ will be necessary to control the interaction between neighbouring squares).
Denote the union of all $\hat{S}_i$ by $\hat{S}=\bigcup_{i=1}^N \hat{S}_i$. Since $\Omega$ is a Lipschitz domain and $\mu \ll \Lebesgue$, $\mu(\Omega \setminus \hat{S}) \to 0$ as $\delta \to 0$.

Let $x_1,\ldots,x_M$ be $M$ points from $\Omega$ and denote by $M_i$ the number of points in a square $S_i$ (points on the square boundaries are assigned to precisely one square). Obviously, $\sum_{i=1}^N M_i\leq M$.
Now pick an arbitrary $\eta>0$ (we will later send $\eta\to\infty$; it has the role to regularize the integrand and thus the cell problem function $\EDens$).
Note that for $x \in \hat{S}_i$ and $M$ large enough (depending on $\delta$) we have
$$\min_j d(x,x_j) \geq \min\left\{\min_{j: x_j \in S_i} d(x,\pi_{\hat{S}_i}(x_j)),\frac{\delta^2}2\right\}
\geq\min\left\{\min_{j: x_j \in S_i} d(x,\pi_{\hat{S}_i}(x_j)),\varepsilon_M\eta\right\}$$
where $\pi_{\hat{S}_i}$ denotes the orthogonal projection onto $\hat{S}_i$.
By the nonnegativity and monotonicity of $-\Floc^\ast(-\cRad(\cdot))$ one thus obtains
\begin{align*}
J_{\varepsilon_M}^M(x_1,\ldots,x_M) & = \int_\Omega \min_{j} -\Floc^\ast\left(-\cRad\left(\tfrac{d(x,x_j)}{\varepsilon_M}\right)\right)\,m(x)\,\d x \\
& \geq \sum_{i=1}^N m_i \int_{\hat{S}_i} \min_{j: x_j \in S_i} -\Floc^\ast\left(-\cRad\left(\min\left\{\tfrac{d(x,\pi_{\hat{S}_i}(x_j))}{\varepsilon_M},\eta\right\}\right)\right)\,\d x
\end{align*}
with $m_i := \inf_{x \in S_i} m(x)$.
Next introduce
$$\phi_\eta(t)=-\Floc^\ast\left(-\cRad(\min\{t,\eta\})\right).$$
One can readily verify that actually $\phi_\eta=-\Floc_\eta^\ast\circ(-\cRad)$, where $\Floc_\eta$ is the lower semi-continuous convex envelope of a modification of $\Floc$,
\begin{equation*}
\Floc_\eta=\tilde\Floc^{\ast\ast}
\qquad\text{with }
\tilde\Floc(s)=\begin{cases}
\Floc(s)&\text{if }s\neq0,\\
-\Floc^\ast\left(-\cRad\left(\eta\right)\right)&\text{else.}
\end{cases}
\end{equation*}
Let 
\[
\hat{M}_i = \# \{ \pi_{\hat{S}_i}(x_j) : x_j \in S_i \} \le M_i
\]
be the number of points in $S_i$ that are distinct after projecting onto $\hat{S}_i$.
We can now apply \cref{thm:lowerBound}
(lower bound for quantization of the Lebesgue measure)
or \cref{rem:CellProblem}
on each square $\hat{S_i}$ separately (the boundary layers of width $\delta^2$ are necessary to control the interaction between neighbouring squares),
\begin{align*}
J_{\varepsilon_M}^M(x_1,\ldots,x_M) & \geq \sum_{i=1}^N m_i \int_{\hat{S}_i} \min_{j: x_j \in S_i}
	\phi_\eta\left(d(x,\pi_{\hat{S}_i}(x_j))/\varepsilon_M \right)\,\d x \\
	& \geq \sum_{i=1}^N m_i\, \hat{M}_i \int_{H(|\hat{S}_i|/\hat{M}_i)} \phi_\eta\left(d(x,0)/\varepsilon_M \right)\,\d x \\
	& = \sum_{i=1}^N m_i\,|\hat{S}_i| \EDens_\eta\left(\tfrac{\varepsilon_M^2\,\hat{M}_i}{|\hat{S}_i|}\right)
\end{align*}
where $\EDens_\eta$ is defined as in \cref{lem:cellProblem}, only with $-\Floc^\ast\circ(-\cRad)$ replaced by $\phi_\eta$ or equivalently $\Floc$ by $\Floc_\eta$.
Now set
\begin{align*}
	E(x) := \begin{cases} \frac{\varepsilon_M^2\,\hat{M}_i}{|\hat{S}_i|} & \tn{ for $x \in \hat{S}_i$}, \\
	0 & \tn{otherwise,}
	\end{cases}
\end{align*}
and let $L_m$ denote the Lipschitz constant of the density $m$. Then
\begin{align*}
\int_{\Omega} \EDens_\eta (E(x)) \,  \d \mu(x)
& = \sum_{i=1}^N \int_{\hat{S}_i} \EDens_\eta \left( \tfrac{\varepsilon_M^2 \, \hat{M}_i}{\delta^2} \right) m(x) \, \d x
+ \int_{\Omega \setminus \hat{S}} \EDens_\eta(0) m(x) \, \d x
\\
& = \sum_{i=1}^N  \int_{\hat S_i} \EDens_\eta \left( \tfrac{\varepsilon_M^2 \, \hat{M}_i}{\delta^2} \right) m_i \, \d x
\\
& \quad + \sum_{i=1}^N \int_{\hat{S}_i} \EDens_\eta \left( \tfrac{\varepsilon_M^2 \, \hat{M}_i}{\delta^2} \right) \left( m(x) - m_i\right)\, \d x
+ \mu(\Omega \setminus \hat{S}) \cdot \EDens_\eta(0)
\\
& \leq 
	\sum_{i=1}^N m_i |\hat S_i|
		\EDens_\eta\left(\tfrac{\varepsilon_M^2\,\hat{M}_i}{\delta^2}\right)
	+ L_m \cdot \sqrt{2} \, \delta \cdot \EDens_\eta(0) \cdot |\Omega|
	+ \mu(\Omega \setminus \hat{S}) \cdot \EDens_\eta(0)
\end{align*}
where in the last step we used $\EDens_\eta(z) \leq \EDens_\eta(0)<\infty$ for $z \geq 0$.
Abbreviate the last two summands (that do not depend on $M$) as $C_\delta\EDens_\eta(0)$ and note that $C_\delta \to 0$ as $\delta \to 0$,
then in summary we have arrived at
\begin{equation}
\label{eq:GammaLimInf}
J_{\varepsilon_M}^M(x_1,\ldots,x_M) \geq \int_{\Omega} \EDens_\eta (E(x)) \,  \d \mu(x) - C_\delta\EDens_\eta(0).
\end{equation}
The function $E$ satisfies $\int_\Omega E(x)\,\d x = \varepsilon_M^2 \sum_{i=1}^N \hat{M}_i \leq \varepsilon_M^2 M$. By minimizing over all such functions we thus obtain a lower bound for the minimum,
\begin{align*}
	\min J_{\varepsilon_M}^M & \geq \inf \left\{
	\int_{\Omega} \EDens_\eta(E(x))\,m(x)\,\d x \,\middle|\,
E \in L^1(\Omega;[0,\infty)),
\int_\Omega E(x)\,\d x \leq \varepsilon_M^2 \cdot M \right\} - C_\delta\EDens_\eta(0). \nonumber
\intertext{Since $\EDens_\eta$ is nonincreasing, the estimate can be rewritten as}
  \min J_{\varepsilon_M}^M & \geq \inf \left\{
  \int_{\Omega} \EDens_\eta(E(x))\,m(x)\,\d x \,\middle|\,
  E \in L^1(\Omega;[0,\infty)),
   \int_\Omega E(x)\,\d x = \varepsilon_M^2 \cdot M \right\} -C_\delta\EDens_\eta(0).
\end{align*}
Now denote by $L_B$ the Lipschitz constant of $\EDens_\eta$ on $[0,\infty)$ (which exists by \cref{lem:cellProblem}).
If $E \in L^1(\Omega;[0,\infty))$ satisfies $\int_\Omega E(x)\,\d x = \varepsilon_M^2 \cdot M$, then $\tilde{E}:=\frac P{(\varepsilon_M^2M)}E$ satisfies $\int_\Omega \tilde{E}(x) \, \d x = P$ and
\begin{multline*}
 \left| \int_{\Omega} \EDens_\eta(\tilde{E}(x)) \, \d \mu(x) - \int_{\Omega} \EDens_\eta(E(x)) \, \d \mu(x) \right|
\leq  L_{\EDens} \int_{\Omega} |E(x)-\tilde{E}(x)| \, \d \mu(x)\\
\leq L_{\EDens}|\varepsilon_M^2M-P|\max_{x\in\Omega}m(x)
=:c_M,
\end{multline*}
which converges to zero as $M\to\infty$. Summarizing, we obtain
\begin{align*}	
\min J_{\varepsilon_M}^M & \geq \inf \left\{
	\int_{\Omega} \EDens_\eta(E(x))\,\d\mu(x) \,\middle|\, 
E \in L^1(\Omega;[0,\infty)),
 \int_\Omega E(x)\,\d x = P \right\}-C_\delta\EDens_\eta(0)-c_M.
\end{align*}
Now let first $M\to\infty$ and then $\delta\to0$.
Introducing a Lagrange multiplier $\lambda$ for the constraint on $E$ we thus find
\begin{align*}
	\liminf_{M \to \infty} \min J_{\varepsilon_M}^M & \geq
 \inf_{E \in L^1(\Omega;[0,\infty))}
	\sup_{\lambda \in \R} \int_{\Omega} \left[
		\EDens_\eta(E(x))\,m(x) - \lambda\,E(x)\right]\,\d x + \lambda \cdot P \\
	& \geq \sup_{\lambda \in \R} 
\, \inf_{E \in L^1(\Omega;[0,\infty))}
	 \int_{\Omega} \left[
		\EDens_\eta(E(x))\,m(x) - \lambda\,E(x)\right]\,\d x + \lambda \cdot P \\
	& \geq \sup_{\lambda \in \R}
		\int_{\Omega} \inf_{E \geq 0} \left[
		\EDens_\eta(E)\,m(x) - \lambda\,E\right]\,\d x + \lambda \cdot P \\
	& = \sup_{\lambda \in \R}
		\int_{\Omega} \left[-m(x) \cdot \EDens_\eta^\ast(\lambda/m(x))
		\right]\,\d x + \lambda \cdot P.
\end{align*}
Now as $\eta \to \infty$ one has the pointwise convergence $\phi_\eta \nearrow -\Floc^\ast\circ(-\cRad)$
and thus $\EDens_\eta \nearrow \EDens$ pointwise by the monotone convergence theorem.
In fact, $\phi_\eta(r)=-\Floc^\ast(-\cRad(r))$ for all $r<\eta$
and therefore $\EDens_\eta(z)=\EDens(z)$ for all $z>2/(3\sqrt3\eta^2)$.
Consequently, also $\EDens_\eta^\ast \searrow \EDens^\ast$ monotonously so that one obtains sharper bounds with increasing $\eta$.
Thus, using again the monotone convergence theorem,
\begin{align*}
	\liminf_{M \to \infty} \min J_{\varepsilon_M}^M & \geq \sup_{\eta>0} \sup_{\lambda \in \R}
		\int_{\Omega} \left[-m(x) \cdot \EDens_\eta^\ast(\lambda/m(x))
		\right]\,\d x + \lambda \cdot P \\
		& = \sup_{\lambda \in \R} \sup_{\eta>0}
		\int_{\Omega} \left[-m(x) \cdot \EDens_\eta^\ast(\lambda/m(x))
		\right]\,\d x + \lambda \cdot P \\
		& = \sup_{\lambda \in \R}
		\int_{\Omega} \left[-m(x) \cdot \EDens^\ast(\lambda/m(x))
		\right]\,\d x + \lambda \cdot P.
\end{align*}
Evaluating the supremum over $\lambda$ we finally obtain
\begin{align*}
	\liminf_{M \to \infty} \min J_{\varepsilon_M}^M &
		\geq \left[\kappa\mapsto\int_\Omega \EDens^\ast(\kappa/m(x))\,\d\mu(x)\right]^\ast(P)\,.
\end{align*}

\textbf{Part 3.2: optimal limit density.}\phantomsection \label{QuantProof:3.2}
Observe that
$\EDens^\ast(z)=\infty$ for $z>0$, $\EDens^\ast$ is convex, lower semi-continuous, nondecreasing, satisfies $\lim_{z \to -\infty} \EDens^\ast(z)/|z|=0$ (which follows from $s \cdot z-\EDens^\ast(z)\leq B(s)<\infty$ for all $s>0$, $z<0$), and has infinite left derivative at $0$ (by \cref{lem:cellProblem}).
Therefore the map
\[
\R \ni \lambda \mapsto \int_{\Omega} \left[-m(x) \cdot \EDens^\ast(\lambda/m(x))
		\right]\,\d x + \lambda \cdot P
\]
is concave and
there exists a maximizing $\lambda<0$ satisfying the necessary and sufficient optimality condition
\begin{multline*}
0 \in \partial \left[ \kappa \mapsto \int_{\Omega} m(x) \cdot \EDens^\ast\left(\tfrac\kappa{m(x)}\right)
		\,\d x - \kappa \cdot P \right](\lambda) \quad \Longleftrightarrow \quad
  \\
P\in\partial\left[\kappa\mapsto\int_\Omega m(x)\EDens^\ast\left(\tfrac\kappa{m(x)}\right)\,\d x\right](\lambda).
\end{multline*}

We next aim to find a function $D_\xi:\Omega\to[0,\infty)$ satisfying $D_\xi(x)\in\partial\EDens^\ast(\lambda/m(x))$ as well as $\int_{\Omega} D_{\xi}(x) \, \d x = P$.
To this end
recall $Z$ and $r$ from \cref{lem:cellProblem} and define
\begin{align*}
	D_\xi(x) & = \begin{cases} (\EDens^\ast)'(\lambda/m(x)) & \tn{if } m(x) > \lambda/r, \\
		\xi \in [0,Z] & \tn{if } m(x)=\lambda/r, \\
		0 & \tn{otherwise.}
		\end{cases}
\end{align*}
Since $m$ is continuous, its superlevel sets are Lebesgue measurable. $D_\xi$ is constructed by assigning new values to the level sets of $m$ in a monotone way (by the monotonicity of $(B^{\ast})'$, see \cref{lem:cellProblem}), hence it is also Lebesgue measurable.
We now pick $\xi(P)\in[0,Z]$ such that $\int_{\Omega} D_{\xi(P)}(x) \, \d x = P$. (Note that necessarily $\xi(P)=0$ if $Z=0$, and the choice of $\xi$ is irrelevant if $\{x \in \Omega\,|\,m(x)=\lambda/r\}$ is a nullset. The following argument still applies in these cases.)
Such a $\xi(P)$ exists due to $\int_\Omega D_0(x)\,\d x\leq P$ and $\int_\Omega D_Z(x)\,\d x\geq P$, as we now show.
Indeed, note by \cref{lem:cellProblem} that for all $x\in\Omega$ the function $D_0(x)$ equals the left derivative of $\EDens^\ast$ at $\lambda/m(x)$ (which by convention shall be $0$ for $m(x)=0$), while $D_Z(x)$ equals the right derivative.
Beppo Levi's monotone convergence theorem thus yields
\begin{align}
\nonumber
\int_\Omega D_0(x)\,\d x
& =\int_\Omega\lim_{\tilde\lambda\nearrow\lambda}\frac{m(x)\EDens^\ast\left(\tfrac{\tilde\lambda}{m(x)}\right)-m(x)\EDens^\ast\left(\tfrac\lambda{m(x)}\right)}{\tilde\lambda-\lambda}\,\d x
\\
\label{eq:secondlimit}
& =\lim_{\tilde\lambda\nearrow\lambda}\frac{\displaystyle \int_\Omega m(x)\EDens^\ast\left(\tfrac{\tilde\lambda}{m(x)}\right)\,\d x-\int_\Omega m(x)\EDens^\ast\left(\tfrac\lambda{m(x)}\right)\,\d x}{\tilde\lambda-\lambda}
\\
\nonumber
& \leq P
\end{align}
since $P\in\partial\left[\kappa\mapsto\int_\Omega m(x)\EDens^\ast\left(\tfrac\kappa{m(x)}\right)\,\d x\right](\lambda)$
and since \eqref{eq:secondlimit} is the left derivative of $\lambda\mapsto\int_\Omega m(x)\EDens^\ast\left(\tfrac\lambda{m(x)}\right)\,\d x$.
The inequality $\int_\Omega D_Z(x)\,\allowbreak\d x\allowbreak \geq P$ follows analogously.
Writing $D=D_{\xi(P)}$ we finally obtain \eqref{eqn:optCond} and
\begin{align*}
\lim_{M \to \infty} \min J_{\varepsilon_M}^M
& \geq\int_\Omega-m(x)\EDens^\ast(\lambda/m(x))\,\d x+\lambda P \\
& =\int_\Omega\frac\lambda{m(x)}D(x)-\EDens^\ast(\lambda/m(x))\,\d\mu(x)
\\
& =\int_\Omega \EDens(D(x))\,\d\mu(x)\,,
\end{align*}
where the last equality follows from the Moreau--Fenchel identity \cite[Prop.\,16.9]{BauschkeCombettes11}, which states that $\EDens(s)+\EDens^\ast(t) = st
\; \Longleftrightarrow \; s \in \partial \EDens^\ast(t) \; \Longleftrightarrow \; t \in \partial \EDens(s)$.

\textbf{Part 3.3: upper bound.}\phantomsection \label{QuantProof:3.3}
Finally, we derive the corresponding upper bound, essentially by constructing a piecewise triangular point configuration with point density approximating the expected point density $D$ from the above lower bound proof.
Fix $M \in \mathbb{N}$, $\delta > 0$, and $\eta \in (0,1)$.
Denote the $1$-neighbourhood of $\Omega$ by $\underline\Omega=\Omega+B_1(0)$
and extend the mass density $m$ and the expected point density $D$ to $\underline\Omega\setminus\Omega$ by zero.
Cover $\Omega$ with a tessellation of squares of side length $\delta$ (we will later send $\delta\to0$). We keep the squares $\{S_i\}_{i=1}^{N_\delta}$ that intersect $\Omega$.
We may assume $\delta$ to be small enough such that all squares lie within $\underline\Omega$.

Define 
$D_\eta: \underline\Omega \to (0,\infty)$
to be a slight modification of the expected point density $D$,
\[
D_\eta = (1-\eta) D + \eta \, \frac{P}{2|\underline\Omega|}.
\]
The main role of the regularization parameter $\eta$ is to ensure that we distribute particles throughout the whole domain $\Omega$, even in regions where $m=0$. We will send $\eta \to 0$ at the very end of the proof. For $i \in \{1,\ldots,N_\delta\}$, define the point number $M_i=M_i(M,\delta,\eta)$ associated with square $S_i$ by
\[
M_i 
= 
\left\lceil 
\tfrac{1}{\varepsilon_M^2} \int_{S_i} D_\eta(x)\,\d x
\right\rceil.
\]
Note that
\[
\sum_{i=1}^{N_\delta} M_i \le M \quad \textrm{if $M$ is sufficiently large.}
\]
Indeed, we have
\begin{multline*}
\varepsilon_M^2 \sum_{i=1}^{N_\delta} M_i 
\le \varepsilon_M^2 \sum_{i=1}^{N_\delta}
\left( \tfrac1{\varepsilon_M^2} \int_{S_i} D_\eta(x)\,\d x + 1\right)
\\
< \left( (1-\eta) \int_\Omega D(x) \, \d x + \eta \tfrac P2 \right)  + \varepsilon^2_M N_\delta
= (1-\tfrac\eta2)P + \varepsilon^2_M N_\delta
\end{multline*}
so that
\[
\varepsilon_M^2 \left( \sum_{i=1}^{N_\delta} M_i - M \right)
< (1-\tfrac\eta2)P - \varepsilon_M^2 M + \varepsilon_M^2 N_\delta
\xrightarrow[M \to \infty]{} -\tfrac\eta2P < 0.
\]

Next, on each $S_i$ we choose $M_i$ quantization points as in \cref{thm:upperBound} 
(upper bound for quantization of the Lebesgue measure)
and distribute the remaining $M-\sum_{i=1}^{N_\delta}M_i$ points arbitrarily in $\Omega$ (which does not increase the cost).
By \cref{rem:CellProblem} we thus obtain
\begin{align*}
J_{\varepsilon_M}^M(x_1,\ldots,x_M) 
& \leq
\sum_{i=1}^{N_\delta} \int_{S_i} 
\, \min_{k \in \{1,\ldots,M\}:x_k\in S_i}-\Floc^\ast
		\left(-\cRad\left(\frac{d(x,x_k)}{\varepsilon_M}\right)\right)
	\,\d \mu(x)
\\
&\leq
\sum_{i=1}^{N_\delta} \max_{S_i}m\, \int_{S_i} 
\, \min_{k \in \{1,\ldots,M\}:x_k\in S_i}-\Floc^\ast
		\left(-\cRad\left(\frac{d(x,x_k)}{\varepsilon_M}\right)\right)
	\,\d x
\\
&\leq
\sum_{i=1}^{N_\delta} \max_{S_i}m
\left[|S_i| \cdot  \EDens\big(\tfrac{\varepsilon_M^2\,M_i}{|S_i|}\big)
- |\partial S_i|\sqrt{\tfrac{8|S_i|}{3\sqrt3M_i}}\,\Floc^\ast\!\!\left(-\cRad\left(\sqrt{
\tfrac{C_Q\,|S_i|}{\varepsilon_M^2M_i}
}\right)\right)\right].
\end{align*}
Exploiting the continuity of $B$ and finiteness and monotonicity of $-\Floc^\ast\circ(-\cRad)$ on $(0,\infty)$ as well as
\begin{equation*}
\lim_{M \to \infty} \varepsilon_M^2 M_i = 
\int_{S_i} D_\eta(x) \, \d x>0
\end{equation*}
we thus obtain
\begin{equation*}
\limsup_{M\to\infty}J_{\varepsilon_M}^M(x_1,\ldots,x_M) 
\leq\sum_{i=1}^{N_\delta} \max_{S_i}m\,
|S_i| \cdot  \EDens\left(\tfrac{1}{|S_i|}\int_{S_i}D_\eta(x)\,\d x\right).
\end{equation*}
Define $E^\delta: \underline\Omega \to [0,\infty)$, $m^\delta:\underline\Omega \to [0,\infty)$ by
\begin{alignat*}{3}
E^{\delta}(x) & = \frac1{|S_i|} \int_{S_i} D_\eta(x) \, \d x &\quad\text{if } x \in S_i&\quad\text{and $E^{\delta}(x)=\frac{\eta P}{2|\underline\Omega|}$ else},
\\
m^\delta(x) & = \max_{S_i} m &\quad\text{if } x \in S_i&\quad\text{and $m^\delta(x)=0$ else}.
\end{alignat*}
Then we can rewrite the previous inequality as follows:
\begin{equation}
\label{eq:AlmostThere1}
\limsup_{M\to\infty}\min J_{\varepsilon_M}^M \le 
\limsup_{M \to \infty} 
J_{\varepsilon_M}^M(x_1,\ldots,x_M)  
\le \int_{\underline\Omega}B\left( E^\delta(x) \right) m^\delta(x) \, \d x.
\end{equation}

Next we pass to the limits $\delta \to 0$ and $\eta \to 0$, in that order. By the Lebesgue Differentiation Theorem, $\lim_{\delta \to 0} E^\delta = D_\eta$ pointwise almost everywhere. Since $m$ is upper semi-continuous, then $\lim_{\delta \to 0} m^\delta =m$ pointwise. Moreover, $E^\delta \ge \eta P/2|\underline\Omega|$ 
and $B$ is nonincreasing
on $(0,\infty)$. Hence
\[
B\left( E^\delta(x) \right) m^\delta(x)
\le B \left( \frac{\eta P}{2|\underline\Omega|} \right) \max_\Omega m.
\]
Therefore, by the Dominated Convergence Theorem,
\begin{equation}
\label{eq:AlmostThere2}
    \lim_{\delta \to 0^+}
\int_{\underline\Omega} B\left( E^\delta(x) \right) m^\delta(x) \, \d x
= \int_\Omega B\left( D_\eta(x) \right) m(x) \, \d x.
\end{equation}
Finally, by the convexity of $B$,
\begin{equation}
\label{eq:AlmostThere4}
 \int_\Omega B\left( D_\eta(x) \right) m(x) \, \d x
 \le
 (1-\eta) \int_\Omega B\left( D(x) \right) m(x) \, \d x
 + \eta B \left( \frac{P}{2|\underline\Omega|} \right) \int_\Omega m(x) \, \d x.
\end{equation}
By taking the limits $\delta \to 0$, then $\eta \to 0$ in \eqref{eq:AlmostThere1} and using 
\eqref{eq:AlmostThere2}-\eqref{eq:AlmostThere4}
we obtain the matching upper bound
\[
\limsup_{M\to\infty}\min J_{\varepsilon_M}^M \le
\int_\Omega B\left( D(x) \right) m(x) \, \d x
\]
as required.
\end{proof}

\begin{remark}[Lipschitz condition]
Inspecting the proof we see that the Lipschitz condition on $m$ can actually be replaced by mere continuity;
then all estimates based on the Lipschitz constant have to be replaced using the modulus of continuity of $m$.
\end{remark}

\begin{remark}[$\Gamma$-convergence for unbalanced quantization]
\label{rem:Gamma}
We conjecture that \Cref{thm:asymptoticQuantization} can be expanded into a $\Gamma$-convergence result in the spirit of \cite[Proposition 7.18]{Santambrogio-CV} for balanced optimal transport. For given $M \in \N$, define the functional
\begin{align*}
	\mc{J}^M(\nu) & := \begin{cases}
		J^M_{\varepsilon_M}(x_1,\ldots,x_M) & \text{if } \nu=\varepsilon_M^2 \sum_{i=1}^M \delta_{x_i} \\
		+ \infty & \tn{otherwise,}
		\end{cases}
\intertext{and let the limit functional be given by}
	\mc{J}(\nu) & := \begin{cases} \int_\Omega \EDens\left(\RadNik{\nu}{\Lebesgue}(x)\right)\,\d\mu(x) & \tn{if } \nu \geq 0, \\
	+ \infty & \tn{otherwise.}
	\end{cases}
\end{align*}
In the definition of $\mc{J}$, the part of $\nu$ that is singular with respect to $\Lebesgue$ is simply discarded.
It seems plausible that $\mc{J}^M$ $\Gamma$-converges to $\mc{J}$ with respect to the weak topology.
\Cref{thm:asymptoticQuantization}\eqref{enm:fewMasses} describes the special case of this result where the limit measure $\nu$ is $0$.
\Cref{thm:asymptoticQuantization}\eqref{enm:normalMasses} describes the special case of $\nu=D \cdot \Lebesgue \neq 0$ that are minimizing for a prescribed mass $P$.
In the regime of \Cref{thm:asymptoticQuantization}\eqref{enm:manyMasses}, the sequence of measures is diverging and thus not described by such a $\Gamma$-convergence.

\hyperref[QuantProof:3.1]{Part 3.1} of the proof of \Cref{thm:asymptoticQuantization} establishes the lim-inf inequality for minimizing sequences.
\hyperref[QuantProof:3.2]{Part 3.2} identifies the optimal density $D$ of the limit functional $\mc{J}$.
\hyperref[QuantProof:3.3]{Part 3.3} essentially derives the lim-sup condition for limit measures $\nu=D \cdot \Lebesgue$, i.e.~without a Lebesgue-singular part. The optimality of $D$ is not leveraged in this part of the proof and it applies to general densities.
A complete $\Gamma$-convergence result would require an extension of the first part to configurations $(x_i)_{i=1}^M$ that are not (approximately) minimal, and to include the Lebesgue-singular part of $\nu$ in both the lower and upper bounds. Given that the proof for the lower bound is already rather technical, we choose to not provide this extension here.
\end{remark}

\begin{remark}[Quantization regimes]\label{rem:quantizationRegimes}
The proof shows that the set of near- or asymptotically optimal point distributions for $\lim_{M\to\infty}\varepsilon_M^2M\in\{0,\infty\}$ is quite degenerate.
Indeed, if the limit is zero, then arbitrarily placed points $x_1,\ldots,x_M\in\Omega$ were shown to asymptotically achieve the optimal energy.
The interpretation is that in the limit $M\to\infty$ no transport takes place between $\mu$ and its discrete quantization approximation so that the quantization energy equals the cost for changing mass distribution $\mu$ to zero.
If on the other hand the limit is infinite, then Dirac masses can be distributed over $\Omega$ in such a dense fashion that all transport distances and thus transport costs become negligibly small.
Thus, to achieve the asymptotic cost $0$ it suffices for instance to have a more or less uniform distribution of $x_1,\ldots,x_M\in\Omega$, but otherwise the point arrangement does not matter.
The case $\lim_{M\to\infty}\varepsilon_M^2M\in(0,\infty)$ seems to be more rigid; here the optimal asymptotic cost is achieved by a construction which locally looks like a triangular lattice.
\end{remark}

\begin{example}[Hellinger--Kantorovich]
The function $\EDens$ from \cref{lem:cellProblem} and its derivative $\EDens'$ can be computed numerically for different unbalanced transport models;
we here consider the Hellinger--Kantorovich setting.
In this case, computing the integral just on one triangular segment of $H(\frac1z)$ we obtain
\begin{align*}
\EDens(z)&=z\left(6\int_{-\pi/6}^{\pi/6}\int_0^{L(\alpha,z)}\sin^2\left(\min\left\{r,\tfrac{\pi}{2}\right\}\right)r\,\d r\,\d\alpha\right)\\
&=3\int_{-\pi/6}^{\pi/6}z\max\left\{\tfrac14+\tfrac{L(\alpha,z)^2}2-\tfrac{\cos(2L(\alpha,z))}4-\tfrac{L(\alpha,z)\sin(2L(\alpha,z))}2,\tfrac12-\tfrac{\pi^2}8+L(\alpha,z)^2\right\}\,\d\alpha
\end{align*}
for $L(\alpha,z)=1/(\sqrt{2\sqrt3z}\cos\alpha)$ the length of the ray starting from the hexagon centre at angle $\alpha$.
The resulting  $\EDens'$ (computed numerically) is shown in \cref{fig:asymptoticDensities}.
Thus, for a given mass distribution $\mu=m \Lebesgue \restr\Omega$ we can compute the asymptotically optimal point density $D$ of the quantization problem from \cref{thm:asymptoticQuantization,rem:calculation}.
\Cref{fig:asymptoticDensities} shows computed examples for such asymptotic densities.
One can see that the variations of $\mu$ are reduced for large values of $P$, but amplified for small values of $P$ (in particular, large areas of $\Omega$ have zero point density).
\end{example}

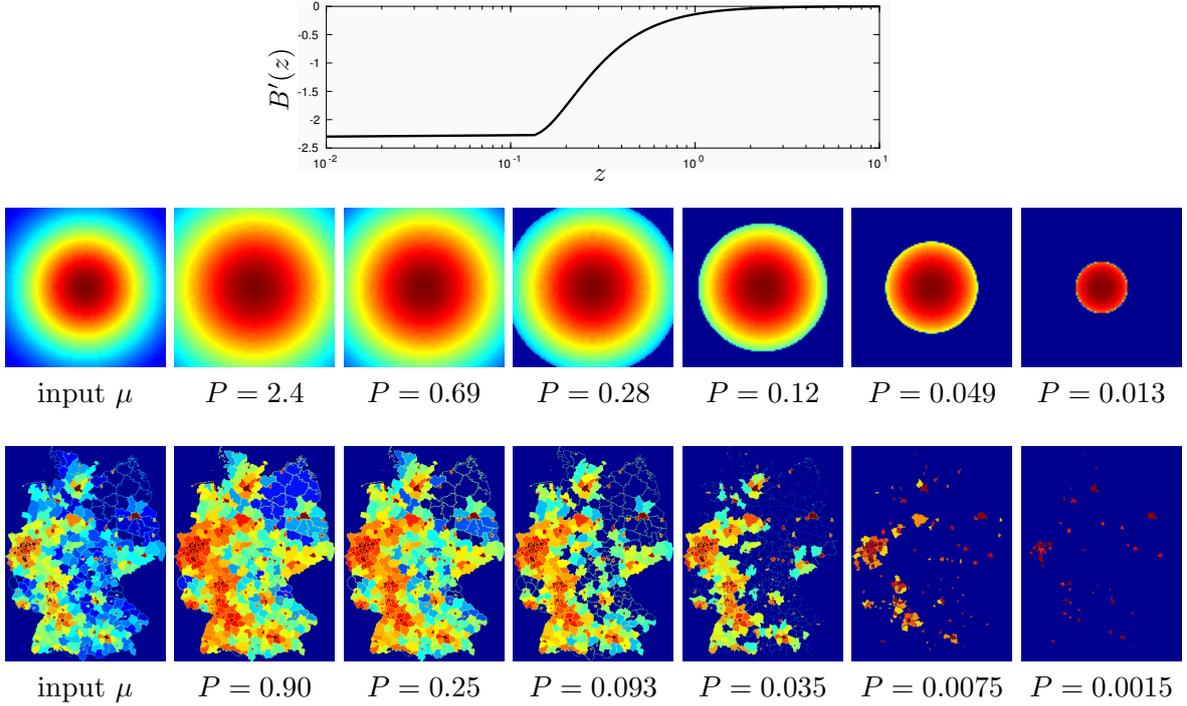
\begin{figure}[H]
\begin{center}
\includegraphics[width=.5\linewidth]{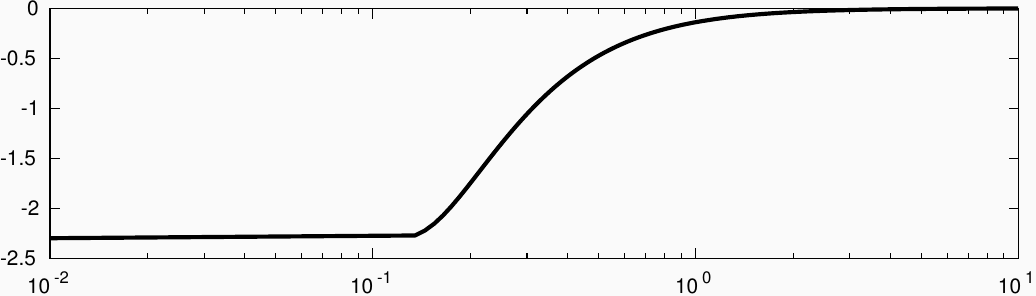}%
\setlength\unitlength{.5\linewidth}%
\begin{picture}(0,0)
\put(-1.05,.1){\rotatebox{90}{\small$B'(z)$}}
\put(-.5,-.02){\small$z$}
\end{picture}%
\\[1\baselineskip]
\end{center}
\setlength\unitlength{.135\linewidth}%
\centering%
\includegraphics[width=\unitlength]{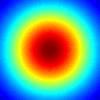}
\includegraphics[width=\unitlength]{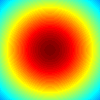}
\includegraphics[width=\unitlength]{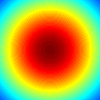}
\includegraphics[width=\unitlength]{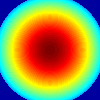}
\includegraphics[width=\unitlength]{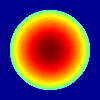}
\includegraphics[width=\unitlength]{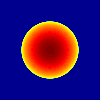}
\includegraphics[width=\unitlength]{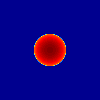}\\
\makebox[\unitlength][c]{input $\mu$}
\makebox[\unitlength][c]{$P=2.4$}
\makebox[\unitlength][c]{$P=0.69$}
\makebox[\unitlength][c]{$P=0.28$}
\makebox[\unitlength][c]{$P=0.12$}
\makebox[\unitlength][c]{$P=0.049$}
\makebox[\unitlength][c]{$P=0.013$}\\[1\baselineskip]
\includegraphics[width=\unitlength]{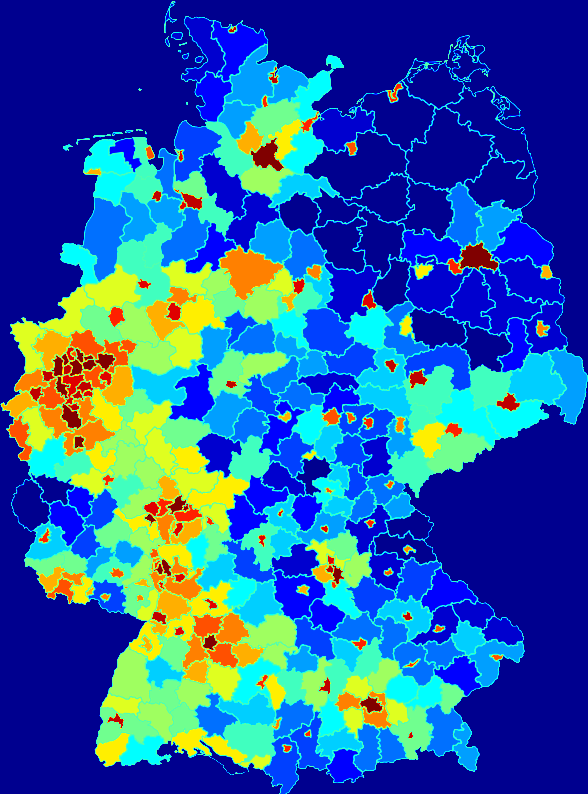}
\includegraphics[width=\unitlength]{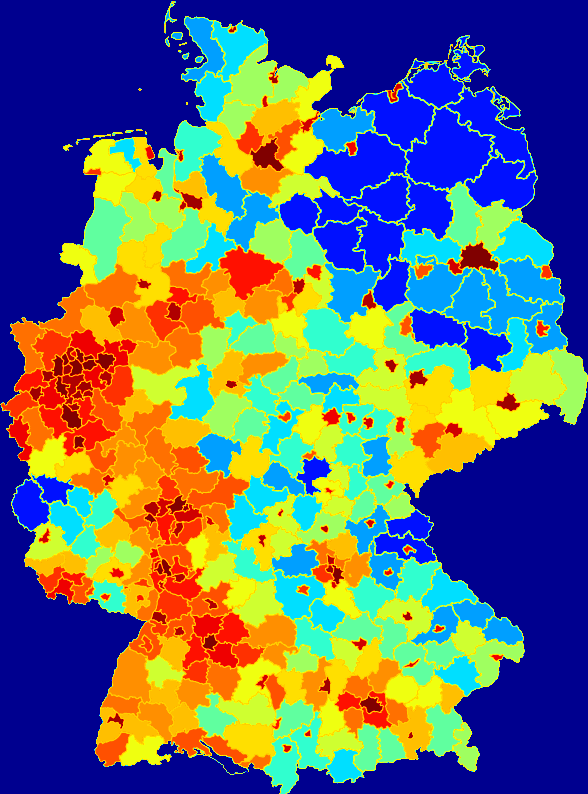}
\includegraphics[width=\unitlength]{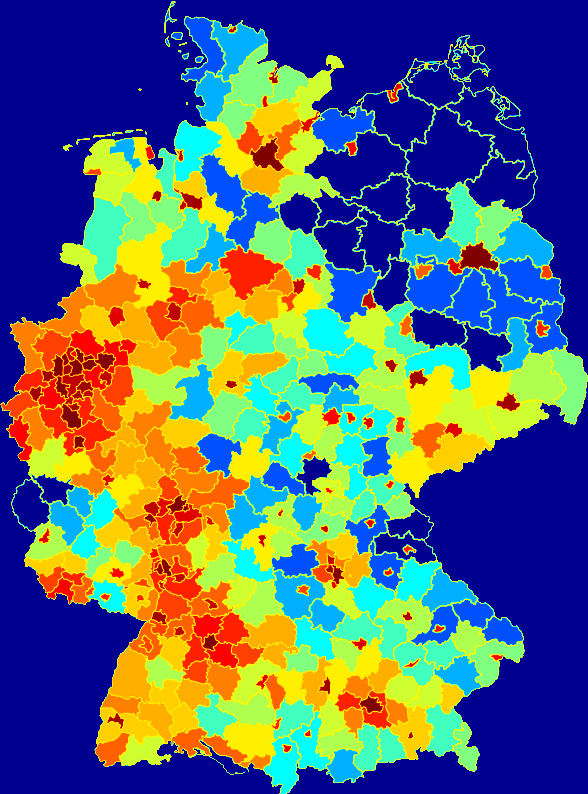}
\includegraphics[width=\unitlength]{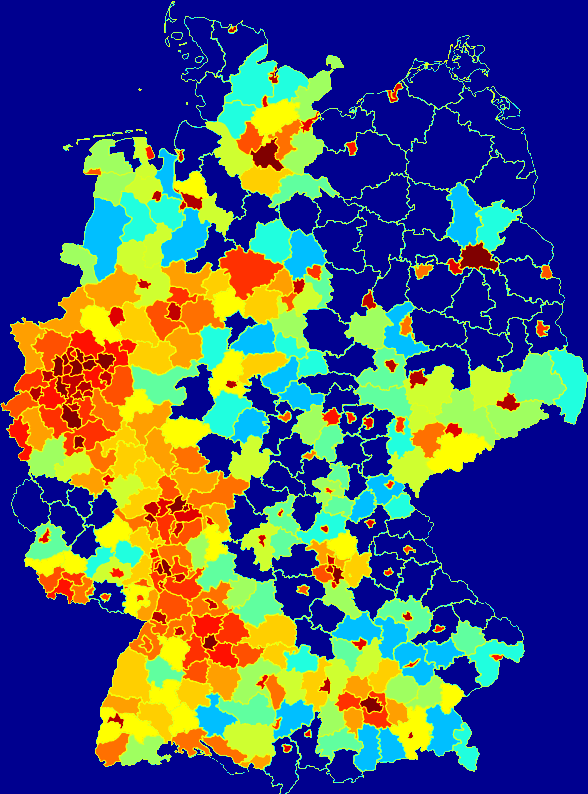}
\includegraphics[width=\unitlength]{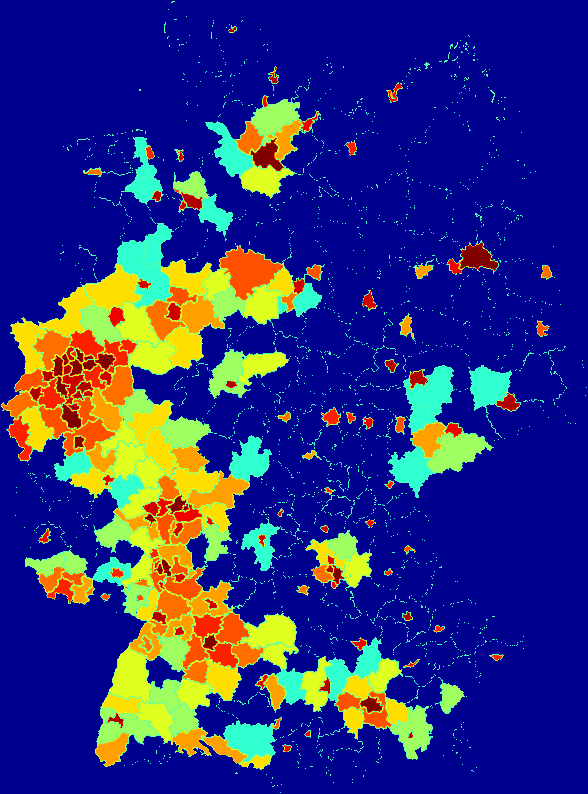}
\includegraphics[width=\unitlength]{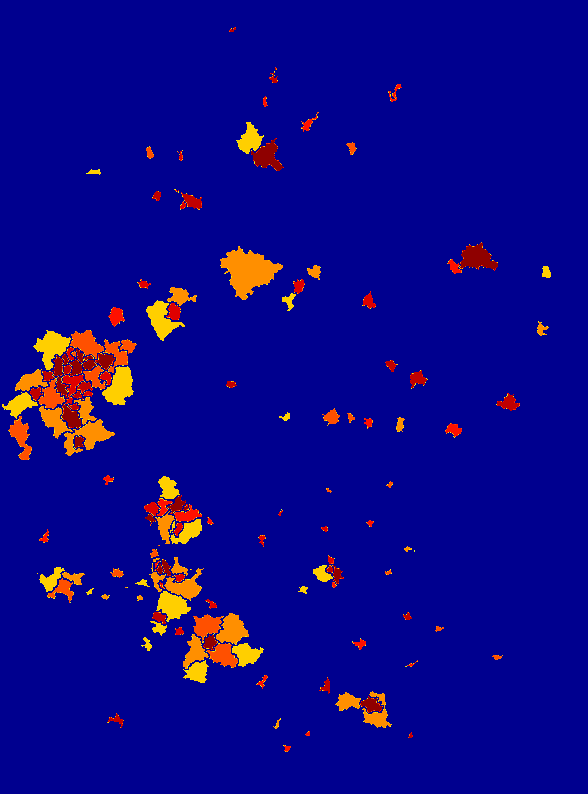}
\includegraphics[width=\unitlength]{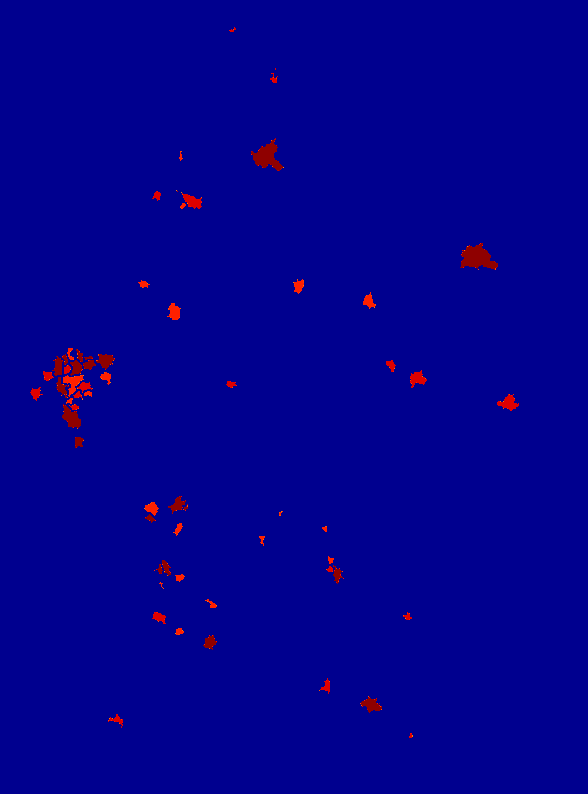}\\
\makebox[\unitlength][c]{input $\mu$}
\makebox[\unitlength][c]{$P=0.90$}
\makebox[\unitlength][c]{$P=0.25$}
\makebox[\unitlength][c]{$P=0.093$}
\makebox[\unitlength][c]{$P=0.035$}
\makebox[\unitlength][c]{$P=0.0075$}
\makebox[\unitlength][c]{$P=0.0015$}
\caption{\textit{Top row:} $\EDens'$ from \cref{lem:cellProblem} for Hellinger--Kantorovich transport.
\textit{Middle and bottom row:} input distribution $\mu$ (a Gaussian and same data as in \cref{fig:quantizationEnergyDecrease}) as well as asymptotically optimal point densities $D$ for different values of $P=\lim_{M\to\infty}\varepsilon_M^2M$ (colour-coding from blue for $0$ to red for maximum value).}
\label{fig:asymptoticDensities}
\end{figure}

\paragraph{Acknowledgments.}
DPB acknowledges financial support from the Engineering and Physical Sciences Research Council (EPSRC) grants EP/R013527/1 Designer Microstructure via Optimal Transport Theory and EP/V00204X/1 Mathematical Theory of Polycrystalline Materials.
BS was supported by the Emmy Noether programme of the Deutsche
Forschungsgemeinschaft (DFG, German Research Foundation), project number
403056140. BW was supported by the Deutsche Forschungsgemeinschaft (DFG, German Research Foundation) under Germany’s Excellence Strategy EXC 2044-390685587, Mathematics M\"unster: Dynamics-Geometry-Structure.
 
\bibliography{references}{}

\begin{thebibliography}{10}

\bibitem{AFP00}
L.~Ambrosio, N.~Fusco, and D.~Pallara.
\newblock {\em Functions of Bounded Variation and Free Discontinuity Problems}.
\newblock {Oxford University Press}, 2000.

\bibitem{AHA98}
F.~Aurenhammer, F.~Hoffmann, and B.~Aronov.
\newblock {M}inkowski-type theorems and least-squares clustering.
\newblock {\em Algorithmica}, 20(1):61--76, 1998.

\bibitem{AurenhammerKleinLee}
F.~Aurenhammer, R.~Klein, and D.-T. Lee.
\newblock {\em Voronoi diagrams and Delaunay triangulations}.
\newblock World Scientific Publishing Company, 2013.

\bibitem{AydinIacobelli}
A.~D. Aydın and M.~Iacobelli.
\newblock Asymptotic quantization of measures on riemannian manifolds via
  covering growth estimates.
\newblock {\em Adv. Math.}, 474:110311, 2025.

\bibitem{BarnesSloane1983}
E.~S. Barnes and N.~J.~A. Sloane.
\newblock The optimal lattice quantizer in three dimensions.
\newblock {\em SIAM J.~Algebraic Discrete Methods}, 4(1):30--41, 1983.

\bibitem{BauschkeCombettes11}
H.~H. Bauschke and P.~L. Combettes.
\newblock {\em Convex Analysis and Monotone Operator Theory in {H}ilbert
  Spaces}.
\newblock Springer, 2011.

\bibitem{BollobasStern72}
B.~Bollob\'as and N.~Stern.
\newblock The optimal structure of market areas.
\newblock {\em J.~Econ.~Theory}, 4(2):174--179, 1972.

\bibitem{BouchitteJimenezMahadevan2011}
G.~Bouchitt\'e, C.~Jimenez, and R.~Mahadevan.
\newblock Asymptotic analysis of a class of optimal location problems.
\newblock {\em J.~Math.~Pures Appl.}, 95(4):382--419, 2011.

\bibitem{BouchitteJimenezRajesh2002}
G.~Bouchitt\'e, C.~Jimenez, and M.~Rajesh.
\newblock Asymptotique d’un probl\`eme de positionnement optimal.
\newblock {\em C.~R.~Acad.~Sci.~Paris Ser.~I Math.}, 335:1--6, 2002.

\bibitem{BKRS20}
D.~P. Bourne, P.~J.~J. Kok, S.~M. Roper, and W.~D.~T. Spanjer.
\newblock Laguerre tessellations and polycrystalline microstructures: a fast
  algorithm for generating grains of given volumes.
\newblock {\em Phil. Mag.}, 100(21):2677--2707, 2020.

\bibitem{BournePeletierRoper}
D.~P. Bourne, M.~A. Peletier, and S.~M. Roper.
\newblock Hexagonal patterns in a simplified model for block copolymers.
\newblock {\em SIAM J. Appl. Math.}, 74(5):1315–--1337, 2014.

\bibitem{BournePeletierTheil14}
D.~P. Bourne, M.~A. Peletier, and F.~Theil.
\newblock Optimality of the triangular lattice for a particle system with
  {Wasserstein} interaction.
\newblock {\em Commun.~Math.~Phys.}, 329(1):117--140, 2014.

\bibitem{BourneRoper15}
D.~P. Bourne and S.~M. Roper.
\newblock Centroidal power diagrams, {Lloyd's} algorithm, and applications to
  optimal location problems.
\newblock {\em SIAM J. Numer. Anal.}, 53(6):2545--2569, 2015.

\bibitem{ButtazzoSantambrogio2009}
G.~Buttazzo and F.~Santambrogio.
\newblock A mass transportation model for the optimal planning of an urban
  region.
\newblock {\em SIAM Rev.}, 51(3):593--610, 2009.

\bibitem{BFRSB24}
M.~Buze, J.~Feydy, S.~Roper, K.~Sedighiani, and D.~Bourne.
\newblock Anisotropic power diagrams for polycrystal modelling: {E}fficient
  generation of curved grains via optimal transport.
\newblock {\em Comput. Mater. Sci.}, 245:113317, 2024.

\bibitem{CagliotiGolseIacobelli2015}
E.~Caglioti, F.~Golse, and M.~Iacobelli.
\newblock A gradient flow approach to quantization of measures.
\newblock {\em Math.~Models Methods Appl.~Sci.}, 25(10):1845--1885, 2015.

\bibitem{LinHK2021}
T.~Cai, J.~Cheng, B.~Schmitzer, and M.~Thorpe.
\newblock The linearized {Hellinger--Kantorovich} distance.
\newblock {\em SIAM J. Imaging Sci.}, 15(1):45--83, 2022.

\bibitem{ChizatOTFR2015}
L.~Chizat, G.~Peyr\'e, B.~Schmitzer, and F.-X. Vialard.
\newblock An interpolating distance between optimal transport and {Fisher--Rao}
  metrics.
\newblock {\em Found.~Comp.~Math.}, 18(1):1--44, 2018.

\bibitem{ChizatDynamicStatic2018}
L.~Chizat, G.~Peyr\'e, B.~Schmitzer, and F.-X. Vialard.
\newblock Unbalanced optimal transport: Dynamic and {Kantorovich} formulations.
\newblock {\em J. Funct. Anal.}, 274(11):3090--3123, 2018.

\bibitem{ChoLu}
R.~Choksi and X.~Y. Lu.
\newblock Bounds on the geometric complexity of optimal {C}entroidal {V}oronoi
  {T}esselations in {3D}.
\newblock {\em Comm. Math. Phys.}, 377(3):2429--2450, 2020.

\bibitem{DuFaberGunzburger99}
Q.~Du, V.~Faber, and M.~Gunzburger.
\newblock Centroidal {Voronoi} tessellations: Applications and algorithms.
\newblock {\em SIAM Rev.}, 41(4):637--676, 1999.

\bibitem{DuGunzburgerJuWang2006}
Q.~Du, M.~Gunzburger, L.~Ju, and X.~Wang.
\newblock Centroidal {V}oronoi tessellation algorithms for image compression,
  segmentation, and multichannel restoration.
\newblock {\em J.~Math.~Imaging Vis}, 24(2):177--194, 2006.

\bibitem{DuWang2005}
Q.~Du and D.~S. Wang.
\newblock The optimal centroidal {V}oronoi tessellations and the {G}ersho's
  conjecture in the three-dimensional space.
\newblock {\em Comput.~Math.~Appl.}, 49(9-10):1355--1373, 2005.

\bibitem{EmelianenkoJuRand08}
M.~Emelianenko, L.~Ju, and A.~Rand.
\newblock Nondegeneracy and weak global convergence of the {L}loyd algorithm in
  $\mathbb{R}^d$.
\newblock {\em SIAM J.~Numer.~Anal.}, 46(3):1423--1441, 2008.

\bibitem{Figalli2010}
A.~Figalli.
\newblock The optimal partial transport problem.
\newblock {\em Arch.~Ration.~Mech.~Anal.}, 195(2):533--560, 2010.

\bibitem{Folland}
G.~B. Folland.
\newblock {\em Real Analysis: Modern Techniques and Their Applications}.
\newblock Wiley, 2nd edition, 1999.

\bibitem{Galichon2016}
A.~Galichon.
\newblock {\em Optimal Transport Methods in Economics}.
\newblock Princeton University Press, 2016.

\bibitem{GallouetMerigot18}
T.~O. Gallou{\"e}t and Q.~M{\'e}rigot.
\newblock A {L}agrangian scheme {\`a} la {B}renier for the incompressible
  {E}uler equations.
\newblock {\em Found. Comput. Math.}, 18(4):835--865, 2018.

\bibitem{Gallouet:2022}
T.~O. Gallou{\"{e}}t, Q.~M{\'{e}}rigot, and A.~Natale.
\newblock Convergence of a {L}agrangian discretization for barotropic fluids
  and porous media flow.
\newblock {\em SIAM J. Math. Anal.}, 54(3):2990--3018, 2022.

\bibitem{Gersho79}
A.~Gersho.
\newblock Asymptotically optimal block quantization.
\newblock {\em IEEE Trans.~on Inform.~Theory}, 25(4):373--380, 1979.

\bibitem{GershoGray1992}
A.~Gersho and R.~Gray.
\newblock {\em Vector Quantization and Signal Compression}.
\newblock Springer, 1992.

\bibitem{GrafLuschgy2000}
S.~Graf and H.~Luschgy.
\newblock {\em Foundations of Quantization for Probability Distributions}.
\newblock Springer, 2000.

\bibitem{Gruber99}
P.~M. Gruber.
\newblock A short analytic proof of {F}ejes {T}{\'o}th's theorem on sums of
  moments.
\newblock {\em Aequationes Math.}, 58(3):291--295, 1999.

\bibitem{Gr04}
P.~M. Gruber.
\newblock Optimum quantization and its applications.
\newblock {\em Adv. Math.}, 186(2):456--497, 2004.

\bibitem{Gruber07}
P.~M. Gruber.
\newblock {\em Convex and Discrete Geometry}.
\newblock Springer, 2007.

\bibitem{Iacobelli2018}
M.~Iacobelli.
\newblock A gradient flow perspective on the quantization problem.
\newblock In P.~Cardaliaguet, A.~Porretta, and F.~Salvarani, editors, {\em PDE
  Models for Multi-Agent Phenomena}, volume~28, pages 145--165. Springer, 2018.

\bibitem{KitagawaMerigotThibert}
J.~Kitagawa, Q.~M\'{e}rigot, and B.~Thibert.
\newblock Convergence of a {N}ewton algorithm for semi-discrete optimal
  transport.
\newblock {\em J. Eur. Math. Soc.}, 21(9):2603--2651, 2019.

\bibitem{Kloeckner2012}
B.~Kloeckner.
\newblock Approximation by finitely supported measures.
\newblock {\em ESAIM Control Optim.~Calc.~Var.}, 18(2):343--359, 2012.

\bibitem{KMV-OTFisherRao-2015}
S.~Kondratyev, L.~Monsaingeon, and D.~Vorotnikov.
\newblock A new optimal transport distance on the space of finite {R}adon
  measures.
\newblock {\em Adv.~Differential Equ.}, 21(11-12):1117--1164, 2016.

\bibitem{LarssonChoksiNave2016}
L.~Larsson, R.~Choksi, and J.~C. Nave.
\newblock Geometric self-assembly of rigid shapes: {A} simple {V}oronoi
  approach.
\newblock {\em SIAM J.~Appl.~Math.}, 76(3):1101--1125, 2016.

\bibitem{LaMi2017}
V.~Laschos and A.~Mielke.
\newblock Geometric properties of cones with applications on the
  {Hellinger--Kantorovich} space, and a new distance on the space of
  probability measures.
\newblock {\em J. Funct. Anal.}, 276(11):3529--3576, 2019.

\bibitem{Lavenant2021}
H.~Lavenant, S.~Zhang, Y.-H. Kim, and G.~Schiebinger.
\newblock {Toward a mathematical theory of trajectory inference}.
\newblock {\em Ann. Appl. Probab.}, 34(1A):428--500, 2024.

\bibitem{LeclercFlows2020}
H.~Leclerc, Q.~Mérigot, F.~Santambrogio, and F.~Stra.
\newblock Lagrangian discretization of crowd motion and linear diffusion.
\newblock {\em SIAM J. Numer. Anal.}, 58(4):2093--2118, 2020.

\bibitem{LevySemiDiscrete2015}
B.~L\'evy.
\newblock A numerical algorithm for {L2} semi-discrete optimal transport in
  {3D}.
\newblock {\em ESAIM Math.~Model.~Numer.~Anal.}, 49(6):1693--1715, 2015.

\bibitem{LieroMielkeSavare-HellingerKantorovich-2015a}
M.~Liero, A.~Mielke, and G.~Savar\'e.
\newblock Optimal entropy-transport problems and a new {Hellinger--Kantorovich}
  distance between positive measures.
\newblock {\em Invent.~Math.}, 211(3):969--1117, 2018.

\bibitem{LMS-GeodesicConvexity}
M.~Liero, A.~Mielke, and G.~Savar\'e.
\newblock Fine properties of geodesics and geodesic lambda-convexity for the
  {Hellinger}--{Kantorovich} distance.
\newblock {\em Arch. Ration. Mech. Anal.}, 247:112, 2023.

\bibitem{Lloyd82}
S.~P. Lloyd.
\newblock Least squares quantization in {PCM}.
\newblock {\em IEEE Trans.~on Inform.~Theory}, 28(2):129--137, 1982.

\bibitem{LuSlepcev2013}
X.~Y. Lu and D.~Slep\v{c}ev.
\newblock Properties of minimizers of average-distance problem via discrete
  approximation of measures.
\newblock {\em SIAM J.~Math.~Anal.}, 45(5):3114--3131, 2013.

\bibitem{MacQueen67}
J.~B. MacQueen.
\newblock Some methods for classification and analysis of multivariate
  observations.
\newblock In {\em Proceedings of 5th Berkeley Symposium on Mathematical
  Statistics and Probability, 1}, pages 281--297. University of California
  Press, 1967.

\bibitem{MultiscaleTransport2011}
Q.~M{\'e}rigot.
\newblock A multiscale approach to optimal transport.
\newblock {\em Comput.~Graph.~Forum}, 30(5):1583--1592, 2011.

\bibitem{MerigotSantambrogioSarrazin}
Q.~M\'erigot, F.~Santambrogio, and C.~Sarrazin.
\newblock Non-asymptotic convergence bounds for {W}asserstein approximation
  using point clouds.
\newblock In M.~Ranzato, A.~Beygelzimer, Y.~Dauphin, P.~Liang, and J.~W.
  Vaughan, editors, {\em Proceedings of Advances in Neural Information
  Processing Systems 34 (NeurIPS 2021)}. 2021.

\bibitem{MerigotThibertOT}
Q.~M\'{e}rigot and B.~Thibert.
\newblock Optimal transport: discretization and algorithms.
\newblock In A.~Bonito and R.~H. Nochetto, editors, {\em Geometric Partial
  Differential Equations - Part II}, volume~22 of {\em Handbook of Numerical
  Analysis}, pages 133--212. 2021.

\bibitem{MeyronMerigotThibert2018}
J.~Meyron, Q.~M\'{e}rigot, and B.~Thibert.
\newblock Light in power: a general and parameter-free algorithm for caustic
  design.
\newblock {\em ACM Trans. Graph.}, 37(6):1--13, 2018.

\bibitem{MorganBolton02}
F.~Morgan and R.~Bolton.
\newblock Hexagonal economic regions solve the location problem.
\newblock {\em Am.~Math.~Monthly}, 109(2):165--172, 2002.

\bibitem{MosconiTilli2005}
S.~Mosconi and P.~Tilli.
\newblock {$\Gamma$}--convergence for the irrigation problem.
\newblock {\em J.~of Conv.~Anal.}, 12(1):145--158, 2005.

\bibitem{Newman1982}
D.~Newman.
\newblock The hexagon theorem.
\newblock {\em IEEE Trans.~on Inform.~Theory}, 28(2):137--139, 1982.

\bibitem{PagesPhamPrintems2004}
G.~Pag\`es, H.~Pham, and J.~Printems.
\newblock Optimal quantization methods and applications to numerical problems
  in finance.
\newblock In S.~Rachev, editor, {\em Handbook of Computational and Numerical
  Methods in Finance}. Birkh\"auser, 2004.

\bibitem{PeyreCuturi2018}
G.~Peyr\'e and M.~Cuturi.
\newblock Computational optimal transport.
\newblock {\em Found. Trends Mach. Learn.}, 11(5--6):355--607, 2019.

\bibitem{Rockafellar1972Convex}
R.~T. Rockafellar.
\newblock {\em Convex Analysis}.
\newblock Princeton University Press, 2nd edition, 1972.

\bibitem{SabinGray86}
M.~Sabin and R.~Gray.
\newblock Global convergence and empirical consistency of the generalized
  {L}loyd algorithm.
\newblock {\em IEEE Trans.~on Inform.~Theory}, 32(2):148--155, 1986.

\bibitem{Santambrogio-OTAM}
F.~Santambrogio.
\newblock {\em Optimal Transport for Applied Mathematicians}.
\newblock Birkh\"auser, 2015.

\bibitem{Santambrogio-CV}
F.~Santambrogio.
\newblock {\em A Course in the Calculus of Variations}.
\newblock Springer, 2023.

\bibitem{SarrazinMFG2020}
C.~Sarrazin.
\newblock Lagrangian discretization of variational mean field games.
\newblock {\em SIAM J. Control Optim.}, 60(3):1365--1392, 2022.

\bibitem{SchmitzerWirth-DynamicW1-2017}
B.~Schmitzer and B.~Wirth.
\newblock Dynamic models of {Wasserstein}-1-type unbalanced transport.
\newblock {\em ESAIM Control Optim. Calc. Var.}, 25:23, 2019.

\bibitem{ThorpeTheilJohansenCade15}
M.~Thorpe, F.~Theil, A.~M. Johansen, and N.~Cade.
\newblock Convergence of the $k$-means minimization problem using
  {$\Gamma$}-convergence.
\newblock {\em SIAM J.~Appl.~Math.}, 75(6):2444--2474, 2015.

\bibitem{UnbalancedGWBrain2022}
A.~Thual, Q.~H. Tran, T.~Zemskova, N.~Courty, R.~Flamary, S.~Dehaene, and
  B.~Thirion.
\newblock Aligning individual brains with fused unbalanced
  {Gromov}--{Wasserstein}.
\newblock In S.~Koyejo, S.~Mohamed, A.~Agarwal, D.~Belgrave, K.~Cho, and A.~Oh,
  editors, {\em Advances in Neural Information Processing Systems}, volume~35,
  pages 21792--21804, 2022.

\bibitem{GFejesToth01}
G.~F. T{\'o}th.
\newblock A stability criterion to the moment theorem.
\newblock {\em Studia Sci.~Math.~Hungar.}, 38(1):209--224, 2001.

\bibitem{FejesToth72}
L.~F. T\'oth.
\newblock {\em {L}agerungen in der {E}bene auf der {K}ugel und im {R}aum}.
\newblock Springer, 1972.

\bibitem{Villani-OptimalTransport-09}
C.~Villani.
\newblock {\em Optimal Transport: Old and New}.
\newblock Springer, 2009.

\bibitem{LevyCentroidalPower}
S.-Q. Xin, B.~L{\'e}vy, Z.~Chen, L.~Chu, Y.~Yu, C.~Tu, and W.~Wang.
\newblock Centroidal power diagrams with capacity constraints: {C}omputation,
  applications, and extension.
\newblock {\em ACM Trans. Graph.}, 35(6):244:1--244:12, 2016.

\bibitem{Za82}
P.~L. Zador.
\newblock Asymptotic quantization error of continuous signals and the
  quantization dimension.
\newblock {\em IEEE Trans. Inform. Theory}, 28(2):139--149, 1982.

\bibitem{Nocedal-LBFGS-TOMS-97}
C.~Zhu, R.~H. Byrd, P.~Lu, and J.~Nocedal.
\newblock Algorithm 778: {L-BFGS-B}: {Fortran} subroutines for large-scale
  bound-constrained optimization.
\newblock {\em ACM Trans.~Math.~Softw.}, 23(4):550--560, 1997.

\end{thebibliography}
\bibliographystyle{plain}

\end{document}